\documentclass{amsart}
\usepackage{amssymb}
\usepackage{hyperref}
\usepackage{pgfplots}
\usepackage{xcolor}

\DeclareMathOperator{\dist}{dist}
\DeclareMathOperator{\broad}{broad}
\newcommand{\R}{\mathbb{R}}

\newcommand{\K}{R}
\newtheorem{theorem}{Theorem}[section]
\newtheorem{corollary}[theorem]{Corollary}
\newtheorem{proposition}[theorem]{Proposition}
\newtheorem{definition}[theorem]{Definition}
\newtheorem{conjecture}[theorem]{Conjecture}
\pgfplotsset{compat=1.18}
\numberwithin{equation}{section}

\DeclareFontFamily{U}{mathx}{\hyphenchar\font45}
\DeclareFontShape{U}{mathx}{m}{n}{
      <5> <6> <7> <8> <9> <10>
      <10.95> <12> <14.4> <17.28> <20.74> <24.88>
      mathx10
      }{}
\DeclareSymbolFont{mathx}{U}{mathx}{m}{n}
\DeclareFontSubstitution{U}{mathx}{m}{n}
\DeclareMathAccent{\widecheck}{0}{mathx}{"71}
\DeclareMathAccent{\wideparen}{0}{mathx}{"75}

\title{Incidence bounds related to circular Furstenberg sets}

\author{John Green}
\address{Department of Mathematics \\
University of Pennsylvania \\
David Rittenhouse Lab \\
209 South 33rd Street \\
Philadelphia \\
PA 19104-6395 \\
USA}
\email{jdgreen@sas.upenn.edu}

\author[Terry Harris]{Terence L.~J.~Harris}
\address{Department of Mathematics\\ University of Wisconsin\\ 480
Lincoln Drive\\ Madison\\ WI\\ 53706\\ USA}
\email{terry.harris@wisc.edu}

\author{Yumeng Ou}
\address{Department of Mathematics\\ University of Pennsylvania\\
David Rittenhouse Lab\\
209 South 33rd Street\\
Philadelphia\\ 
PA 19104-6395\\ 
USA}
\email{yumengou@sas.upenn.edu}

\author{Kevin Ren}
\address{Department of Mathematics, Princeton University}
\email{kevinren@princeton.edu}

\author{Sarah Tammen}
\address{Department of Mathematics\\ University of Wisconsin\\ 480
Lincoln Drive\\ Madison\\ WI\\ 53706\\ USA}
\email{tammen2@wisc.edu}

\begin{document} 
\begin{abstract} We prove bounds on approximate incidences between families of circles and families of points in the plane. As a consequence, we prove a lower bound for the dimension of circular $(u,v)$-Furstenberg sets, which is new for large $u$ and $v$. 
\end{abstract}
\maketitle

\small
\setcounter{tocdepth}{2}
\tableofcontents
\normalsize

\section{Introduction}
Given $0 < u \leq 1$ and $0 \leq v \leq 3$, a set $F \subseteq \mathbb{R}^2$ is called a \emph{circular $(u,v)$-Furstenberg set} if there is $E \subseteq \mathbb{R}^2 \times (0, \infty)$ of Hausdorff dimension $\dim E \geq v$, such that $\dim\left( C(x,t) \cap F\right) \geq u$ for all $(x,t) \in E$. Here $C(x,t)$ is the circle in the plane centred at $x \in \mathbb{R}^2$ of radius $t>0$. We define \emph{sine wave Furstenberg sets} similarly by replacing $C(x,t)$ with $\Gamma_x \subseteq \mathbb{R}^2$, where $x = (x_1,x_2,x_3) \in \mathbb{R}^3$, and $\Gamma_x$ is the graph of the function 
\[ \theta \mapsto \left\langle (x_1,x_2,x_3), \frac{1}{\sqrt{2}} \left( \cos \theta, \sin \theta, 1 \right) \right \rangle = \frac{ x_1 \cos \theta + x_2 \sin \theta + x_3}{\sqrt{2}}, \] 
over the interval $[0, 2\pi]$. 
Our main result is the following theorem. \begin{theorem} \label{curvedfurstenberg} Let $0 < u \leq 1$ and $0 \leq v \leq 3$. If $F \subseteq \mathbb{R}^2$ is a  circular $(u,v)$-Furstenberg set,  or a sine wave $(u,v)$-Furstenberg set, then 
\begin{equation} \label{bound1} \dim F \geq \begin{cases} 4u+v-3 & 3u+v \leq 4 \\
u+1 & 3u+v > 4, \end{cases} \end{equation}
and
\begin{equation} \label{sine2} \dim F \geq \begin{cases} 2u + \frac{2v}{3} -1 & 3u + 2v \leq 6 \\
 u+1 & 3u+2v > 6, \end{cases} \end{equation}
 and (by \cite{liu} in the circular case)
  \begin{equation} \label{sine1} \dim F \geq u+ \frac{v}{3}. \end{equation}
\end{theorem}

Some of our methods also extend to cinematic $(u,v)$-Furstenberg sets (see Definition~\ref{cinematicdefn}).
\begin{theorem} \label{cinematicfurstenberg}
Let $0 < u \leq 1$ and $0 \leq v \leq 3$. If $F \subseteq \mathbb{R}^2$ is a  cinematic $(u,v)$-Furstenberg set, then 
    \begin{equation} \label{bound1'} \dim F \geq \begin{cases} 4u+v-3 & 3u+v \leq 4 \\
    u+1 & 3u+v > 4, \end{cases} \end{equation}
    and
    \begin{equation} \label{bound2} \dim F \geq \begin{cases} 2u+v - 2 & u+v \leq 3 \\
    u+1 & u+v > 3. \end{cases} \end{equation}
\end{theorem}

Previously, the best known bounds for circular Furstenberg sets due to ~\cite{liu},~\cite{fasslerliuorponen} are:
\begin{equation}\label{circular best known}
    \dim F \ge u + \max\left\{ \frac{v}{3}, \min \{ u, v \} \right\}.
\end{equation}
A comparison of Theorem~\ref{curvedfurstenberg} with previous results is shown in Figure~\ref{curvedfurstenbergpicture}.
The only known bound for cinematic Furstenberg sets (including sine Furstenberg sets) is due to Zahl~\cite{zahl}, who proved that $\dim F \ge u + \min \{ u, v \}$ for $0 < u \le 1$, $0 \le v \le 3$, which is sharp. (In fact, Zahl's proof holds if the family of cinematic curves is replaced by any family of smooth curves in the plane forbidding $k$-th order tangencies for some fixed (but arbitrarily large) $k$.)


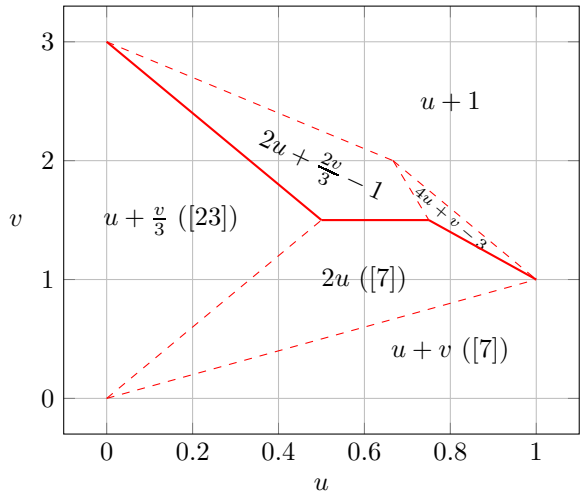
\begin{figure} \label{curvedfurstenbergpicture}
\begin{tikzpicture}[
declare function={
	  func(\x)= (0<=\x<=1) * (1.5);
    func2(\x)= (0.3333333<=\x<=1) * (4-3*\x);
    func3(\x)= (0<=\x<=1) * (\x);
    func4(\x)= (0<=\x<=1) * (3*\x);
    func5(\x)= (0<=\x<=1) * (3-2*\x);
    func6(\x)= (0<=\x<=1) * (3-1.5*\x);
    func7(\x)= (0<=\x<=1) * (3-3*\x);
    func8(\x)= (0<=\x<=1) * (6-6*\x);
  }]
  \begin{axis}[ 
	grid=major,
    xlabel=$u$,
    ylabel={$v$},
    ylabel style={rotate=-90},
		legend style={
legend pos=outer north east,
}
  ] 
\coordinate (A) at (axis cs:0.8,2.5);
\node at (A) {$u+1$};
\coordinate (A) at (axis cs:0.80,1.51);
\node[rotate=-40] at (A) {\tiny $4u+v-3$};
\coordinate (C) at (axis cs:0.15,1.5);
\node at (C) {$u+\frac{v}{3}$ (\cite{liu})};
\coordinate (D) at (axis cs:0.8,0.4);
\node at (D) {$u+v$ (\cite{fasslerliuorponen})};
\coordinate (E) at (axis cs:0.6,1);
\node at (E) {$2u$ (\cite{fasslerliuorponen})};
\coordinate (F) at (axis cs:0.5,1.95);
\node[rotate=-25] at (F) {$2u+\frac{2v}{3}-1$};
\addplot[red,thick,domain=0.5:0.75]{func(x)}; 
\addplot[red,dashed,domain=0.6666666:1]{func2(x)}; 
\addplot[red,dashed,domain=0:1]{func3(x)}; 
\addplot[red,dashed,domain=0:0.5]{func4(x)}; 
\addplot[red,thick,domain=0.75:1]{func5(x)}; 
\addplot[red,dashed,domain=0:0.6666666]{func6(x)}; 
\addplot[red,thick,domain=0:0.5]{func7(x)}; 
\addplot[red,dashed,domain=0.6666666:0.75]{func8(x)}; 
  \end{axis}
\end{tikzpicture} \caption{The current best known lower bounds for circular Furstenberg sets, including our results in Theorem~\ref{curvedfurstenberg}. The known bounds for sine curves are the same, but the lower bound $\min\{u+v,2u\}$ is from \cite{zahl} in this case, and $u+\frac{v}{3}$ is Theorem~\ref{curvedfurstenberg}. } \end{figure}

In~\cite{wolff}, Wolff proved\footnote{Wolff proved the special case where the planar set contains a circle centred at $x$ for all $x$ in a $v$-dimensional subset of $\mathbb{R}^2$, but as explained in \cite{fasslerliuorponen}, the methods work for general families of circles.} the special case $\dim F \geq 1+v$ of \eqref{circular best known} when $u=1$ and $0 \leq v \leq 1$. Setting $u=1$ in Theorem~\ref{curvedfurstenberg} gives the lower bound of $1+v$ for all $0\leq v \leq 1$, so Theorem~\ref{curvedfurstenberg} is a different generalisation of Wolff's result.  

The circular Furstenberg problem is a variant of the linear Furstenberg problem in the plane, which was solved in~\cite{orponenshmerkin},~\cite{renwang}. In the more general setting of ``families of curves forbidding $k$-th order tangency'', considered in \cite{zahl}, the linear Furstenberg problem is an example where $k=1$, and the circular and sine wave Furstenberg problems are examples where $k=2$. 

\subsection{Description of the approach} Our approach to Theorem~\ref{curvedfurstenberg} is via bounds on incidences between  families of circles and families of discs in the plane satisfying certain dimension conditions, and can be viewed as a curved version of the Fourier analytic approach in \cite{furen}. Most of our incidence bounds use the ``high-low method'', which originated in \cite{GSW}, which sorts the incidences into different ``frequency buckets''. In the standard high-low argument, the incidences are represented by an integral $\int fg$ over Euclidean space, and either $f$ or $g$ (equivalently both) is broken up into parts where their Fourier transforms are supported near the origin (low) or away from the origin (high). In the circular case, we represent incidences by integrals of the form $\int f (g \ast \sigma_t)(x) \, dx \, dt$, where $\sigma_t$ is the probability measure on the circle around 0 of radius $t$, and decompose $g$ into high and low frequencies. This is related to \cite[Corollary~3]{wolff}, where integrals of the form $\int_A (\chi_U \ast \sigma_t)$ are analysed via a Littlewood-Paley decomposition of $\chi_U$.

The low frequency incidences are always handled by induction. In the high case, the proof of \eqref{bound1} uses the sharp local smoothing inequality of Guth, Wang, and Zhang from \cite{guthwangzhang}. The proof of \eqref{bound2} is similar, but uses a simpler $L^2$ bound in the high case in place of the local smoothing inequality.  

The proof of \eqref{sine2} also uses the high-low method, but in the high case it uses the trilinear Fourier restriction theorem in $\mathbb{R}^3$ specialised to the cone (this special case of trilinear restriction is discussed in detail in \cite{leevargas}). For the trilinear approach, rather than bounding incidences between families of circles and balls, we only bound the ``broad'' incidences, which roughly speaking rule out incidences concentrated in small angular sectors of the circles. Although incidence bounds of this kind are weaker than standard incidence bounds, they are equally applicable to circular Furstenberg sets, since the Furstenberg property ensures that the points on each circle do not all concentrate in a finite number of small angular sectors. This seems to be the first application of multilinear Fourier restriction to Furstenberg type problems, though the argument has some similarities to \cite{duzhang}. Our definition of ``broad incidences'' is based on the broad norms defined in \cite{guth2}. The proof of \eqref{sine1} for sine waves also uses ``broad'' (or trilinear) incidences, but instead of the high-low method it uses a simpler geometric approach which avoids Fourier analysis.

It may be possible to generalise the trilinear Fourier restriction approach to cinematic $(u,v)$-Furstenberg sets, using e.g.~variable coefficient trilinear Fourier restriction \cite[Theorem~1.16]{tao}, but we have not attempted to carry this out, so we leave this as an open question whether such a generalisation is possible. 

\subsection{Connection to restricted families of projections} It is well known that by point-line duality (see e.g.~\cite{renwang,orponenshmerkin}), the linear Furstenberg problem is related to exceptional sets for orthogonal projections in the plane. Similarly, the duality from \cite{kaenmakiorponenvenieri} between Kakeya problems of planar sine waves and restricted projections to lines in $\mathbb{R}^3$, implies that the sine wave Furstenberg problem is related to exceptional sets of restricted projections to lines in $\mathbb{R}^3$. We will briefly explain this duality here. For each $\theta \in [0, 2\pi)$, let $\rho_{\theta}$ be orthogonal projection onto the span of $\left( \cos \theta, \sin \theta, 1 \right)$. We identify a point $z \in \mathbb{R}^3$ with the graph of a sine wave $\theta \mapsto \Gamma_z(\theta) = \frac{z_1 \cos \theta + z_2 \sin \theta + z_3}{\sqrt{2}}$ over the interval $[0, 2\pi]$. If $A \subseteq \mathbb{R}^3$ has $\dim A =v$, if $s \leq \min\{1,v\}$ and $\dim \Theta = u$, where
\[ \Theta = \left\{ \theta \in [0, 2\pi]: \dim(\rho_{\theta}(A)) < s \right\}, \]
then 
\begin{equation} \label{furstenbergsine} F := \left\{ (\theta, \Gamma_z(\theta) ) :  \theta \in \Theta, z \in A \right\} \end{equation}
is a sine wave $(u,v)$-Furstenberg set. Each $x$-section at $x=\theta$ of $F$ is $\rho_{\theta}(A)$, which has dimension $< s$ whenever $\theta \in \Theta$, so a heuristic application of Fubini's theorem leads us to expect that $\dim F \le u + s$. Using the lower bounds for $\dim F$ provided by Theorem \ref{curvedfurstenberg} and \eqref{circular best known}, we deduce upper bounds on the dimension $u$ of $\Theta$. For example, a corollary of \eqref{sine1} (in the sine wave case) is the following.
\begin{proposition} \label{projection} If $A \subseteq \mathbb{R}^3$ is a Borel set, then 
\begin{equation} \label{zerodimensions} \dim\left\{ \theta \in [0, 2\pi) : \dim \rho_{\theta}(A) < \frac{\dim A}{3} \right\} = 0. \end{equation} \end{proposition}
Proposition~\ref{projection} is known, as it follows from the deep result of He~\cite[Corollary~5]{he}. In fact, He's result shows that ``$< \dim A/3$'' can be replaced by the weaker condition ``$\leq \dim A/3$'', but we will explain why Proposition~\ref{projection} is essentially a special case of \eqref{sine1}.  Let $v = \dim A$ and suppose for a contradiction that the exceptional set in \eqref{zerodimensions} has positive dimension. Then, for some $\epsilon>0$, the set \[ \left\{ \theta \in [0, 2\pi) : \dim \rho_{\theta}(A) < \frac{v}{3}-\epsilon \right\} \]
has positive dimension $u>0$. Then, the sine wave $(u,v)$-Furstenberg set $F$ in \eqref{furstenbergsine}) will have dimension $\dim F \le u + \frac{v}{3}-\epsilon$. But this contradicts \eqref{sine1}, which says that  $\dim F \ge u+\frac{v}{3}$. 

Strictly speaking, the above argument is not rigorous, since Fubini's theorem does not hold for Hausdorff dimension. However, our results also imply a discretised version of Theorem~\ref{curvedfurstenberg}, which (by following the argument above) implies a discretised version of Proposition~\ref{projection}, and the discretised version implies the continuous version. Since this kind of argument is already included in \cite{orponenshmerkin} in the linear case, and since Proposition~\ref{projection} already follows from \cite[Corollary~5]{he}, we will not include further details. From now on, we will only use the connection to projections as a heuristic; it will not be used in proofs. 

In Section~\ref{projectionsection} we show that Proposition~\ref{projection} is sharp in the sense that there are examples where the set of exceptions has positive dimension if $\dim A/3$ is replaced by any larger number. Proposition~\ref{projection} seems to be the analogue of D.~Oberlin's \cite[Theorem~1.2]{oberlin}, which states that if $A \subseteq \mathbb{R}^2$ is a Borel set with $\dim A \leq 1$, then the set of directions onto which the projection of $A$ has dimension strictly smaller than $\dim A/2$ has dimension zero. Interpolating between this and Kaufman's exceptional set estimate \cite{kaufman}, Oberlin then conjectured that if $A$ is a set in the plane of dimension $t$, and $s \leq \min\{ t, 1\}$, then the set of directions onto which the projection of $A$ has dimension strictly smaller than $s$, has Hausdorff dimension at most $\max\{2s-t,0\}$ (this was proven in \cite{renwang}). Following this approach, we formulate an analogue of Oberlin's conjecture in the restricted projection setting of $\R^3$.
\begin{conjecture} \label{oberlinR3} Let $t \in [0,3]$ and let $s \in \left[0, \min\left\{ t, 1 \right\} \right]$. If $A \subseteq \mathbb{R}^3$ is a Borel set with $\dim A =t$, then 
\[ \dim\left\{ \theta \in [0, 2\pi) : \dim \rho_{\theta}(A) < s \right\} \leq \max\left\{\frac{3s}{2} - \frac{t}{2},0 \right\}. \] \end{conjecture}
If $t \leq 1$ and $s=t$ then Conjecture~\ref{oberlinR3} reduces to a theorem of Pramanik, Yang, and Zahl (\cite[Theorem~1.3]{pramanikyangzahl}), which is an $\mathbb{R}^3$ restricted projection analogue of Kaufman's exceptional set estimate in the plane, and which is known to be sharp. If $s = t/3$ then Conjecture~\ref{oberlinR3} reduces to Proposition~\ref{projection}, so Conjecture~\ref{oberlinR3} interpolates between the cases $s=t/3$ and $s=t$. The case $t\geq 1$ and $s=1$ of Conjecture~\ref{oberlinR3} was solved by Gan, Guth, and Maldague \cite{ganguthmaldague}. We can also compare the case $s =t/3 + \epsilon$ with He~\cite[Corollary~5]{he}, which can be viewed as a ``restricted projections'' generalisation of Bourgain's discretised projection theorem. Following the heuristic outlined at the start of this subsection, a Furstenberg ``generalisation'' of Conjecture~\ref{oberlinR3} is the following. 
\begin{conjecture} \label{furstenbergconjecture} Let $0 < u \leq 1$ and $0 \leq v \leq 3$. If $F \subseteq \mathbb{R}^2$ is a sine wave $(u,v)$-Furstenberg set, or a circular $(u,v)$-Furstenberg set, then 
\begin{equation} \label{casesversion} \dim F \geq \begin{cases}  u+v  & 0 \leq v \leq u \\
\frac{5u+v}{3} & 0 \leq u < v \text{ and } 2u +v \leq 3  \\
u+1 & 2u + v  > 3, \end{cases} \end{equation}
or equivalently 
\begin{equation} \label{minversion} \dim F \geq \min\left\{ u+v,\frac{5u+v}{3}, u+1 \right\}.  \end{equation} \end{conjecture}
The first line in \eqref{casesversion} is known \cite{fasslerorponen}; it is included for completeness. 
The heuristic for the second line in \eqref{casesversion} comes from setting $t=v$ and observing that $\frac{5u+v}{3} \leq u+s$ is equivalent to $u \leq \frac{3s}{2} - \frac{t}{2}$. The restrictions $2u+v \leq 3$ and $u \le v$ follow from the conditions $s \leq \min\{ t, 1 \}$ and $u \leq \frac{3s}{2} - \frac{t}{2}$. Examples of product sets show that the lower bound of $u+1$ is best possible whenever it holds, so $u+1$ is the natural conjecture for $2u+v > 3$ (as this would follow from the case $2u+v=3$). The outer two parts of the conjectured minimum in \eqref{minversion} match the lower bounds for the linear Furstenberg problem in the plane \cite{renwang}, but the middle term $\frac{5u+v}{3}$ is always strictly smaller than the corresponding term $\frac{3u+v}{2}$ for the linear case whenever it is the minimum (i.e.~whenever $u < v$). Liu~\cite[p.~302]{liu} conjectured that for $v=1$, any circular $(u,1)$ Furstenberg set has dimension at least $\frac{1+3u}{2}$ (the same lower bound as for linear $(u,1)$-Furstenberg sets in the plane), so Conjecture~\ref{furstenbergconjecture} is strictly weaker than Liu's conjecture in this case. 

\subsection{Sets containing many cinematic curves} If we replace the assumption $u \in [0,1]$ in Theorem~\ref{curvedfurstenberg} with the stronger assumption that each curve intersects the planar set in a set of positive 1-dimensional outer Lebesgue measure, and assume that $v>1$, then we can obtain the following stronger conclusion that the set has positive 2-dimensional outer Lebesgue measure. 
\begin{theorem} \label{positivearea} If $\left\{ \Sigma_{x,t} \right\}_{(x,t)}$ is a family of cinematic functions (see Definition~\ref{cinematicdefn}), if $E \subseteq \mathbb{R}^3$ is a set with $\dim E > 1$, and if $F \subseteq \mathbb{R}^2$ is such that $\mathcal{H}^1\left( \Sigma_{x,t} \cap F \right) >0$ for every $(x,t) \in E$, then $\mathcal{H}^2(F) >0$. \end{theorem}
This is a generalisation of  Corollary~3 in \cite{wolff}, which proved Theorem~\ref{positivearea} for circles. The proof is similar, except that we use variable coefficient local smoothing in place of local smoothing (or decoupling). Above, $\mathcal{H}^s$ is the $s$-dimensional Hausdorff outer measure, which for $s=2$ equals the outer Lebesgue measure in the plane, and for $s=1$ equals the 1-dimensional arc length outer measure when restricted to any curve. If Theorem~\ref{positivearea} is specialised to sine waves, then by using the previously discussed duality with restricted projections in $\mathbb{R}^3$, we obtain a different proof of the known result that if $A \subseteq \mathbb{R}^3$ is Borel and $\dim A >1$, then $\rho_{\theta}(A)$ has positive 1-dimensional Lebesgue measure for a.e.~$\theta \in [0, 2\pi)$. See~\cite{kaenmakiorponenvenieri} for a discussion of this problem.

\subsection{Notation} \label{notation} If $\mathbb{T}$ is a family of $\delta$-balls in Euclidean space, and $\alpha \geq 0$, we let $K_{\alpha, \mathbb{T}, \delta}$ be the smallest constant such that for any $r \geq \delta$ and any ball $B_r$ of radius $r$,
\[ \left\lvert\left\{ D \in \mathbb{T} : D \cap B_r \neq \emptyset \right\} \right\rvert \leq K_{\alpha,\mathbb{T},\delta} \left( \frac{r}{\delta} \right)^{\alpha}. \]
If $K_{\alpha, \mathbb{T}, \delta} < \infty$, we call $\mathbb{T}$ a $(\delta, \alpha, K_{\alpha, \mathbb{T}, \delta})$-set. Note the ``a priori'' bound $K_{\alpha, \mathbb{T}, \delta} \le |\mathbb{T}|$. We define this similarly if $\mathbb{T}$ is a family of points; interchanging points with $\delta$-balls does not make a substantial difference. If $P$ is a family of $\delta$-neighbourhoods of circles in the plane, and $\beta \geq 0$, we define $K_{\beta,P,\delta} = K_{\beta, \Phi(P),\delta}$, where $\Phi$ sends the $\delta$-neighbourhood of a circle centred at $x \in \mathbb{R}^2$ of radius $r>0$ to the $\delta$-ball centred at $(x,r) \in \mathbb{R}^3$. Similarly,  if $P$ is a family of $\delta$-neighbourhoods of ``sine waves'', i.e.~curves of the form 
\[ \left\{  \left(\theta, \frac{a \cos \theta + b \sin \theta +c}{\sqrt{2}}\right):  \theta \in [0, 2\pi) \right\}, \]
where $(a,b,c) \in \mathbb{R}^3$, and if $\beta \geq 0$, we define $K_{\beta, P,\delta} = K_{\beta, \Psi(P),\delta}$, where $\Psi$ sends the $\delta$-neighbourhood of a sine wave to the $\delta$-ball centred at the corresponding point $(a,b,c) \in \mathbb{R}^3$. We define $K_{\beta,P,\delta}$ similarly if circles or sine waves are replaced more generally by families of cinematic curves (see Definition~\ref{cinematicdefn}). If $\mathbb{T}$ is a family of $\delta$-neighbourhoods $T$ of ``light planes'', i.e.~planes in $\mathbb{R}^3$ of the form $t_T\gamma(\theta_T) + \gamma(\theta_T)^{\perp}$, where $\theta_T \in [0, 2\pi]$ and $\gamma(\theta) = \frac{1}{\sqrt{2}}\left( \cos \theta, \sin \theta, 1 \right)$, we define $K_{\alpha, \mathbb{T}, \delta} = K_{\alpha, \widetilde{\mathbb{T} },\delta}$, where $\widetilde{\mathbb{T}}$ is the family of $\delta$-balls in $\mathbb{R}^2$, where the centres of the balls are the points $\left(\theta_T, t_T \right) \in [0, 2\pi] \times \mathbb{R}$.

For any finite sets $A$ and $B$, we define the number of incidences between them by $I(A,B) = \left\lvert \left\{ (a,b) \in A \times B : a \cap b \neq \emptyset \right\} \right\rvert$.

\subsection{Outline}  In Section~\ref{incidence}, we prove incidence bounds that will imply Theorem~\ref{curvedfurstenberg} and Theorem~\ref{cinematicfurstenberg}. In Section~\ref{conversion}, we show how these incidence bounds imply Theorem~\ref{curvedfurstenberg} and Theorem~\ref{cinematicfurstenberg}. In Section~\ref{positiveareasection}, we prove Theorem~\ref{positivearea}. In Section~\ref{projectionsection}, we give a counterexample related to Conjecture~\ref{oberlinR3} and Conjecture~\ref{furstenbergconjecture}. In the appendix, for the case of sine waves we give an alternative proof of the incidence bound related to \eqref{bound1} which uses small cap decoupling instead of variable coefficient local smoothing. We also give some incidence bounds that improve upon those in Section~\ref{incidence} in some cases and are special to sine waves, though they do not yield any new information about sine wave Furstenberg sets. 
\subsection{Acknowledgments}
Y.O. is supported in part by NSF DMS-2142221 and NSF DMS-2055008. K.R. is supported by an NSF GRFP fellowship.  S.T. is supported in part by NSF DMS-2037851.

\section{Incidence bounds} \label{incidence}

\subsection{Circle incidences via local smoothing} 
In this subsection, we prove the circle incidence bound corresponding to \eqref{bound1} in Theorem~\ref{curvedfurstenberg}. The idea we use, of approaching circular Kakeya-type problems via local smoothing inequalities for the wave equation, is inspired by the approach used by Wolff in \cite{wolff}. In particular, the ``high-low'' decomposition of a function that we use is similar to the decomposition in \cite[Lemma~6.2]{wolff}.

The local smoothing inequality of Guth, Wang, and Zhang \cite[Theorem~1.2]{guthwangzhang} states that
\begin{equation} \label{guthwangzhang} \left( \int_1^2 \int_{\mathbb{R}^2} \left\lvert u(x,t)   \right\rvert^4 \, dx \, dt \right)^{1/4} \leq C_{\epsilon} \left(\|u_0\|_{4,\epsilon} + \| u_1 \|_{4, \epsilon-1} \right), \end{equation}
for any $\epsilon>0$, where $u$ solves the wave equation $\Delta u - \partial_{tt} u = 0$ with Schwartz initial data $u(x,0) = u_0(x)$ and $(\partial_t u)(x,0) = u_1(x)$. The solution of the wave equation is characterised by
\[ \widehat{u} (\xi,t) = \widehat{u_0}(\xi) \cos\left( 2 \pi |\xi| t \right) + \frac{ \widehat{u_1}(\xi) }{2\pi |\xi| } \sin\left( 2 \pi |\xi| t\right), \]
where $\widehat{u}$ indicates the Fourier transform of $u$ in the first two variables. Recall that for Schwartz $g$ with $\widehat{g}(\xi) = (1+4\pi^2|\xi|^2)^{-\alpha/2} \widehat{h}(\xi)$, and $0 < \alpha < 2$, the $\|g\|_{p, \alpha}$ norm of $g$ is defined to be $\|h\|_p$ (see e.g.~\cite[p.~135, Eq.~(39)]{stein}). It is straightforward to check that a corollary of \eqref{guthwangzhang} is that if $g$ is a Schwartz function with $\widehat{g}$ supported in $B(0,R) \setminus B(0,R/2)$, then for any $\epsilon >0$, 
\begin{equation} \label{guthwangzhang2} \left( \int_1^2 \int_{\mathbb{R}^2} \left\lvert\int_{\mathbb{R}^2} e^{2 \pi i \langle \xi, x \rangle} \cos\left( 2 \pi |\xi| t \right) \widehat{g}(\xi)\, d\xi   \right\rvert^4 \, dx \, dt \right)^{1/4} 
\leq C_{\epsilon} R^{\epsilon} \left\lVert g \right\rVert_4. \end{equation}
and
\begin{equation} \label{guthwangzhang3} \left(\int_1^2  \int_{\mathbb{R}^2} \left\lvert\int_{\mathbb{R}^2} e^{2 \pi i \langle \xi, x \rangle} \sin\left( 2 \pi |\xi| t \right) \widehat{g}(\xi)\, d\xi   \right\rvert^4 \, dx \, dt \right)^{1/4} \\
\leq C_{\epsilon} R^{\epsilon} \left\lVert g \right\rVert_4. \end{equation}
Inequalities~\eqref{guthwangzhang2} and \eqref{guthwangzhang3} are the versions of the local smoothing inequality we will use below.

\begin{theorem} \label{theoremincidencecircle} Let $\alpha \in [0, 2]$, $\beta \in [0,3]$, and $\delta \in (0,1)$. Let $P$ be a $\left( \delta, \beta, K_{\beta,P,\delta}\right)$-set of $\delta$-neighbourhoods of circles in the plane of radii between 1 and 2. Let $\mathbb{T}$ be a $\left( \delta, \alpha, K_{\alpha, \mathbb{T}, \delta}\right)$-set of $\delta$-discs in the plane. If $3\alpha + \beta \leq 7$, then for any $\epsilon >0$, 
\begin{equation} \label{incidencebound9} I(P, \mathbb{T}) \leq C_{\epsilon} \delta^{-\epsilon}\delta^{-3/4} K_{\beta,P,\delta}^{1/4}   K_{\alpha,\mathbb{T},\delta}^{3/4} |P|^{3/4} |\mathbb{T}|^{1/4} . \end{equation}
If $3\alpha+\beta > 7$, then with $\lambda = \frac{4}{3\alpha+\beta-3} < 1$, for any $\epsilon >0$,
\begin{equation} \label{incidencebound10} I(P, \mathbb{T}) \leq C_{\epsilon} \delta^{-\epsilon} \delta^{-3\lambda/4}   K_{\beta,P,\delta}^{\lambda/4} K_{\alpha,\mathbb{T},\delta}^{3\lambda/4} |P|^{1-\frac{\lambda}{4}} |\mathbb{T}|^{1-\frac{3\lambda}{4}} . \end{equation}
\end{theorem} 
\begin{proof}  By Hölder's inequality and the assumption on the radii of the circles in $P$, we can assume that the balls in $\mathbb{T}$ all lie in $B_2(0,1)$ (if $3\alpha+\beta >7$, we also need to use that $\lambda \leq 1$ and that $\|x\|_p \leq \|x\|_q$ when $p \geq q$), so this will be assumed throughout the proof.  By scaling, we can assume that the radii of the circles in $P$ actually lie between 1.1 and 1.9. We will prove the theorem by induction on $\delta$. The case case $\delta \sim 1$ follows from the trivial bound $I(P, \mathbb{T}) \leq |P| |\mathbb{T}|$.

We first show that \eqref{incidencebound9} implies \eqref{incidencebound10} if $3\alpha+\beta > 7$, since this will be used a few times throughout the proof. Assume that \eqref{incidencebound9} holds. By the trivial bound 
\begin{equation} \label{trivial} I(P, \mathbb{T}) \leq |P| |\mathbb{T}| = |P|^{1-\frac{\lambda}{4}} |\mathbb{T}|^{1-\frac{3\lambda}{4} } \left( |P|^{1/4} |\mathbb{T}|^{3/4} \right)^{\lambda}, \end{equation}
we may assume that $|P|^{1/4} |\mathbb{T}|^{3/4} \geq \delta^{-3/4} K_{\beta,P,\delta}^{1/4} K_{\alpha, \mathbb{T}, \delta}^{3/4}$, since otherwise substituting this (reversed) bound into \eqref{trivial} yields \eqref{incidencebound10}. But if 
\[|P|^{1/4} |\mathbb{T}|^{3/4} \geq \delta^{-3/4} K_{\beta,P,\delta}^{1/4} K_{\alpha, \mathbb{T}, \delta}^{3/4}, \] then \eqref{incidencebound9} gives 
\begin{multline*} I(P, \mathbb{T}) \leq C_{\epsilon} \delta^{-\epsilon} \delta^{-3/4} K_{\beta,P,\delta}^{1/4} K_{\alpha,\mathbb{T},\delta}^{3/4}  |P|^{3/4} |\mathbb{T}|^{1/4} \\
= C_{\epsilon} \delta^{-\epsilon}\delta^{-3/4}  K_{\beta,P,\delta}^{1/4} K_{\alpha,\mathbb{T},\delta}^{3/4}   |P|^{1-\frac{\lambda}{4}} |\mathbb{T}|^{1-\frac{3\lambda}{4}} \left(|P|^{1/4} |\mathbb{T}|^{3/4} \right)^{-(1-\lambda) } \\
\leq  C_{\epsilon} \delta^{-\epsilon}   \delta^{-3\lambda/4} K_{\beta,P,\delta}^{\lambda/4} K_{\alpha,\mathbb{T},\delta}^{3\lambda/4}   |P|^{1-\frac{\lambda}{4}} |\mathbb{T}|^{1-\frac{3\lambda}{4}} \delta^{-3\lambda/4}.
\end{multline*}
This verifies that \eqref{incidencebound9} implies \eqref{incidencebound10} if $3\alpha+\beta > 7$.

Let $\sigma$ be the uniform probability measure on the unit circle, and for each $t>0$ let $\sigma_t$ be the pushforward of $\sigma$ under $x \mapsto tx$ (so that $\sigma_t$ is the uniform probability measure on the circle around 0 of radius $t$). By definition, 
\[ I(P, \mathbb{T}) = \sum_{B \in P} \sum_{T \in \mathbb{T} : T \cap B \neq \emptyset} 1. \]
If $T \cap B \neq \emptyset$ for some $B \in P$, then for any $(x,t) \in B((x_B,t_B), \delta)$,
\begin{equation} \label{referencedbelow} \delta \lesssim \left(\sigma_t \ast \chi_{10T}\right)(x), \end{equation}
where $\left(x_B,t_B\right) \in \mathbb{R}^2 \times [1,2]$ is such that $B$ is the $\delta$-neighbourhood of the circle centred at $x_B$ of radius $t_B$. For technical reasons, we need $\chi_{10T}$ to be a smooth bump function equal to $1$ on $10T$ and vanishing outside $20T$, so we will abuse notation and assume this. If $T \cap B \neq \emptyset$, then by \eqref{referencedbelow}, 
\[ \delta^4 \lesssim \int_{B\left(\left(x_B, t_B\right),\delta\right)}  (\sigma_t \ast \chi_{10T})(x) \, dx \, dt \qquad \text{if }  T \cap B \neq \emptyset. \]
Therefore, if we define 
\[ f = \sum_{T \in \mathbb{T}} \chi_{10T}, \]
then
\begin{equation} \label{incidenceintegral} I(P, \mathbb{T}) \lesssim \delta^{-4} \int  \sum_{B \in P} \chi_{B\left(\left(x_B, t_B\right),\delta\right)} \left(\sigma_t \ast f\right)(x) \, dx \, dt. \end{equation}
Let $\phi$ be a smooth bump function equal to 1 on $B_2\left(0, \delta^{-1-\epsilon^2}\right) \setminus B_2\left(0, \delta^{\epsilon^2-1}\right)$,  vanishing outside $B\left(0, 2\delta^{-1-\epsilon^2} \right)$ and vanishing in $B_2\left(0, \delta^{\epsilon^2-1}/2\right)$. Let $\psi$ be a smooth bump function supported equal to $1-\phi$ in $B_2\left(0, \delta^{\epsilon^2-1} \right)$ and extended by zero outside this ball. We define the ``high'' and ``low'' parts $f_h$, $f_l$ of $f$ by
\[ \widehat{f_h} = \widehat{f} \phi, \qquad \widehat{f_l} = \psi \widehat{f}. \]
Then, apart from a very small error term, $f = f_h + f_l$. Hence, by \eqref{incidenceintegral},
\begin{multline} \label{highlow}  I(P, \mathbb{T}) \lesssim \delta^{-4} \int \sum_{B \in P} \chi_{B\left(\left(x_B, t_B\right),\delta\right)}  \left\lvert \left(\sigma_t \ast f_h\right)(x)\right\rvert \, dx \, dt \\
+\delta^{-4} \int \sum_{B \in P} \chi_{B\left(\left(x_B, t_B\right),\delta\right)}  \left\lvert \left(\sigma_t \ast f_l\right)(x)\right\rvert \, dx \, dt, \end{multline}
where we can assume that the very small error terms do not dominate, as in this case the conclusion of the theorem holds trivially. By Fourier inversion, 
\[ \left(\sigma_t \ast f_h\right)(x) = \int_{\mathbb{R}^2} e^{2\pi i \langle x, \xi \rangle } \widehat{\sigma}(t \xi) \widehat{f_h}(\xi) \, d\xi, \]
for any $x \in \mathbb{R}^2$. The asymptotic formula (see \cite[Appendix~B]{grafakos}) for $\widehat{\sigma}$ is
\begin{multline} \label{asymptotic} \widehat{\sigma}(\xi) = \frac{1}{\pi |\xi|^{1/2}} \cos\left( 2\pi |\xi| - \frac{\pi}{4} \right) + O\left(|\xi|^{-3/2} \right) \\
= \frac{1}{\pi \sqrt{2}|\xi|^{1/2}} \cos\left( 2 \pi |\xi| \right) + \frac{1}{\pi \sqrt{2}|\xi|^{1/2}} \sin\left( 2 \pi |\xi| \right) + O\left(|\xi|^{-3/2}\right), \end{multline}
so we define $\widehat{\sigma^{(1)}}$ and $\widehat{\sigma^{(2)}}$ by 
\[  \widehat{\sigma^{(1)}}(\xi)  = \frac{1}{\pi \sqrt{2} |\xi|^{1/2}} \cos\left( 2\pi |\xi| \right)  + \frac{1}{\pi \sqrt{2}|\xi|^{1/2}} \sin\left( 2 \pi |\xi| \right), \]
and 
\[ \widehat{\sigma^{(2)}}(\xi) = \widehat{\sigma}(\xi) - \widehat{\sigma^{(1)}}(\xi). \]
We do not define $\sigma^{(1)}$ and $\sigma^{(2)}$, so for example $\widehat{\sigma^{(1)}}$ is only a formal symbol and not necessarily a Fourier transform. By \eqref{asymptotic}, 
\begin{equation} \label{errorbound} \left\lvert\widehat{\sigma^{(2)}}(\xi)  \right\rvert \lesssim |\xi|^{-3/2},\qquad \xi \in \mathbb{R}^2. \end{equation} Similarly, we define $\widehat{\sigma^{(1)}_t}$ and $\widehat{\sigma^{(2)}_t}$ by $\widehat{\sigma^{(1)}_t}(\xi) = \widehat{\sigma_1}(t\xi)$ and $\widehat{\sigma^{(2)}_t}(\xi) = \widehat{\sigma_2}(t \xi)$ (this implies that $\widehat{\sigma_t} = \widehat{\sigma^{(1)}_t}+\widehat{\sigma^{(2)}_t}$; because $\widehat{\sigma_t}(\xi) = \widehat{\sigma}(t \xi )$). By \eqref{highlow}, \begin{multline} \label{threeterms} I(P, \mathbb{T}) \lesssim \delta^{-4} \int \sum_{B \in P} \chi_{B\left(\left(x_B, t_B\right),\delta\right)}  \left\lvert \int_{\mathbb{R}^2} e^{2\pi i \langle x, \xi \rangle } \widehat{\sigma^{(1)}_t}(\xi) \widehat{f_h}(\xi) \, d\xi \right\rvert \, dx \, dt \\+\delta^{-4} \int \sum_{B \in P} \chi_{B\left(\left(x_B, t_B\right),\delta\right)}  \left\lvert \int_{\mathbb{R}^2} e^{2\pi i \langle x, \xi \rangle } \widehat{\sigma^{(2)}_t}(\xi) \widehat{f_h}(\xi) \, d\xi \right\rvert \, dx \, dt \\
+\delta^{-4} \int \sum_{B \in P} \chi_{B\left(\left(x_B, t_B\right),\delta\right)}  \left\lvert \left(\sigma_t \ast f_l\right)(x)\right\rvert\, dx \, dt. \end{multline}
Suppose first that the middle term dominates in \eqref{threeterms}. This is the error term, so we expect it to satisfy better bounds than those stated in the theorem. By Cauchy-Schwarz, Plancherel, \eqref{errorbound}, then Plancherel again, and finally the definition of $(\delta, \alpha, K)$-set and a priori bound in Section \ref{notation},
\begin{align}\notag I(P, \mathbb{T}) \notag &\lesssim \delta^{-4} \left(K_{\beta,P,\delta} |P| \delta^3 \right)^{1/2} \left(\int_1^2 \int_{\mathbb{R}^2} \left\lvert \widehat{f_h}(\xi) \widehat{\sigma^{(2)}_t}(\xi) \right\rvert^2 \, d\xi \, dt \right)^{1/2} \\
\notag &\lesssim \delta^{-4} \delta^{3/2 -O(\epsilon^2)} \left(K_{\beta,P,\delta} |P| \delta^3 \right)^{1/2} \left(\int_1^2 \int_{\mathbb{R}^2} \left\lvert \widehat{f}(\xi) \right\rvert^2 \, d\xi \, dt \right)^{1/2} \\
\notag &\lesssim \delta^{-4} \delta^{3/2 -O(\epsilon^2)} \left(K_{\beta,P,\delta} |P| \delta^3 \right)^{1/2} \left( K_{\alpha, \mathbb{T}, \delta} |\mathbb{T}| \delta^2 \right)^{1/2} \\
\notag &= \delta^{-O(\epsilon^2)} K_{\alpha, \mathbb{T}, \delta}^{1/2} K_{\beta,P,\delta}^{1/2} |P|^{1/2} |\mathbb{T}|^{1/2} \\\notag &\lesssim  \delta^{- \frac{\alpha}{4}-O(\epsilon^2)} K_{\alpha, \mathbb{T}, \delta}^{3/4} K_{\beta,P,\delta}^{1/2} |P|^{1/2} |\mathbb{T}|^{1/4} \\
&\le \delta^{- \frac{\alpha}{4}-O(\epsilon^2)} K_{\alpha, \mathbb{T}, \delta}^{3/4} K_{\beta,P,\delta}^{1/4} |P|^{3/4} |\mathbb{T}|^{1/4}, \end{align}
which is better than the required inequality \eqref{incidencebound9} since $\alpha \leq 2 \leq 3$. This proves \eqref{incidencebound9} if the middle term dominates in \eqref{threeterms}. Since  \eqref{incidencebound9} implies \eqref{incidencebound10} when $3\alpha+\beta > 7$, this shows that the theorem is true for any $\alpha$ and $\beta$ if the middle term dominates in \eqref{threeterms}, so we may assume from now on that 
\begin{multline} \label{twoterms} I(P, \mathbb{T}) \lesssim \delta^{-4} \int \sum_{B \in P} \chi_{B\left(\left(x_B, t_B\right),\delta\right)} \left\lvert \int_{\mathbb{R}^2} e^{2\pi i \langle x, \xi \rangle } \widehat{\sigma^{(1)}_t}(\xi) \widehat{f_h}(\xi) \, d\xi \right\rvert\, dx \, dt \\
+\delta^{-4} \int \sum_{B \in P} \chi_{B\left(\left(x_B, t_B\right),\delta\right)} \left\lvert \left(\sigma_t \ast f_l\right)(x)\right\rvert \, dx \, dt. \end{multline}
The ``high case'' is when the first term dominates, and the ``low case'' is when the second term dominates. If we are in the high case, then by Hölder's inequality,
\begin{align*} I(P, \mathbb{T}) &\lesssim \delta^{-4} \left( K_{\beta,P,\delta}^{1/3} |P| \delta^3 \right)^{3/4} \left(\int_1^2 \int_{\mathbb{R}^2} \left\lvert \int_{\mathbb{R}^2} e^{2\pi i \langle x, \xi \rangle } \widehat{\sigma^{(1)}_t}(\xi) \widehat{f_h}(\xi) \, d\xi \right\rvert^4 \, dx \, dt\right)^{1/4}. \end{align*}
Hence, by the definition of $\widehat{\sigma^{(1)}_t}$ (see \eqref{asymptotic}), either
\begin{multline} \label{incidencewaves} I(P, \mathbb{T}) \lesssim \delta^{-4} \left(  K_{\beta,P,\delta}^{1/3}|P| \delta^3 \right)^{3/4} \times \\\
\left( \int_1^2\int_{\mathbb{R}^2}\left\lvert \int_{\mathbb{R}^2} e^{2 \pi i \langle \xi, x \rangle } \cos\left( 2 \pi |\xi| t \right) \frac{\widehat{f_h}(\xi)}{|\xi|^{1/2}} \, d\xi   \right\rvert^4 \, dx \, dt \right)^{1/4}, \end{multline}
or 
\begin{multline} \label{incidencewaves2}  I(P, \mathbb{T}) \lesssim \delta^{-4} \left(  K_{\beta,P,\delta}^{1/3}|P| \delta^3 \right)^{3/4} \times \\
\left( \int_1^2\int_{\mathbb{R}^2}\left\lvert \int_{\mathbb{R}^2} e^{2 \pi i \langle \xi, x \rangle } \sin\left( 2 \pi |\xi| t \right) \frac{\widehat{f_h}(\xi)}{|\xi|^{1/2}} \, d\xi   \right\rvert^4 \, dx \, dt \right)^{1/4}. \end{multline}
If we apply \eqref{guthwangzhang2} and \eqref{guthwangzhang3} to \eqref{incidencewaves} and \eqref{incidencewaves2}, (recall $\widehat{f_h} = \widehat{f} \phi$), we get 
\[ I(P, \mathbb{T}) \lesssim \delta^{-O(\epsilon^2)}\delta^{-4} \left( K_{\beta,P,\delta}^{1/3}|P| \delta^3 \right)^{3/4} \left\lVert f \ast \mathcal{F}^{-1}\left(\frac{\phi}{|\cdot|^{1/2} } \right) \right\rVert_4. \]
Hence (by Young's convolution inequality),
\begin{equation} \label{annoying} I(P, \mathbb{T}) \lesssim \delta^{-O(\epsilon^2)} \delta^{-4} \left(K_{\beta,P,\delta}^{1/3} |P| \delta^3 \right)^{3/4} \| f \|_4 \left\lVert \mathcal{F}^{-1}\left(\frac{\phi}{|\cdot|^{1/2} } \right) \right\rVert_1. \end{equation}
It is straightforward to show that
\[ \left\lVert \mathcal{F}^{-1}\left(\frac{\phi}{|\xi|^{1/2} } \right) \right\rVert_1 \lesssim\delta^{1/2-O(\epsilon^2)}, \]
so \eqref{annoying} becomes
\begin{multline*} I(P, \mathbb{T}) \lesssim \delta^{-O(\epsilon^2)} \delta^{-4} \delta^{1/2} \left( K_{\beta,P,\delta}^{1/3}|P| \delta^3 \right)^{3/4} \left( K_{\alpha, \mathbb{T}, \delta}^3 |\mathbb{T}| \delta^2 \right)^{1/4} \\
= \delta^{-O(\epsilon^2)}\delta^{-3/4} K_{\beta,P,\delta}^{1/4} K_{\alpha, \mathbb{T}, \delta}^{3/4}|P|^{3/4} |\mathbb{T}|^{1/4}. \end{multline*}
This implies \eqref{incidencebound9}. Since \eqref{incidencebound9} implies \eqref{incidencebound10} when $3\alpha + \beta >7$, this proves the bound in the high case.

Now suppose we are in the low case; so that the second term dominates in \eqref{twoterms}:
\[ I(P, \mathbb{T}) \lesssim  \delta^{-4} \int \sum_{B \in P} \chi_{B\left(\left(x_B, t_B\right),\delta\right)} \left\lvert \left(\sigma_t \ast f_l\right)(x)\right\rvert \, dx \, dt. \]
In this case, by definition, $f_l = \sum_{T \in \mathbb{T}} \chi_{10T} \ast \widecheck{\psi}$. By Hausdorff-Young, 
\[ \left\lVert \chi_{10T} \ast \widecheck{\psi}\right\rVert_{\infty} \lesssim \|\psi\|_1 \delta^2 \lesssim \delta^{2\epsilon^2}. \] Moreover, $\widecheck{\psi}$ is negligible outside $B_2\left(0, \delta^{1-\epsilon^2 -\epsilon^4}\right)$. Hence, if we let $S = \delta^{-\epsilon^2-\epsilon^4}$, then apart from a negligible error term,
\[ \left\lvert \chi_{10T} \ast \widecheck{\psi} \right\rvert \lesssim \delta^{-O(\epsilon^4)} S^{-2} \chi_{ST}.\]
Hence, for any $(x,t) \in \mathbb{R}^2 \times [1,2]$,
\begin{equation} \label{usedagainlater} \left\lvert \left(\sigma_t \ast f_l\right)(x)\right\rvert \lesssim S^{-2} \delta^{-O(\epsilon^4)} \sigma_t \ast \left( \sum_{T \in \mathbb{T}} \chi_{ST}\right)(x), \end{equation}
apart from a negligible error term, where $ST$ is the disc with the same centre as $T$ but with radii scaled by $S$. Thus
\begin{multline} \label{incidencethickening} I(P, \mathbb{T}) \lesssim  S^{-2} \delta^{-4-O(\epsilon^4)} \int \sum_{B \in P} \chi_{B\left(\left(x_B, t_B\right),\delta\right)} \sigma_t \ast \left( \sum_{T \in \mathbb{T}} \chi_{ST}\right)(x) \, dx \, dt \\
\lesssim S^{-1} \delta^{-O(\epsilon^4)} I(P_S, \mathbb{T}_S), \end{multline}
where $P_S$ is the set of $\delta S$ neighbourhoods of circles from $P$, $\mathbb{T}_S$ is the set of discs from $\mathbb{T}$ with the same centres, but with radii scaled by $S$. As before, we may assume the negligible error terms do not contribute to \eqref{incidencethickening}, since the conclusion of the theorem follows in this case. By induction on $\delta$, we can apply the theorem at scale $\delta S$ to \eqref{incidencethickening}, to get
\begin{multline} \label{secondlast} I(P, \mathbb{T}) \lesssim \delta^{-O(\epsilon^4) }S^{-1} \times\\
\Bigg\{\begin{aligned} &(\delta S)^{-3/4-\epsilon}  K_{\beta,P_S,\delta S}^{1/4} K_{\alpha, \mathbb{T}_S, \delta S}^{3/4}  |P_S|^{3/4} |\mathbb{T}_S|^{1/4} && 3\alpha + \beta \leq 7 \\
&(\delta S)^{-3\lambda/4-\epsilon}   K_{\beta,P_S,\delta S}^{\lambda/4} K_{\alpha, \mathbb{T}_S, \delta S}^{3\lambda/4}|P_S|^{1-\frac{\lambda}{4}} |\mathbb{T}_S|^{1-\frac{3\lambda}{4}} && 3\alpha + \beta >7. \end{aligned} \end{multline}
By the inequalities
\[ K_{\alpha, \mathbb{T}_S, \delta S} \leq S^{\alpha} K_{\alpha, \mathbb{T}, \delta}, \qquad K_{\beta, P_S, \delta S} \leq S^{\beta} K_{\beta, P, \delta},\]
and since $|\mathbb{T}_S| = |\mathbb{T}|$ and $|P_S| = |P|$, 
\eqref{secondlast} implies that
\begin{multline*}  I(P, \mathbb{T}) \lesssim \delta^{\epsilon^3-O(\epsilon^4) } \times\\
\Bigg\{\begin{aligned} & S^{-\frac{7}{4} + \frac{3\alpha}{4} + \frac{\beta}{4}} \delta^{-\frac{3}{4}-\epsilon}   K_{\beta,P,\delta}^{1/4}K_{\alpha, \mathbb{T}, \delta}^{3/4} |P|^{3/4} |\mathbb{T}|^{1/4} && 3\alpha + \beta \leq 7 \\
& S^{-1-\frac{3\lambda}{4} +\frac{3\lambda \alpha}{4} + \frac{\lambda \beta}{4} }\delta^{-
\frac{3\lambda}{4} - \epsilon}   K_{\beta,P,\delta}^{\lambda/4} K_{\alpha, \mathbb{T}, \delta}^{3\lambda/4}|P|^{1-\frac{\lambda}{4}} |\mathbb{T}|^{1-\frac{3\lambda}{4}} && 3\alpha + \beta >7. \end{aligned} \end{multline*}
In either case, the exponent of $S$ is non-positive, so by induction the conclusion of the theorem holds in the low case, and this finishes the proof. \end{proof}

\subsection{Broad incidences between circles and discs}


In this subsection, we prove an incidence estimate for a set of discs and a set of annuli.  Specifically, we will prove an upper bound for the number of \emph{broad incidences} (to be defined momentarily) between an ($\alpha$-dimensional) set of $\delta$-discs and a $(\beta$-dimensional) set of $\delta$-neighborhoods of circles with radius between $1$ and $2$.  This incidence estimate corresponds to \eqref{sine2} in Theorem ~\ref{curvedfurstenberg} and is proved using an $L^3$ trilinear Fourier restriction for the cone.

\begin{definition}
Given a set $P$ of $\delta$-neighbourhoods of circles in $\mathbb{R}^2$ of radii between 1 and 2, and a set $\mathbb{T}$ of $\delta$-discs in $\mathbb{R}^2$, let $A$ and $\K$ be large parameters, and let $\{\tau\}_{\tau}$ be a boundedly overlapping cover of $S^1$ by arcs of diameter $\K^{-1}$. Define
\[ I_{A,\K}^{\broad}(P, \mathbb{T} ) = \sum_{B \in P} \inf_{(V_1, V_2, \dotsc, V_A) \in G(1,2) } \sup_{\dist(\tau, V_a) \geq \K^{-1} \forall \, a } I(B_{\tau}, \mathbb{T}), \]
where, given a circle $B$ centred at $x_B$ of radius $t_B$, corresponding to the annulus $B$, $B_{\tau}$ is the arc of the annulus corresponding to the arc $\tau$, i.e.
\[ B_{\tau} = \left\{x_B + te : e \in \tau, \left\lvert t-t_B\right\rvert < \delta\right\}. \] 
\end{definition}
This setup is based on the broad norms defined in \cite{guth2}. The reason for using ``broad incidences'' here rather than ``trilinear incidences'' is that we will  need a quasi-triangle inequality to separate the high and low cases when using the high-low method. Broad norms were introduced in \cite{guth2} precisely because they satisfy a quasi-triangle inequality.

Bounds on $I_{A,\K}(P, \mathbb{T})$ (for $1 \leq A,\K \ll \delta^{-\epsilon}$) are just as applicable to circular Furstenberg sets as bounds on $I(P, \mathbb{T})$. 

\begin{theorem} \label{trilinearincidencecircle} Let $0 \leq \alpha \leq 2$ and $0 \leq \beta \leq 3$ with $3\alpha + 2\beta \leq 9$.  Let $0 < \delta   < 1$, let $P$ be a $(\delta, \beta, K_{\beta,P,\delta})$-set of $\delta$-neighbourhoods of circles in $\mathbb{R}^2$ of radii between 1 and 2, and let $\mathbb{T}$ be a $(\delta, \alpha, K_{\alpha, \mathbb{T}, \delta})$-set of $\delta$-discs in $\mathbb{R}^2$. Then for any $\epsilon >0$, $A \geq |\log \delta|$ with $A \in \mathbb{N}$ and $\K \geq 1$,
\begin{equation} \label{trilinearrestrictioncircle} I_{A,\K}^{\broad}(P, \mathbb{T}) \leq C_{\epsilon} \K^{100} \delta^{-\epsilon - 1/2} K_{\beta,P,\delta}^{1/3} K_{\alpha, \mathbb{T}, \delta}^{1/2} |P|^{2/3} \left\lvert \mathbb{T} \right\rvert^{1/2}. \end{equation}
\end{theorem}

\begin{proof} If $\K \geq \delta^{-1/2}$ then this is trivial, so we can assume that $\K < \delta^{-1/2}$. By Hölder's inequality, using that the radii of the circles are between 1 and 2, and since $\|x  \|_{q} \leq \|x\|_1$ when $q \geq 1$, it may be assumed that the $\delta$-discs in $\mathbb{T}$ all lie in $B_2(0,1)$, and that the centres of the circles in $P$ all lie in $B_2(0,10)$. We will prove the theorem by induction on $\delta$. The base case $\delta \sim 1$ is clear by the trival bound of $|P| |\mathbb{T} |$. By definition,
\begin{multline} \label{incidenceintegral3}  I_{A,\K}^{\broad}(P, \mathbb{T}) \lesssim \delta^{-4-30\epsilon^2} \sum_{B \in P} \\
\inf_{(V_1,\dotsc, V_A) \in G(1,2) } \sup_{d(\tau, V_a) \geq \K^{-1} \forall \, a } \int_{B\left((x_B,t_B), \delta^{1+10\epsilon^2}\right)} \left\lvert \sum_{T \in \mathbb{T}} \chi_{T} \ast \sigma_{\tau,t}(x) \right\rvert \, dx \, dt, \end{multline}
where $x_B$ is the centre of $B$, $t_B$ is the radius, and $\sigma_{\tau,t}$ is (essentially) the restriction of $\sigma_t$ to the arc corresponding to $\tau$, given by multiplying $\sigma$ multiplied by $\psi_{\tau}$, where $\psi_{\tau}$ is a smooth bump function $\sim 1$ on $\tau$ and vanishing outside $1.1\tau$, and then pushing forward $\psi_{\tau} \sigma$ by $x\mapsto tx$. The use of $\delta^{1+10\epsilon^2}$ rather than $\delta$ in \eqref{incidenceintegral3} is a very minor technicality which can be mostly ignored. We now do the same high-low decomposition as in the proof of Theorem~\ref{theoremincidencecircle}. Let $\phi$ be a smooth bump function equal to 1 on $B_2\left(0, \delta^{-1-\epsilon^2}\right) \setminus B_2\left(0, \delta^{\epsilon^2-1}\right)$,  vanishing outside $B_2\left(0, 2\delta^{-1-\epsilon^2} \right)$ and vanishing in $B_2\left(0, \delta^{\epsilon^2-1}/2\right)$. Let $\psi$ be a smooth bump function supported equal to $1-\phi$ in $B_2\left(0, \delta^{\epsilon^2-1} \right)$ and extended by zero outside this ball. Then by \eqref{incidenceintegral3}, 
\begin{multline} \label{incidenceintegral4} I_{A,\K}^{\broad}(P, \mathbb{T}) \lesssim \delta^{-4-30\epsilon^2} \sum_{B \in P} \inf_{(V_1,\dotsc, V_A) \in G(1,2) } \sup_{d(\tau, V_a) \geq \K^{-1} \forall \, a } \\
\Bigg(\int_{B\left((x_B,t_B), \delta^{1+10\epsilon^2}\right)} \left\lvert \sum_{T \in \mathbb{T}} \chi_{T} \ast \sigma_{\tau,t}  \ast \widecheck{\phi}(x) \right\rvert \, dx \, dt \\
+ \int_{B\left((x_B,t_B), \delta^{1+10\epsilon^2}\right)}  \left\lvert \sum_{T \in \mathbb{T}} \chi_{T} \ast \sigma_{\tau,t}  \ast \widecheck{\psi}(x) \right\rvert \, dx \, dt \Bigg),\end{multline}
apart from a negligible error term which we can ignore (here we need to replace $\chi_T$ by an equivalent smooth bump equal to 1 on $T$ and vanishing outside $2T$). We can assume that $|\log \delta | \geq 4$ since the theorem holds trivially for large $\delta$. The above inequality \eqref{incidenceintegral4} implies that
\begin{multline} \label{highlowdecomp2} I_{A,\K}^{\broad}(P, \mathbb{T}) \lesssim \delta^{-4-30\epsilon^2} \sum_{B \in P} \inf_{(V_1,V_2,V_3,V_4) \in G(1,2) } \sup_{d(\tau, V_a) \geq \K^{-1} \forall \, a \in \{1,2,3,4\} } \\ \int_{B\left((x_B,t_B), \delta^{1+10\epsilon^2}\right)} \left\lvert \sum_{T \in \mathbb{T}} \chi_{T} \ast \sigma_{\tau,t}  \ast \widecheck{\phi}(x) \right\rvert \, dx \, dt \\
+ \delta^{-4-30\epsilon^2} \sum_{B \in P} \inf_{(V_5, \dotsc, V_A) \in G(1,2) } \sup_{d(\tau, V_a) \geq \K^{-1} \forall \, a \in \{5, \dotsc, A\} } \\  \int_{B\left((x_B,t_B), \delta^{1+10\epsilon^2}\right)}  \left\lvert \sum_{T \in \mathbb{T}} \chi_{T} \ast \sigma_{\tau,t} \ast \widecheck{\psi}(x) \right\rvert \, dx \, dt.\end{multline}
The high case is when the first term dominates, and the low case is when the second term dominates.  In the low case, let $S = \delta^{-\epsilon^2-\epsilon^4}$. 
By Young's convolution inequality followed by the Hausdorff-Young inequality,
\[ \left\lVert \chi_T \ast \widecheck{\psi} \right\rVert_{\infty} \lesssim \delta^{2\epsilon^2}, \]
and hence (apart from a negligible error term which can be ignored)
\[ \left\lvert \chi_T \ast \widecheck{\psi} \right\rvert \lesssim \delta^{2\epsilon^2} \chi_{(S/2)T}. \]
Therefore, in the low case,
\begin{multline} \label{inductivecircle} I_{A,\K}^{\broad}(P, \mathbb{T}) \lesssim \delta^{-4-30\epsilon^2} \sum_{B \in P} \inf_{(V_5, \dotsc, V_A) \in G(1,2) } \sup_{d(\tau, V_a) \geq \K^{-1} \forall \, a \in \{5, \dotsc, A\} }\\
  \int_{B\left((x_B,t_B), \delta^{1+10\epsilon^2}\right)}  \left\lvert \sum_{T \in \mathbb{T}} \chi_{T} \ast \sigma_{\tau,t}  \ast \widecheck{\psi}(x) \right\rvert \, dx \, dt \\
  \lesssim \delta^{-4-30\epsilon^2} \delta^{2\epsilon^2} \sum_{B \in P} \inf_{(V_5, \dotsc, V_A) \in G(1,2) } \sup_{d(\tau, V_a) \geq \K^{-1} \forall \, a \in \{5, \dotsc, A\} }\\
  \int_{B\left((x_B,t_B), \delta^{1+10\epsilon^2}\right)}  \left\lvert \sum_{T \in \mathbb{T}} \chi_{(S/2)T} \ast \sigma_{\tau,t} (x) \right\rvert \, dx \, dt \\
\lesssim S^{-1} \delta^{-2\epsilon^4}I^{\broad}_{A-4,\K} \left(P_S, \mathbb{T}_S \right)
 \end{multline}
 where $P_S$ consists of the same annuli of circles as in $P$, except of width $\delta S$ instead of $\delta$, and $\mathbb{T}_S$ consists discs of the same centres as those in $\mathbb{T}$, but of radius $\delta S$ instead of $\delta$. 

The assumption $A \geq |\log \delta|$ implies that $A-4 \geq |\log \delta S |$ provided that $\delta$ is sufficiently small depending only on $\epsilon$, with the remaining cases being handled by the constant $C_{\epsilon}$ in the theorem statement. Therefore, we can apply induction on $\delta$ to \eqref{inductivecircle} to get 
\[ I_{A,\K}^{\broad}(P, \mathbb{T}) \lesssim \delta^{-2\epsilon^4}S^{\frac{\alpha}{2} + \frac{\beta}{3} - \frac{3}{2} - \epsilon} C_{\epsilon}  \K^{100} \delta^{-\epsilon - 1/2} K_{\beta,P,\delta}^{1/3} K_{\alpha, \mathbb{T}, \delta}^{1/2} |P|^{2/3} \left\lvert \mathbb{T} \right\rvert^{1/2}. \]
The exponent of $S$ is negative by the assumption $3\alpha+2\beta \leq 9$, and the negative power of $S$ eliminates the $\delta^{-2\epsilon^4}$ factor and implicit constants, so this finishes the proof of \eqref{trilinearrestrictioncircle} in the low case. 

Now suppose we are in the high case, so that the first part of \eqref{highlowdecomp2} dominates:
\begin{multline} \label{firstpart} I_{A,\K}^{\broad}(P, \mathbb{T}) \lesssim \delta^{-4-30\epsilon^2} \sum_{B \in P} \inf_{(V_1,V_2,V_3,V_4) \in G(1,2) } \sup_{d(\tau, V_a) \geq \K^{-1} \forall \, a \in \{1,2,3,4\} } \\\int_{B\left((x_B,t_B), \delta^{1+10\epsilon^2}\right)} \left\lvert \sum_{T \in \mathbb{T}} \chi_{T} \ast \sigma_{\tau,t} \ast \widecheck{\phi}(x) \right\rvert \, dx \, dt.\end{multline}
Let $g = \sum_{T \in \mathbb{T}} \chi_{T} \ast \widecheck{\phi}$. Then by Fourier inversion, the function in the integral above is
\begin{equation} \label{abbreviation} g \ast \sigma_{\tau,t}(x) = \int_{\mathbb{R}^2} e^{2 \pi i \langle x, \xi\rangle} \widehat{g}(\xi) \widehat{\sigma_{\tau}}(t\xi) \, d\xi. \end{equation}
Via stationary phase, we will show that the above satisfies
\begin{multline} \label{provedbelow} \int_{\mathbb{R}^2} e^{2 \pi i \langle x, \xi\rangle} \widehat{g}(\xi) \widehat{\sigma_{\tau}}(t\xi) \, d\xi =C_1 \int_{\mathbb{R}^2} e^{2\pi i \left\langle (x, t), (\xi, |\xi| ) \right\rangle } \widehat{g}(\xi) |t \xi|^{-1/2} \psi_{\tau} (-\xi/ |\xi|)  \, d\xi \\
+ C_2 \int_{\mathbb{R}^2} e^{2\pi i \left\langle (x, t), (\xi, -|\xi| ) \right\rangle } \widehat{g}(\xi) |t \xi|^{-1/2} \psi_{\tau} (\xi/ |\xi|)  \, d\xi \\+ \int_{\mathbb{R}^2} e^{2 \pi i \langle x, \xi\rangle} \widehat{g}( \xi) O\left(\K^{O(1)}|\xi|^{-3/2}\right) \, d\xi,\end{multline}
for some absolute constants $C_1$ and $C_2$. Here the last integral means 
\[ \int_{\mathbb{R}^2} e^{2 \pi i \langle x, \xi\rangle} \widehat{g}( \xi) k(\xi) \, d\xi \]for some function $k$ satisfying $k(\xi) = O\left(\K^{O(1)}|\xi|^{-3/2}\right)$. We first write
\[ \widehat{\sigma_{\tau}}(t\xi) = \int e^{-2\pi i \langle t \xi, y \rangle} \psi_{\tau}(y) \, d\sigma(y). \]
The actual powers of $\K$ will not be important, but eventually they will be much smaller than the $\K^{100}$ factor in the theorem statement. Given $\xi$ above, let $U=U_{\xi}$ be a rotation of the plane satisfying $U\xi = |\xi| e_2 = (0, |\xi|)$, so that by rotation invariance of $\sigma$, 
\begin{align*} \widehat{\sigma_{\tau}}(t\xi) &= \int e^{-2\pi i \langle t \xi, U^*y \rangle} \psi_{\tau}(U^*y) \, d\sigma(y)\\
&=\int e^{-2\pi i t\langle |\xi| e_2, y \rangle} \psi_{\tau}(U^*y) \, d\sigma(y), \end{align*}
where $U^* = U^T$ is the transpose of $U$. As in \cite[p.~42]{wolff2}, we decompose $\sigma$ using a partition of unity, and then apply non-stationary phase (e.g.~\cite[Proposition~6.1]{wolff2}) to get that only small neighbourhoods of $y= \pm e_2$ contribute significantly to the above\footnote{Instead of using \cite[Proposition~6.1]{wolff}, this can easily be proved directly by breaking the circle into four pieces and integrating by parts many times.}. This yields
\begin{align*}
\widehat{\sigma_{\tau}}(t\xi) &= \int_{-1/2}^{1/2} e^{2\pi i t|\xi| \sqrt{1-y^2}} h_1(y) \psi_{\tau}(U^*(y, -\sqrt{1-y^2})) \, dy \\
&\quad +\int_{-1/2}^{1/2} e^{-2\pi i t|\xi| \sqrt{1-y^2}} h_2(y) \psi_{\tau}(U^*(y, \sqrt{1-y^2})) \, dy + O\left(\K^{O(1)} |\xi|^{-3/2}\right), \end{align*}
where $h_1,h_2$ are smooth functions supported in $[-1/2,1/2]$ independent of $\xi$ and $\K$. By stationary phase (e.g.~\cite[Proposition~6.4]{wolff2} with $N=0$),
\begin{multline} \label{sigmadecomp} \widehat{\sigma_{\tau}}(t\xi)  = C_1 |t\xi|^{-1/2} e ^{2\pi i t |\xi|}\psi_{\tau}(-\xi/|\xi|) + C_2 |t\xi|^{-1/2} e^{-2\pi i t |\xi| }\psi_{\tau}(\xi/|\xi|) \\
+ O\left(\K^{O(1)}|\xi|^{-3/2}\right), \end{multline}
where $C_1$, $C_2$ are absolute constants. This verifies \eqref{provedbelow}; by substitution of \eqref{sigmadecomp} into the left-hand side of \eqref{provedbelow}. Applying \eqref{provedbelow} to \eqref{abbreviation} and then \eqref{firstpart} gives
\begin{multline} \label{surfacedecomp} I_{A,\K}^{\broad}(P, \mathbb{T}) \lesssim \delta^{-4-30\epsilon^2} \sum_{B \in P} \inf_{(V_1,V_2) \in G(1,2) } \sup_{d(\tau, V_a) \geq \K^{-1} \forall \, a \in \{1,2\} } \\\int_{B\left((x_B,t_B), \delta^{1+10\epsilon^2}\right)} \left\lvert \sum_{T \in \mathbb{T}} \chi_{T,\tau} \ast \sigma_t^{(1)} \ast \widecheck{\phi}(x) \right\rvert \, dx \, dt\\
+\delta^{-4-30\epsilon^2} \sum_{B \in P} \inf_{(V_3,V_4) \in G(1,2) } \sup_{d(\tau, V_a) \geq \K^{-1} \forall \, a \in \{3,4\} } \\\int_{B\left((x_B,t_B), \delta^{1+10\epsilon^2}\right)} \left\lvert \sum_{T \in \mathbb{T}} \chi_{T,\tau} \ast \sigma_t^{(2)}  \ast \widecheck{\phi}(x) \right\rvert \, dx \, dt \\
+\delta^{-4-30\epsilon^2} \sum_{B \in P} \sup_{\tau} \int_{B\left((x_B,t_B), \delta^{1+10\epsilon^2}\right)} \left\lvert \sum_{T \in \mathbb{T}} \chi_{T} \ast \sigma_{t,\tau}^{(3)} \ast \widecheck{\phi}(x) \right\rvert \, dx \, dt\end{multline}
where $\sigma_t^{(1)}$ and $\sigma_t^{(2)}$ are defined by 
\begin{equation} \label{sigmaoneseconddefn} \widehat{\sigma_t^{(1)}}(\xi) + \widehat{\sigma_t^{(2)}}(\xi)   = \frac{C_1 e^{2 \pi i |\xi| t}}{ |t|^{1/2} | \xi|^{1/2}}   + \frac{C_2 e^{-2 \pi i |\xi| t}}{ |t|^{1/2} | \xi|^{1/2}}, \end{equation}
where $\sigma_{t,\tau}^{(3)}$ is the remainder term which satisfies 
\begin{equation} \label{remainderdecay} \left\lvert \widehat{\sigma_{t,\tau}^{(3)}}(\xi) \right\rvert \lesssim \K^{O(1)} |\xi|^{-3/2}, \end{equation}
and where $\chi_{T,\tau}$ is multiplication of $\chi_T$ on the Fourier side by either $\psi_{\tau}(-\xi/|\xi|)$ or $\psi_{\tau}(\xi/|\xi|)$, with $\psi_{\tau}(-\xi/|\xi|)$ being used for $\sigma_t^{(1)}$, and $\psi_{\tau}(\xi/|\xi|)$ being used for $\sigma_t^{(2)}$, in \eqref{surfacedecomp}. 
If the third sum dominates in \eqref{surfacedecomp},  then by pigeonholing (resulting in a loss by a factor $\K$) we can assume the $\tau$ attaining the sup is independent of $B$. Applying Cauchy-Schwarz, followed by Plancherel's theorem and the Fourier decay of $\sigma_{t,\tau}^{(3)}$ from \eqref{remainderdecay}, then Plancherel again, gives the upper bound 
\begin{multline*} I_{A,\K}^{\broad}(P) \lesssim \K^{10}\delta^{-O(\epsilon^2)} \delta^{-4} \delta^{3/2} \left( |P| \delta^3 \right)^{1/2} K_{\beta,P,\delta}^{1/2} K_{\alpha,\mathbb{T},\delta}^{1/2}\left( |\mathbb{T}| \delta^2 \right)^{1/2} \\
= \K^{10}\delta^{-O(\epsilon^2)} K_{\beta,P,\delta}^{1/2} K_{\alpha,\mathbb{T},\delta}^{1/2}|P|^{1/2} |\mathbb{T}|^{1/2}. \end{multline*}
Since $K_{\beta,P,\delta} \le |P|$, this implies the bound stated in the theorem.

Therefore, by symmetry we may assume that the first sum in \eqref{surfacedecomp}, corresponding to $\sigma_t^{(1)}$, dominates. For each $B \in P$, we claim that there exists a triple $(\tau_1, \tau_2, \tau_3)$, depending on $B$, with $\dist(\tau_i, \tau_j) \gtrsim \K^{-1}$ for all $i \neq j$, and such that
\begin{multline} \label{broadtrilinearcircle} \inf_{(V_1,V_2) \in G(1,2) } \sup_{d(\tau, V_a) \geq \K^{-1} \forall \, a \in \{1,2\} } \\ \int_{B\left((x_B,t_B), \delta^{1+10\epsilon^2}\right)} \left\lvert \sum_{T \in \mathbb{T}} \chi_{T,\tau} \ast \sigma_t^{(1)} \ast \widecheck{\phi}(x) \right\rvert \, dx \, dt\\
\lesssim \prod_{j=1}^3 \left(\int_{B\left((x_B,t_B), \delta^{1+10\epsilon^2}\right)} \left\lvert \sum_{T \in \mathbb{T}} \chi_{T,\tau_j} \ast \sigma_t^{(1)} \ast \widecheck{\phi}(x) \right\rvert \, dx \, dt \right)^{1/3}. \end{multline}
This can be shown as follows. Let $(V_1, V_2)$ attain the infimum in the left-hand side. Let $\tau_1$ be the arc attaining the inner supremum given $V_1$ and $V_2$, so that $\dist(\tau_1,V_j) \geq \K^{-1}$ for $j \in \{1,2\}$ and
\begin{multline*} \sup_{d(\tau, V_a) \geq \K^{-1} \forall \, a \in \{1,2\} } \int_{B\left((x_B,t_B), \delta^{1+10\epsilon^2}\right)} \left\lvert \sum_{T \in \mathbb{T}} \chi_{T,\tau} \ast \sigma_t^{(1)}  \ast \widecheck{\phi}(x) \right\rvert \, dx \, dt \\
=  \int_{B\left((x_B,t_B), \delta^{1+10\epsilon^2}\right)} \left\lvert \sum_{T \in \mathbb{T}} \chi_{T,\tau_1} \ast \sigma_t^{(1)}  \ast \widecheck{\phi}(x) \right\rvert \, dx \, dt. \end{multline*}
There must exist $\tau_2$ with $\dist(\tau_2, \tau_1) \gtrsim \K^{-1}$ and 
\begin{multline*} \int_{B\left((x_B,t_B), \delta^{1+10\epsilon^2}\right)} \left\lvert \sum_{T \in \mathbb{T}} \chi_{T,\tau_2} \ast \sigma_t^{(1)}  \ast \widecheck{\phi}(x) \right\rvert \, dx \, dt \\
\gtrsim \int_{B\left((x_B,t_B), \delta^{1+10\epsilon^2}\right)} \left\lvert \sum_{T \in \mathbb{T}} \chi_{T,\tau_1} \ast \sigma_t^{(1)} \ast \widecheck{\phi}(x) \right\rvert \, dx \, dt, \end{multline*} 
since otherwise we would have
\begin{multline} \label{contradictioncircle} \inf_{(W_1,W_2) \in G(1,2) } \sup_{d(\tau, W_a) \geq \K^{-1} \forall \, a \in \{1,2\} } \\ \int_{B\left((x_B,t_B), \delta^{1+10\epsilon^2}\right)} \left\lvert \sum_{T \in \mathbb{T}} \chi_{T,\tau} \ast \sigma_t^{(1)}\ast \widecheck{\phi}(x) \right\rvert \, dx \, dt \\
\ll \int_{B\left((x_B,t_B), \delta^{1+10\epsilon^2}\right)} \left\lvert \sum_{T \in \mathbb{T}} \chi_{T,\tau_1} \ast \sigma_t^{(1)}  \ast \widecheck{\phi}(x) \right\rvert \, dx \, dt; \end{multline}
by picking both which $W_i$ to be lines passing through the origin and the centre of $\tau_1$, and this would give a contradiction since the right-hand side of \eqref{contradictioncircle} equals the left-hand side of \eqref{contradictioncircle}; by the definition of $\tau_1$. Similarly, there must exist $\tau_3$ with $\dist(\tau_3, \tau_1) \gtrsim \K^{-1}$, $\dist(\tau_3, \tau_2) \gtrsim \K^{-1}$, and 
\begin{multline*} \int_{B\left((x_B,t_B), \delta^{1+10\epsilon^2}\right)} \left\lvert \sum_{T \in \mathbb{T}} \chi_{T,\tau_3} \ast \sigma_t^{(1)} \ast \widecheck{\phi}(x) \right\rvert \, dx \, dt \\
\gtrsim \int_{B\left((x_B,t_B), \delta^{1+10\epsilon^2}\right)} \left\lvert \sum_{T \in \mathbb{T}} \chi_{T,\tau_1} \ast \sigma_t^{(1)}  \ast \widecheck{\phi}(x) \right\rvert \, dx \, dt; \end{multline*} 
since otherwise we could take, for $i \in \{1,2\}$, $W_i$ to be the line through the origin and the centre of $\tau_i$ in \eqref{contradictioncircle}. Hence, by \eqref{broadtrilinearcircle} and since the first sum in \eqref{surfacedecomp} dominates,
\begin{multline*} I_{A,\K}^{\broad}(P, \mathbb{T}) \lesssim \\\delta^{-4-30\epsilon^2} \sum_{B \in P} \prod_{j=1}^3 \left(\int_{B\left((x_B,t_B), \delta^{1+10\epsilon^2}\right)}\left\lvert \sum_{T \in \mathbb{T}} \chi_{T,\tau_j(B)} \ast \sigma_t^{(1)}  \ast \widecheck{\phi}(x) \right\rvert \, dx \, dt \right)^{1/3}.\end{multline*}
The number of distinct triples $(\tau_{1}(B), \tau_{2}(B), \tau_{3}(B))$ as $B$ varies over $P$ is $\lesssim \K^3$, so by pigeonholing we can find a fixed triple $(\tau_{1}, \tau_{2}, \tau_{3})$, such that 
\begin{multline} \label{allPcircle} I_{A,\K}^{\broad}(P, \mathbb{T}) \lesssim \\\delta^{-4-30\epsilon^2}  \K^3 \sum_{B \in P} \left( \prod_{j=1}^3  \int_{B\left((x_B,t_B), \delta^{1+10\epsilon^2}\right)} \left\lvert \sum_{T \in \mathbb{T}} \chi_{T,\tau_j} \ast \sigma_t^{(1)}  \ast \widecheck{\phi}(x) \right\rvert \, dx \, dt\right)^{1/3}.\end{multline}
The idea of the rest of the proof is that due to the Fourier support of the integrand, for each ball $B \in P$ we can treat the integrands above as constant, and therefore we can interchange the product and $1/3$ power with the integral. 
If this were literally true, then moving the product and $1/3$ power inside the integral, applying Hölder's inequality, then then the trilinear Fourier restriction theorem for the cone from \cite[Theorem~2.1]{leevargas} would finish the proof. This idea is only a heuristic and not strictly correct, so the rest of the proof consists of the technically correct version of this idea. We note that this is a standard heuristic; see e.g.~\cite{duzhang,lucarogers}.

For each $j \in \{1,2,3\}$ let $f_j(x,t) = \sum_{T \in \mathbb{T}} \chi_{T,\tau_j} \ast \sigma_t^{(1)} \ast \widecheck{\phi}(x)$. Then 
\[ f_j = f_j \ast \widecheck{\eta}, \]
where $\eta$ is a smooth bump function supported in a ball of radius $\delta^{-1-10\epsilon^2}$ in $\mathbb{R}^3$, and equal to 1 on a slightly smaller ball containing the Fourier support of $f_j$. Therefore,
\[ |f_j| \lesssim |f_j| \ast \zeta, \] where 
\[ \zeta(x) = \frac{\delta^{-3-30\epsilon^2}}{1 + \delta^{-4(1+10\epsilon^2 ) } |x|^4 }. \]
 The number $4$ above is not imporant; all that matters is that it is some number greater than 3. The function $\zeta$ is positive and locally constant on balls of radius $\delta^{1+10\epsilon^2}$ up to a factor $\sim 1$ (this can be shown by rescaling and considering the case $\delta=1$). Therefore, $ |f_j| \ast \zeta$ is also constant on balls of radius $\delta$, up to a factor $\sim 1$. The function $\zeta$ satisfies 
\begin{equation} \label{zetaboundscircle} \left\lVert \zeta \right\rVert_1 \lesssim 1, 
\end{equation} 
By the locally constant property, \eqref{allPcircle} implies that
\[ I_{A,\K}^{\broad}(P, \mathbb{T}) \lesssim \delta^{-4-30\epsilon^2}  \K^3 \sum_{B \in P} \int \chi_{B\left((x_B,t_B), \delta^{1+10\epsilon^2}\right)} \prod_{j=1}^3 \left(|f_j| \ast \zeta \right)^{1/3} \, dx \, dt. \]
Moving the sum inside the integral and then applying Hölder's inequality to the above gives
\begin{multline} \label{incidencetrilinearcircle} I_{A,\K}^{\broad}(P, \mathbb{T}) \lesssim \delta^{-4-30\epsilon^2}  \K^3  \left( \delta^3 |P| \right)^{2/3} K_{\beta,P,\delta}^{1/3} \times \\ \left(\iiint  \prod_{k=1}^3 \zeta(y_k) \int_{B_2(0,20) \times [1/2,5/2]} \prod_{j=1}^3 |f_j((x,t)-y_j)| \, dx \, dt \, dy_1 \, dy_2 \, dy_3 \right)^{1/3},
\end{multline}
where above we used the assumption that the centres of the circles in $P$ all lie in $B_2(0,10)$ to restrict the domain of integration. Since $f_j(x,t) = \sum_{T \in \mathbb{T}} \chi_{T,\tau_j} \ast \sigma_t^{(1)} \ast\widecheck{\phi}(x)$, let $g_j = \sum_{T \in \mathbb{T}} \chi_{T,\tau_j} \ast \widecheck{\phi}$, so that $f_j(x,t) = g_j \ast \sigma_t^{(1)}(x)$. By recalling the definition of $\sigma_t^{(1)}$ from \eqref{sigmaoneseconddefn}, for $1/2 \leq t \leq 5/2$,
\[ \left\lvert f_j(x,t)\right\rvert \sim \left\lvert\int_{\mathbb{R}^2} \frac{e^{2\pi i \langle (\xi, |\xi|), (x,t) \rangle} \widehat{g_j}(\xi)}{|\xi|^{1/2}} \, d\xi\right\rvert. \]
Let $E$ be the extension operator for the light cone in $\mathbb{R}^3$, given by
\[ Eg(x,t) =  \int_{\mathbb{R}^2} e^{2\pi i \langle (\xi, |\xi|), (x,t) \rangle} g(\xi) \, d\xi, \]
so that by the above, for $1/2 \leq t \leq 5/2$,  and any $y_j$,
\[ \left\lvert f_j((x,t)-y_j)\right\rvert \sim \left\lvert E\left( e^{- 2\pi i \langle (\cdot, |\cdot| ), y_j \rangle }\widehat{g_j}/|\cdot|^{1/2}\right)(x,t)\right\rvert. \]
Hence, by a change of variables, fo any $y_1$, $y_2$, $y_3$,
\begin{multline} \label{firstvariablechange} \int_{B_2(0,20) \times [1/2, 5/2]} \prod_{j=1}^3 |f_j((x,t)-y_j)| \, dx \, dt \\
\lesssim \delta^3\int_{B_3(0, 100\delta^{-1} ) } \prod_{j=1}^3  \left\lvert E\left(e^{- 2\pi i \langle (\cdot, |\cdot| ), y_j \rangle }\widehat{g_j}/|\cdot|^{1/2}\right)(\delta x,\delta t)\right\rvert \, dx \, dt. \end{multline}
For each $j \in \{1,2,3\}$, define $h_j^{(y_j)}$ by $h_j^{(y_j)}(\eta) = |\eta|^{-1/2} e^{- 2\pi i \langle (\eta/\delta, |\eta/\delta| ), y_j \rangle }\widehat{g_j}(\eta/\delta)$, so that by another change of variables
\[ E\left(e^{- 2\pi i \langle (\cdot, |\cdot| ), y_j \rangle }\widehat{g_j}/|\cdot|^{1/2}\right)(\delta x,\delta t) = \delta^{-3/2} Eh_j^{(y_j)}(x,t). \]
Hence \eqref{firstvariablechange} becomes
\begin{multline*} \int_{B_2(0,20) \times [1/2,5/2]} \prod_{j=1}^3 |f_j((x,t)-y_j)| \, dx \, dt \\
\lesssim \delta^{-3/2} \int_{B_3(0,100 \delta^{-1} ) } \prod_{j=1}^3 \left\lvert Eh_j^{(y_j)}(x,t) \right\rvert \, dx \, dt. \end{multline*}
For each $j$ and $y_j$, the support of $h_j^{(y_j)}$ is a rescaling of the Fourier support of $g_j$, which is contained in the sector $\{\xi : \xi/|\xi| \in \tau_j \}$. Thus, the $\K^{-1}$ separation of the arcs $\tau_j$ implies that the functions $h_j^{(y_j)}$ have $\K^{-1}$ angularly separated supports from other. Therefore, we can apply the trilinear Fourier restriction theorem for the cone (see \cite[Theorem~2.1]{leevargas} for the exact version) to the right-hand side of the above, to get
\begin{multline*} \int_{B_3(0,100)} \prod_{j=1}^3 |f_j((x,t)-y_j)| \, dx \, dt \lesssim \delta^{-3/2-O(\epsilon^2)} \K^{O(1)} \prod_{j=1}^3 \left\lVert h_j^{(y_j)}\right\rVert_2 \\
\lesssim \delta^{3/2-O(\epsilon^2)} \K^{O(1)} \prod_{j=1}^3 \|g_j\|_2. \end{multline*}
Substituting the above into \eqref{incidencetrilinearcircle} and using \eqref{zetaboundscircle} gives
\begin{multline*} I_{A,\K}^{\broad}(P, \mathbb{T}) \lesssim \delta^{-4-O(\epsilon^2)}  \K^{50}  \left( \delta^3 |P| \right)^{2/3} K_{\beta,P,\delta}^{1/3}  K_{\alpha,\mathbb{T},\delta}^{1/2} \delta^{1/2} \left( \delta^2 \left\lvert \mathbb{T} \right\rvert\right)^{1/2} \\
 = \delta^{-\frac{1}{2}-O(\epsilon^2)}  \K^{50}  K_{\beta,P,\delta}^{1/3}  K_{\alpha,\mathbb{T},\delta}^{1/2}   |P|^{2/3} \left\lvert \mathbb{T} \right\rvert^{1/2}.\end{multline*}
 This finishes the proof of \eqref{trilinearrestrictioncircle} in the remaining high case. \end{proof}

\subsection{Broad incidences between sine waves and discs}
In this subsection, we prove the (broad) incidence bounds corresponding to \eqref{sine2} and \eqref{sine1} in Theorem~\ref{curvedfurstenberg} for sine waves using \texorpdfstring{$L^3$}{L3} trilinear cone restriction and a trilinear Kakeya-type bound for slabs.  

\begin{definition}
Let $P$ be a set of $\delta$-balls in $\mathbb{R}^3$.  Let $\mathbb{T}$ be a set of $\delta$-slabs normal to light rays, i.e., suppose that for each $T \in \mathbb{T}$, there is an angle $\theta = \theta(T) \in [0, 2 \pi)$ so that $T$ is normal to $\left( \cos \theta, \sin \theta, 1 \right)$. Given $\kappa>0$, define the number $I_{\kappa}(P, \mathbb{T})$ of ``$\kappa$-trilinear incidences'' between $P$ and $\mathbb{T}$ to be 
\begin{multline*} I_{\kappa}(P, \mathbb{T} ) = \sum_{B \in P}\\
 \left\lvert \left\{ (T_1,T_2,T_3) \in \mathbb{T}^3: T_1 \cap T_2 \cap T_3 \cap B \neq \emptyset, \quad 
\min_{i \neq j} \left\lvert \theta(T_i)-\theta(T_j) \right\rvert  \geq \kappa 
 \right\} \right\rvert^{1/3}, \end{multline*}
 where $\theta(T)$ is the angle $\theta \in [0, 2\pi)$ such that $T$ is normal to $\left( \cos \theta, \sin \theta, 1 \right)$. Dualizing the above definition, if $P$ is a set of $\delta$-neighbourhoods of sine waves in $\mathbb{R}^2$ (i.e.~graphs of the functions $\theta \mapsto \frac{ a \cos \theta + b \sin \theta + c }{\sqrt{2}}$ over $[0, 2\pi]$), and $\mathbb{T}$ is a set of $\delta$-discs in the plane, and $\kappa>0$, define the number $I_{\kappa}(P, \mathbb{T})$ of $\kappa$-trilinear incidences between $P$ and $\mathbb{T}$ to be 
\begin{multline*} I_{\kappa}(P, \mathbb{T} ) = \sum_{B \in P}\\
 \left\lvert \left\{ (T_1,T_2,T_3) \in \mathbb{T}^3: T_j \cap B \neq \emptyset \, \forall j, \quad 
\min_{i \neq j}\dist(T_i, T_j)  \geq \kappa 
 \right\} \right\rvert^{1/3}. \end{multline*}
 \end{definition}
For any $\kappa>0$, $I_{\kappa}(P, \mathbb{T}) \leq I_0 (P, \mathbb{T}) = I(P, \mathbb{T})$, but upper bounds on $I_{\kappa}(P, \mathbb{T})$ with $\kappa$ an arbitrarily small positive constant will be applicable to Furstenberg sets in a similar way to upper bounds on $I(P, \mathbb{T})$. 

As in the circular case, we also define a slightly weaker ``broad'' version, which works better with the high-low method. This setup is based on the broad norms defined in \cite{guth2}. Let $A$ and $\K$ be large parameters, and let $\{\tau\}_{\tau}$ be a boundedly overlapping cover of the circle
\[ S^2 \cap \left\{ (x_1,x_2,x_3) \in \mathbb{R}^3 : x_3 = \sqrt{x_1^2+x_2^2} \right\}, \]
by arcs of diameter $\K^{-1}$. Given a set $P$ of $\delta$-balls in $\mathbb{R}^3$, a set $\mathbb{T}$ of $\delta$-slabs normal to light rays, let $\mathbb{T}_{\tau}$ be the subset of $T \in \mathbb{T}$ with normals contained in $\tau$. Define
\[ I_{A,\K}^{\broad}(P, \mathbb{T} ) = \sum_{B \in P} \inf_{(V_1, V_2, \dotsc, V_A) \in G(1,3) } \sup_{\dist(\tau, V_a) \geq \K^{-1} \forall \, a } I(P, \mathbb{T}_{\tau} ). \]
Again, we dualize the above definition as follows. If $P$ is a set of $\delta$-neighbourhoods of sine waves in $\mathbb{R}^2$, and $\mathbb{T}$ is a set of $\delta$-discs in the plane, for any interval $\tau \subseteq [0, 2\pi]$ of diameter $\K^{-1}$, let $\mathbb{T}_{\tau}$ be the set of discs $D \in \mathbb{T}$ with $D \cap \tau \times \mathbb{R} \neq \emptyset$. Define
\[ I_{A,\K}^{\broad}(P, \mathbb{T} ) = \sum_{B \in P} \inf_{(\theta_1, \theta_2, \dotsc, \theta_A) \in [0, 2\pi) } \sup_{\dist(\tau, \theta_a) \geq \K^{-1} \forall \, a } I(P, \mathbb{T}_{\tau} ). \]

\begin{theorem} \label{trilinearincidence} Let $0 \leq \alpha \leq 2$, $0 \leq \beta \leq 3$, let $\delta>0$, let $P$ be a $(\delta, \beta, K_{\beta,P,\delta})$-set of $\delta$-balls in $\mathbb{R}^3$, and let $\mathbb{T}$ be a set of $\delta$-slabs in $\mathbb{R}^3$. Then for any $\kappa>0$, there is a constant $C=C(\kappa) = \kappa^{-O(1)}$ such that 
\begin{equation} \label{trilinearkakeya} I_{\kappa}(P, \mathbb{T}) \leq C_{\kappa}   K_{\beta,P,\delta}^{1/3}  |P|^{2/3}\left\lvert \mathbb{T} \right\rvert. \end{equation}
If additionally $\mathbb{T}$ is a $(\delta, \alpha, K_{\alpha, \mathbb{T}, \delta})$-set of $\delta$-slabs in $\mathbb{R}^3$, if $3\alpha + 2\beta \leq 9$, and if the balls in $P$ are all contained in a fixed ball of radius 1, then for any $\epsilon >0$, $A \geq |\log \delta|$ and $K \geq 1$,
\begin{equation} \label{trilinearrestriction} I_{A,\K}^{\broad}(P, \mathbb{T}) \leq  C_{\epsilon} \K^{100} \delta^{-\epsilon - 1/2} K_{\beta,P,\delta}^{1/3} K_{\alpha, \mathbb{T}, \delta}^{1/2} |P|^{2/3} \left\lvert \mathbb{T} \right\rvert^{1/2}. \end{equation}
Both bounds hold under the same conditions on the parameters if $P$ is a set of $\delta$-neighbourhoods of sine waves in the plane, and $\mathbb{T}$ is a set of $\delta$-discs in the plane. 
\end{theorem}

\begin{proof} Let $P$ be a $(\delta, \beta, K_{\beta,P,\delta})$-set of $\delta$-balls in $\mathbb{R}^3$, and let $\mathbb{T}$ be a set of $\delta$-slabs in $\mathbb{R}^3$. By definition, 
\[I_{\kappa}(P, \mathbb{T}) \lesssim \delta^{-3} \int \left( \sum_{B \in P } \chi_{10B} \right) \left( \sum_{\substack{T_1,T_2, T_3 \in \mathbb{T}\\ \min_{i \neq j} \left\lvert \theta(T_i)-\theta(T_j) \right\rvert  \geq \kappa }} \chi_{T_1} \chi_{T_2} \chi_{T_3} \right)^{1/3}. \]
By Hölder's inequality, 
\[ I_{\kappa}(P, \mathbb{T}) \lesssim \delta^{-3} \left( \delta^3 |P| \right)^{2/3} K_{\beta,P,\delta}^{1/3} \left(\int  \sum_{\substack{T_1,T_2, T_3 \in \mathbb{T}\\ \min_{i \neq j} \left\lvert \theta(T_i)-\theta(T_j) \right\rvert  \geq \kappa }} \chi_{T_1} \chi_{T_2} \chi_{T_3} \right)^{1/3}. \]
If $T_1,T_2, T_3 \in \mathbb{T}$ are such that $\prod_{i \neq j} \left\lvert \theta(T_i)-\theta(T_j) \right\rvert \geq \kappa$, then $T_1 \cap T_2 \cap T_3$ is contained\footnote{This is the sine wave analogue of the geometric observation in \cite{liu} that any three $\sim 1$-separated $\delta$-discs in the plane uniquely determine a circle (up to an error $\delta$) passing through them. This observation was used in \cite{liu} to prove the same lower bound for circular Furstenberg sets as what will follow from \eqref{trilinearkakeya} for sine wave Furstenberg sets.} in a ball of radius $\lesssim_{\kappa} \delta$. Applying this\footnote{This elementary trilinear Kakeya-type inequality has been observed before, see e.g.~\cite[Theorem~11.8]{demeter}.} to the above gives 
\[ I_{\kappa}(P, \mathbb{T}) \lesssim_{\kappa} \delta^{-3} \left( \delta^3 |P| \right)^{2/3} K_{\beta,P,\delta}^{1/3} \delta \lvert \mathbb{T} \rvert =|P|^{2/3} K_{\beta,P,\delta}^{1/3} \lvert \mathbb{T} \rvert. \]
This proves \eqref{trilinearkakeya}. 

Now suppose additionally that $\mathbb{T}$ is a $(\delta, \alpha, K_{\alpha, \mathbb{T}, \delta})$-set of $\delta$-slabs in $\mathbb{R}^3$, and that $3\alpha + 2\beta \leq 9$. Let $\epsilon >0$, $A \geq |\log \delta|$, and $\K \geq 1$. We can assume that $\K \leq \delta^{-1/2}$, since otherwise \eqref{trilinearrestriction} is trivial. Assume that $P$ is contained in a fixed ball of radius 1, which we can assume by translation invariance is the unit ball. We will prove \eqref{trilinearrestriction} by induction on $\delta$. The base case $\delta \sim 1$ follows from the trivial bound $I(P, \mathbb{T}) \leq |P||\mathbb{T}|$. By definition, 
\begin{multline} \label{beforehighlow} I_{A,\K}^{\broad}(P, \mathbb{T}) \lesssim \delta^{-3-30\epsilon^2} \sum_{B \in P} \\
\inf_{(V_1,\dotsc, V_A) \in G(1,3) } \sup_{d(\tau, V_a) \geq \K^{-1} \forall \, a } \int_{B\left(x_B, \delta^{1+10\epsilon^2}\right)} \left\lvert \sum_{T \in \mathbb{T}_{\tau} } \chi_{10T} \right\rvert, \end{multline}
where $x_B$ is the centre of $B$, the $\chi_{10T}$ are smooth cutoffs adapted to slabs of dimensions $\sim 1 \times 1 \times \delta$ (here we use that the balls in $P$ are contained a ball of radius 1). Let $\phi$ be a smooth bump function equal to 1 on $B_3\left(0, \delta^{-1-\epsilon^2}\right) \setminus B_3\left(0, \delta^{\epsilon^2-1}\right)$,  vanishing outside $B_3\left(0, 2\delta^{-1-\epsilon^2} \right)$ and vanishing in $B_3\left(0, \delta^{\epsilon^2-1}/2\right)$. Let $\psi$ be a smooth bump function equal to $1-\phi$ in $B_3\left(0, \delta^{\epsilon^2-1} \right)$ and extended by zero outside this ball. By \eqref{beforehighlow}, we have (apart from a negligible error term) 
\begin{multline*} I_{A,\K}^{\broad}(P, \mathbb{T}) \lesssim \delta^{-3-30\epsilon^2} \sum_{B \in P} \inf_{(V_1,\dotsc, V_A) \in G(1,3) } \sup_{d(\tau, V_a) \geq \K^{-1} \forall \, a } \\
\left(\int_{B\left(x_B, \delta^{1+10\epsilon^2}\right)} \left\lvert \sum_{T \in \mathbb{T}_{\tau} } \chi_{10T}  \ast \widecheck{\phi}\right\rvert + \int_{B\left(x_B, \delta^{1+10\epsilon^2}\right)} \left\lvert \sum_{T \in \mathbb{T}_{\tau} } \chi_{10T}  \ast \widecheck{\psi}\right\rvert \right).\end{multline*}
This implies that
\begin{multline} \label{highlowdecomp} I_{A,\K}^{\broad}(P, \mathbb{T}) \lesssim \delta^{-3-30\epsilon^2} \sum_{B \in P} \\\inf_{(V_1,V_2) \in G(1,3) } \sup_{d(\tau, V_a) \geq \K^{-1} \forall \, a \in \{1,2\} } \int_{B\left(x_B, \delta^{1+10\epsilon^2}\right)} \left\lvert \sum_{T \in \mathbb{T}_{\tau} } \chi_{10T}  \ast \widecheck{\phi}\right\rvert \\
+ \delta^{-3-30\epsilon^2} \sum_{B \in P} \\\inf_{(V_3, \dotsc, V_A) \in G(1,3) } \sup_{d(\tau, V_a) \geq \K^{-1} \forall \, a \in \{3, \dotsc, A\} }  \int_{B\left(x_B, \delta^{1+10\epsilon^2}\right)} \left\lvert \sum_{T \in \mathbb{T}_{\tau} } \chi_{10T}  \ast \widecheck{\psi}\right\rvert.\end{multline}
The high case is when the first term dominates, and the low case is when the second term dominates.  In the low case, let $S = \delta^{-\epsilon^2-\epsilon^4}$. Then, using
\[ \left\lVert \chi_{10T} \ast \widecheck{\psi} \right\rVert_{\infty} \leq \left\lVert \widehat{ \chi_{10T}} \psi \right\rVert_{1} \lesssim \delta^{\epsilon^2}, \]
and that $\widecheck{\psi}$ is negligible outside $B_3(0, \delta S)$, we have that (apart from a negligible error term)
\[ \left\lvert\chi_{10T} \ast \widecheck{\psi} \right\rvert \lesssim \delta^{\epsilon^2} \chi_{ST}, \]
where $ST$ is the slab $T$ thickened by the factor $S$ (this is where we use that the $\chi_{10T}$ are smooth). Hence, apart from a negligible error term, 
\[ I_{A,\K}^{\broad}(P, \mathbb{T}) \lesssim  \delta^{-\epsilon^4} S^{-1} \sum_{B \in P} \inf_{(V_3, \dotsc, V_A) \in G(1,3) } \sup_{d(\tau, V_a) \geq \K^{-1} \forall \, a \in \{3, \dotsc, A\} }  I\left( \mathbb{T}_{S, \tau}, B \right). \]
 where the balls in $P_S$ are the balls from $P$ scaled by $S$, $\mathbb{T}_S$ consists of the $\delta S$ neighbourhoods of the planes defining the slabs in $\mathbb{T}$. This gives
 \begin{equation} \label{inductive} I_{A, \K}^{\broad}(P, \mathbb{T} ) \lesssim S^{-1} \delta^{-\epsilon^4} I_{A-2, \K}^{\broad} \left(P_S, \mathbb{T}_S \right). \end{equation}
The assumption $A \geq |\log \delta|$ implies that $A-2 \geq |\log \delta S |$ provided that $\delta$ is sufficiently small (depending only on $\epsilon$, the remaining cases being handled by the constant $C_{\epsilon}$ in the theorem statement).  Since 
\[ K_{\beta, P_S, \delta S} \lesssim S^{\beta} K_{\beta,P,\delta}, \qquad K_{\alpha, \mathbb{T}_S, \delta} \lesssim S^{\alpha} K_{\alpha, \mathbb{T}, \delta}, \]
 we can apply induction to \eqref{inductive} to get 
\[ I_{A,\K}^{\broad}(P, \mathbb{T}) \lesssim S^{\frac{\alpha}{2} + \frac{\beta}{3} - \frac{3}{2} - \epsilon} \delta^{-\epsilon^4} C_{\epsilon} \K^{100} \delta^{-\epsilon - 1/2} K_{\beta,P,\delta}^{1/3} K_{\alpha, \mathbb{T}, \delta}^{1/2} |P|^{2/3} \left\lvert \mathbb{T} \right\rvert^{1/2}. \]
The exponent of $S$ is negative by the assumption that $3\alpha+2\beta \leq 9$, so this shows that \eqref{trilinearrestriction} holds in the low case. 

Now suppose we are in the high case, so that the first part of \eqref{highlowdecomp} dominates:
\begin{multline} \label{star} I_{A,\K}^{\broad}(P, \mathbb{T}) \lesssim \delta^{-3-30\epsilon^2} \sum_{B \in P} \\\inf_{(V_1,V_2) \in G(1,3) } \sup_{d(\tau, V_a) \geq \K^{-1} \forall \, a \in \{1,2\} } \int_{B\left(x_B, \delta^{1+10\epsilon^2}\right)} \left\lvert \sum_{T \in \mathbb{T}_{\tau} } \chi_{10T}  \ast \widecheck{\phi}\right\rvert.\end{multline}
For each $B \in P$, we claim that there exists a triple $(\tau_1, \tau_2, \tau_3)$, depending on $B$, with $\dist(\tau_i, \tau_j) \gtrsim \K^{-1}$ for all $i \neq j$, and such that
\begin{multline} \label{broadtrilinear} \inf_{(V_1,V_2) \in G(1,3) } \sup_{d(\tau, V_a) \geq \K^{-1} \forall \, a \in \{1,2\} } \int_{B\left(x_B, \delta^{1+10\epsilon^2}\right)} \left\lvert \sum_{T \in \mathbb{T}_{\tau} } \chi_{10T}  \ast \widecheck{\phi}\right\rvert\\
\lesssim \prod_{j=1}^3 \left(\int_{B\left(x_B, \delta^{1+10\epsilon^2}\right)} \left\lvert \sum_{T \in \mathbb{T}_{\tau_j} } \chi_{10T} \ast \widecheck{\phi} \right\rvert \right)^{1/3}. \end{multline}
To see this, let $(V_1, V_2)$ attain the infimum in the left-hand side in \eqref{broadtrilinear}. Let $\tau_1$ attain the inner supremum,  so that $\dist(\tau_1,V_j) \geq \K^{-1}$ for $j \in \{1,2\}$ and
\begin{multline*} \sup_{d(\tau, V_a) \geq \K^{-1} \forall \, a \in \{1,2\} } \int_{B\left(x_B, \delta^{1+10\epsilon^2}\right)} \left\lvert \sum_{T \in \mathbb{T}_{\tau} } \chi_{10T}  \ast \widecheck{\phi}\right\rvert  \\
= \int_{B\left(x_B, \delta^{1+10\epsilon^2}\right)}\left\lvert \sum_{T \in \mathbb{T}_{\tau_1} } \chi_{10T}  \ast \widecheck{\phi}\right\rvert. \end{multline*}
There must exist $\tau_2$ with $\dist(\tau_2, \tau_1) \gtrsim \K^{-1}$ and 
\[ \int_{B\left(x_B, \delta^{1+10\epsilon^2}\right)} \left\lvert \sum_{T \in \mathbb{T}_{\tau_2} } \chi_{10T} \ast \widecheck{\phi} \right\rvert \gtrsim \int_{B\left(x_B, \delta^{1+10\epsilon^2}\right)} \left\lvert \sum_{T \in \mathbb{T}_{\tau_1} } \chi_{10T} \ast \widecheck{\phi} \right\rvert, \]
since otherwise we could take $W_1$ and $W_2$ to be the lines through the centre of $\tau_1$, and get 
\begin{multline} \label{contradiction} \inf_{W_1,W_2 \in G(1,3) } \sup_{d(\tau, W_j) \geq \K^{-1} \forall j \in \{1,2\} } \int_{B\left(x_B, \delta^{1+10\epsilon^2}\right)} \left\lvert \sum_{T \in \mathbb{T}_{\tau} } \chi_{10T}  \ast \widecheck{\phi}\right\rvert \\
\ll \int_{B\left(x_B, \delta^{1+10\epsilon^2}\right)} \left\lvert \sum_{T \in \mathbb{T}_{\tau_1} } \chi_{10T}  \ast \widecheck{\phi}\right\rvert, \end{multline}
which is a contradiction since both sides of \eqref{contradiction} are equal by the definition of $\tau_1$. Similarly, there must exist $\tau_3$ with $\dist(\tau_3, \tau_1) \gtrsim \K^{-1}$, $\dist(\tau_3, \tau_2) \gtrsim \K^{-1}$, and 
\[ \int_{B\left(x_B, \delta^{1+10\epsilon^2}\right)} \left\lvert \sum_{T \in \mathbb{T}_{\tau_3} } \chi_{10T} \ast \widecheck{\phi} \right\rvert \gtrsim \int_{B\left(x_B, \delta^{1+10\epsilon^2}\right)} \left\lvert \sum_{T \in \mathbb{T}_{\tau_1} } \chi_{10T} \ast \widecheck{\phi} \right\rvert; \]
since otherwise we could take $W_1,W_2$ to be the lines through the centres of $\tau_1$ and $\tau_2$ in \eqref{contradiction}. This verifies \eqref{broadtrilinear}, and therefore by \eqref{star}, 
\[ I_{A,\K}^{\broad}(P, \mathbb{T}) \lesssim \delta^{-3-30\epsilon^2} \sum_{B \in P} \prod_{j=1}^3 \left(\int_{B\left(x_B, \delta^{1+10\epsilon^2}\right)} \left\lvert \sum_{T \in \mathbb{T}_{\tau_j(B)} } \chi_{10T}  \ast \widecheck{\phi} \right\rvert\right)^{1/3}.\]
The number of distinct triples $(\tau_1(B), \tau_2(B), \tau_3(B))$ as $B$ varies over $P$ is $\lesssim \K^3$, so by pigeonholing we can find a fixed triple $(\tau_1, \tau_2, \tau_3)$, such that 
\begin{equation} \label{allP} I_{A,\K}^{\broad}(P, \mathbb{T}) \lesssim \delta^{-3-30\epsilon^2}  \K^3 \sum_{B \in P} \left( \prod_{j=1}^3  \int_{B\left(x_B, \delta^{1+10\epsilon^2}\right)} \left\lvert \sum_{T \in \mathbb{T}_{\tau_j} } \chi_{10T}   \ast \widecheck{\phi}  \right\rvert\right)^{1/3}.\end{equation}

 For each $j \in \{1,2,3\}$, let $f_j = \sum_{T \in \mathbb{T}_{\tau_j} } \chi_{10T}   \ast \widecheck{\phi}$, so that
\[ f_j = f_j \ast \widecheck{\eta}, \]
where $\eta$ is a smooth bump function supported in a ball of radius $\delta^{-1-10\epsilon^2}$, obtained by rescaling a bump function on the unit ball independent of $\delta$. Therefore,
\begin{equation} \label{reproducingformula} |f_j| \lesssim |f_j| \ast \zeta, \end{equation} where 
\[ \zeta(x) = \frac{\delta^{-3-30\epsilon^2}}{1 + \delta^{-4(1+10\epsilon^2 ) } |x|^4 }. \]
 The function $\zeta$ is positive and locally constant on balls of radius $\delta^{1+10\epsilon^2}$ up to a factor $\sim 1$, meaning that $\zeta(x) \lesssim \zeta(y)$ when $|x-y| \leq \delta^{1+10\epsilon^2}$ (this is easy to show by rescaling and considering the case $\delta=1$).  Since the locally constant property is preserved by convolution, $ |f_j| \ast \zeta$ is also locally constant on balls of radius $\delta^{1+10\epsilon^2}$, up to a factor $\sim 1$. By a change of variables, 
\begin{equation} \label{zetabounds} \left\lVert \zeta \right\rVert_1 \sim 1. \end{equation} 
By \eqref{reproducingformula} and the locally constant property, applied to \eqref{allP},
\[ I_{A,\K}^{\broad}(P, \mathbb{T}) \lesssim \delta^{-3-30\epsilon^2}  \K^3 \sum_{B \in P} \int_{B} \prod_{j=1}^3 \left(|f_j| \ast \zeta \right)^{1/3}. \]
By Hölder's inequality,
\begin{multline} \label{incidencetrilinear} I_{A,\K}^{\broad}(P, \mathbb{T}) \lesssim \delta^{-3-30\epsilon^2}  \K^3  \left( \delta^3 |P| \right)^{2/3} K_{\beta,P,\delta}^{1/3} \\ \left(\iiint  \zeta(y_1)\zeta(y_2)\zeta(y_3) \int_{B_3(0,2)} \prod_{j=1}^3 |f_j(x-y_j)| \, dx \, dy_1 \, dy_2 \, dy_3 \right)^{1/3}, \end{multline}
where we used the assumption that the balls in $P$ are contained in the unit ball to restrict the domain of integration to $B_3(0,2)$. Fix $y_1, y_2, y_3$ and let $F_j(x) = f_j(\delta x-  y_j)$ for $j \in \{1,2,3\}$. The trilinear Fourier restriction theorem for the cone in $\mathbb{R}^3$ implies that (apart from a negligible error term which we can ignore)
\begin{equation} \label{leevargaseq} \left\lVert \prod_{j=1}^3 F_j \right\rVert_{L^1(B_3(0, 2\delta^{-1} ) ) } \lesssim  \K^{O(1)}\delta^{\frac{3}{2} - O(\epsilon^2) } \prod_{j=1}^3 \|F_j\|_2; \end{equation}
see e.g.~\cite[Eq.~5]{leevargas} for the particular version of trilinear cone restriction that we are using. By rescaling \eqref{leevargaseq}, we get
\[ \left\lVert \prod_{j=1}^3 f_j(\cdot - y_j) \right\rVert_{L^1(B_3(0, 2 ) ) } \lesssim \K^{O(1)} \delta^{- O(\epsilon^2) } \prod_{j=1}^3 \|f_j\|_2. \]
Substituting the above into \eqref{incidencetrilinear} and using \eqref{zetabounds} gives
\begin{multline*} I_{A,\K}^{\broad}(P, \mathbb{T}) \lesssim \K^{O(1)}\delta^{-3-O(\epsilon^2)}   \left( \delta^3 |P| \right)^{2/3} K_{\beta,P,\delta}^{1/3}  K_{\alpha,\mathbb{T},\delta}^{1/2} \left( \delta \left\lvert \mathbb{T} \right\rvert\right)^{1/2} \\
 = \K^{O(1)}\delta^{-\frac{1}{2}-O(\epsilon^2)}   K_{\beta,P,\delta}^{1/3}  K_{\alpha,\mathbb{T},\delta}^{1/2}   |P|^{2/3} \left\lvert \mathbb{T} \right\rvert^{1/2}.\end{multline*}
 This finishes the proof of \eqref{trilinearrestriction} in the high case (the implicit constant in $O(1)$ is smaller than 100). \end{proof}

\subsection{Cinematic curve incidences via variable coefficient local smoothing}
In this subsection, we prove the incidence bound corresponding to \eqref{bound2}. We will in fact present the proof in the special case of circles and then explain the methods for generalising both \eqref{bound1} and \eqref{bound2} to cinematic curves (which includes sine waves and circles).

The incidence bound below for circles gives weaker results for circular Furstenberg sets than the combination of Theorem~\ref{theoremincidencecircle} and Theorem~\ref{trilinearincidencecircle}, but it has a simple proof and straightforward generalisation to cinematic curves, and in this level of generality there are some ranges for which it gives the best bounds for cinematic Furstenberg sets. 
\begin{theorem} \label{theoremincidencecircleL2} Let $0 \leq \beta \leq 3$, $0 \leq \alpha \leq 2$, and $0 < \delta < 1$. Let $P$ be a $\left( \delta, \beta, K_{\beta,P,\delta}\right)$-set of $\delta$-neighbourhoods of circles in the plane, of radii between 1 and 2. Let $\mathbb{T}$ be a $\left( \delta, \alpha, K_{\alpha, \mathbb{T}, \delta}\right)$-set of $\delta$-discs in the plane.  If $\alpha + \beta \leq 4$, then for any $\epsilon >0$, 
\begin{equation} \label{incidenceboundcircleL2} I(P, \mathbb{T}) \leq C_{\epsilon} \delta^{-\epsilon} \delta^{-1}  K_{\beta,P,\delta}^{1/2} K_{\alpha,\mathbb{T},\delta}^{1/2}  |P|^{1/2} |\mathbb{T}|^{1/2} .  \end{equation}
If $\alpha + \beta > 4$, then for any $\epsilon >0$, 
\begin{equation} \label{incidenceboundcircleL217} I(P, \mathbb{T}) \leq C_{\epsilon} \delta^{-\epsilon} \delta^{-2\lambda},  K_{\beta,P,\delta}^{\lambda} K_{\alpha,\mathbb{T},\delta}^{\lambda} |P|^{1-\lambda} |\mathbb{T}|^{1-\lambda}  \end{equation}
where $\lambda = \frac{1}{\alpha+\beta-2} < 1/2$. 
\end{theorem} 
\begin{proof} The setup of the proof is similar to the start of the proof of Theorem~\ref{theoremincidencecircle}.  We will use throughout the proof that \eqref{incidenceboundcircleL2} implies \eqref{incidenceboundcircleL217} when $\alpha + \beta > 4$. If we set 
\[ f = \sum_{T \in \mathbb{T}} \chi_{10T}, \]
where each $\chi_{10T}$ is a smooth bump function equal to 1 on $10T$ and vanishing outside $20T$, then
\[ I(P, \mathbb{T}) \lesssim \delta^{-4} \int \sum_{B \in P} \chi_{B\left(\left(x_B, t_B\right),\delta\right)} \left(\sigma_t \ast f\right)(x) \, dx \, dt. \]
Let $\phi$ be a smooth bump function equal to 1 on $B_2\left(0, \delta^{-\epsilon^2-1}\right) \setminus B_2\left(0, \delta^{\epsilon^2-1}\right)$, vanishing outside $B_2\left(0, 2\delta^{-\epsilon^2-1}\right)$ and vanishing in $B_2\left(0, \delta^{\epsilon^2-1}/2\right)$. Let $\psi$ be a smooth bump function supported equal to $1-\phi$ in $B_2\left(0, \delta^{\epsilon^2-1} \right)$ and extended by zero outside this ball. We define the ``high'' and ``low'' parts $f_h$, $f_l$ of $f$ by
\[ \widehat{f_h} = \widehat{f} \phi, \qquad \widehat{f_l} = \psi \widehat{f}. \]
Then
\begin{multline*} I(P, \mathbb{T}) \lesssim \delta^{-4} \int \sum_{B \in P} \chi_{B\left(\left(x_B, t_B\right),\delta\right)} \left\lvert \left(\sigma_t \ast f_h\right)(x)\right\rvert \, dx \, dt \\
+\delta^{-4} \int \sum_{B \in P} \chi_{B\left(\left(x_B, t_B\right),\delta\right)} \left\lvert \left(\sigma_t \ast f_l\right)(x)\right\rvert \, dx \, dt, \end{multline*}
where again we can assume that the negligible error terms do not dominate, as in this case the conclusion of the theorem follows. If we are in the high case, then by the Cauchy-Schwarz inequality,
\begin{equation} \label{latergeneralisation} I(P, \mathbb{T}) \lesssim \delta^{-4} \left( K_{\beta,P,\delta} |P| \delta^3 \right)^{1/2} \left(\int_{1/2}^{5/2} \int_{\mathbb{R}^2} \left\lvert \left(\sigma_{t} \ast f_h\right)(x) \right\rvert^2 \, dx \, dt\right)^{1/2}. \end{equation}
By the decay $|\widehat{\sigma}(\xi)| \lesssim |\xi|^{-1/2}$ (see \eqref{asymptotic})  and by Plancherel in the $x$-variable,
\begin{multline*} I(P, \mathbb{T})  \lesssim \delta^{-4} \left( K_{\beta,P,\delta} |P| \delta^3 \right)^{1/2} \delta^{1/2-O(\epsilon^2)} \left( \delta^2 K_{\alpha, \mathbb{T}, \delta} |\mathbb{T}| \right)^{1/2} \\
= \delta^{-1-O(\epsilon^2)} K_{\beta,P,\delta}^{1/2} K_{\alpha, \mathbb{T}, \delta}^{1/2} |P|^{1/2} |\mathbb{T}|^{1/2}. \end{multline*} This shows that \eqref{incidenceboundcircleL2} holds in the high case. Since  \eqref{incidenceboundcircleL2} implies \eqref{incidenceboundcircleL217} when $\alpha + \beta >4$, this proves the bound in the high case.

Now suppose we are in the low case, so that the second term dominates in \eqref{twoterms}:
\[ I(P, \mathbb{T}) \lesssim  \delta^{-4} \int \sum_{B \in P} \chi_{B\left(\left(x_B, t_B\right),\delta\right)} \left\lvert \left(\sigma_t \ast f_l\right)(x)\right\rvert \, dx \, dt. \]
This case is similar to the low case in the proof of Theorem~\ref{theoremincidencecircle}. Let $S = \delta^{-\epsilon^2-\epsilon^4}$, so that for any $(x,t) \in \mathbb{R}^2 \times [1,2]$ (ignoring negligible error terms), by the same justification as for \eqref{usedagainlater},
\[ (\sigma_t \ast f_l)(x) \lesssim S^{-2} \delta^{-O(\epsilon^4)} \sigma_t \ast \left( \sum_{T \in \mathbb{T}} \chi_{ST}\right)(x), \]
and so (assuming the negligible error terms do not dominate)
\begin{multline*} I(P, \mathbb{T}) \lesssim  S^{-2} \delta^{-4-O(\epsilon^4)} \int \sum_{B \in P} \chi_{B\left(\left(x_B, t_B\right),\delta\right)} \sigma_t \ast \left( \sum_{T \in \mathbb{T}} \chi_{ST}\right)(x) \, dx \, dt\\
\lesssim S^{-1} \delta^{-O(\epsilon^4)} I(P_S, \mathbb{T}_S), \end{multline*}
where $P_S$ is the set of $\delta S$ neighbourhoods of the same circles defining $P$, $\mathbb{T}_S$ is the set of discs from $\mathbb{T}$ with the same centres but with radii scaled by $S$. By induction, this gives
\begin{multline} \label{secondlastL2} I(P, \mathbb{T}) \lesssim \delta^{-O(\epsilon^4) }S^{-1} \times\\
\Bigg\{\begin{aligned} &(\delta S)^{-1-\epsilon}  K_{\alpha, \mathbb{T}_S, \delta S}^{1/2} K_{\beta,P_S,\delta S}^{1/2} |P_S|^{1/2} |\mathbb{T}_S|^{1/2} && \alpha + \beta \leq 4 \\
&(\delta S)^{-2\lambda-\epsilon} K_{\alpha, \mathbb{T}_S, \delta S}^{\lambda}  K_{\beta, P_S, \delta S}^{\lambda}|P_S|^{1-\lambda} |\mathbb{T}_S|^{1-\lambda} && \alpha + \beta> 4. \end{aligned} \end{multline}
By the inequalities
\[ K_{\alpha, \mathbb{T}_S, \delta S} \lesssim S^{\alpha} K_{\alpha, \mathbb{T}, \delta}, \qquad K_{\beta,P_S,\delta S} \lesssim S^{\beta} K_{\beta,P,\delta},\]
and since $|\mathbb{T}_S| = |\mathbb{T}|$ and $|P_S| = |P|$, 
\eqref{secondlastL2} becomes
\begin{multline*} I(P, \mathbb{T}) \lesssim \delta^{\epsilon^3-O(\epsilon^4) } \times\\
\Bigg\{\begin{aligned} & S^{-2 + \frac{\alpha}{2} + \frac{\beta}{2}} \delta^{-1-\epsilon}  K_{\alpha, \mathbb{T}, \delta}^{1/2} K_{\beta,P,\delta}^{1/2} |P|^{1/2} |\mathbb{T}|^{1/2} && \alpha + \beta \leq 4 \\
& S^{-1-2\lambda +\lambda \alpha + \lambda \beta}\delta^{-2\lambda-\epsilon} K_{\alpha, \mathbb{T}, \delta}^{\lambda}  K_{\beta,P,\delta}^{\lambda}|P|^{1-\lambda} |\mathbb{T}|^{1-\lambda} && \alpha + \beta >4. \end{aligned} \end{multline*}
In either case, the exponent of $S$ is non-positive, so the induction closes in the low case and this finishes the proof. \end{proof}

Below, we assume the definition of Fourier integral operator from the survey article \cite{beltranhickmansogge}. Suppose that $\{\Sigma_{x,t}\}_{x,t}$ is a family of smooth curves in the plane parametrised by $(x,t) \in \mathbb{R}^2 \times [1,2]$, such that 
\[ \Sigma_{x,t} = \left\{ y \in \mathbb{R}^2 : \Phi_t(x,y) = 0 \right\}, \]
for some smooth function $\Phi_t$ satisfying the ``Phong-Stein rotational curvature condition''
\begin{equation} \label{phongstein} \det\begin{pmatrix} \Phi_t & (\partial_y \Phi_t)^T \\
\partial_x \Phi_t & \partial_{xy} \Phi_t \end{pmatrix} \neq 0 \qquad \text{ when } \Phi_t = 0. \end{equation}
Let $\sigma_{x,t}$ be the measure $\frac{\mathcal{H}^1}{\left\lvert \partial_y \Phi_t\right\rvert}$ restricted to $\Sigma_{x,t}$, where $\mathcal{H}^1$ is the arc length measure when restricted to $\Sigma_{x,t}$. The denominator $\left\lvert \partial_y \Phi_t\right\rvert$ is nonvanishing on $\Sigma_{x,t}$ by the condition \eqref{phongstein}. The measure $\sigma_{x,t}$ is equivalent to
\[ \int f \, d\sigma_{x,t} = \lim_{\epsilon \to 0^+} \frac{1}{2\epsilon} \int_{|\Phi_t(x,y)| < \epsilon} f(y) \, dy, \]
for any continuous function $f$ of compact support; see \cite[p.~498]{stein2}.

 In~\cite[Example~2.14]{beltranhickmansogge} (for example) it is shown that if $a$ is a smooth function compactly supported in $\mathbb{R}^2 \times [1,2] \times \mathbb{R}^2$,
%
then the averaging operator
\begin{equation} \label{averaging} A_tf(x) = \int_{\Sigma_{x,t}} a(x,t,y) f(y) \, d\sigma_{x,t}(y) \end{equation}
is (for each fixed $t$) a Fourier integral operator of order $-1/2$. Rather than defining the cinematic curvature condition for Fourier integral operators, we will just use the characterisation from \cite[Theorem~2.1]{kung}, which says that if $\Phi_t(x,y) = y_2 - g(x,y_1,t)$ for some smooth function $g$, and if \eqref{phongstein} holds, then the Fourier integral operator in \eqref{averaging}, where $t$ is now a variable, satisfies the cinematic curvature condition if and only if
\begin{equation} \label{cinematic} \det\begin{pmatrix} g_{x_1} & g_{x_2} & g_t \\
 g_{x_1y_1} & g_{x_2y_1} & g_{ty_1} \\
  g_{x_1y_1y_1} & g_{x_2y_1y_1} & g_{ty_1y_1} \end{pmatrix} \neq 0. \end{equation}
The symbol $a$ does not affect the phase in the proof of Theorem~2.1 in \cite{kung} (see \cite[Example~2.14]{beltranhickmansogge}). Under the assumption that $\Phi_t(x,y) = y_2 - g(x,y_1,t)$, \eqref{phongstein} becomes  
\begin{equation} \label{phongstein2} g_{x_1} g_{x_2y_1} - g_{x_2} g_{x_1y_1} \neq 0. \end{equation}

In \cite[Example~2.3]{kung}, this characterisation is used to show that the circular averaging Fourier integral operators $f \mapsto (\sigma_t \ast f)(x)$ satisfy the cinematic curvature condition. 
  
In the case of sine waves, $g$ is the function 
  \[ g(x_1,x_2,y_1,t) = \frac{1}{\sqrt{2}}\left\langle (x_1,x_2,t), \left(\cos y_1, \sin y_1, 1 \right) \right\rangle. \]
  It is straightforward to check that
  \[ g_{x_1} g_{x_2y_1} - g_{x_2} g_{x_1y_1} =\frac{1}{\sqrt{2}}, \]
  which verifies \eqref{phongstein2}, and 
  \[ \det\begin{pmatrix} g_{x_1} & g_{x_2} & g_t \\
 g_{x_1y_1} & g_{x_2y_1} & g_{ty_1} \\
  g_{x_1y_1y_1} & g_{x_2y_1y_1} & g_{ty_1y_1} \end{pmatrix} = \frac{1}{2\sqrt{2}}, \]
 which verifies \eqref{cinematic}. Therefore, the averaging operator \eqref{averaging} arising from sine waves satisfies the cinematic curvature condition.

\begin{definition} \label{cinematicdefn} A family $\{\Sigma_{x,t}\}_{x,t}$ as above with $\Phi_t(x,y) = y_2 - g(x,y_1,t)$ is called a family of cinematic functions if \eqref{phongstein2} and \eqref{cinematic} both hold. A set $F \subseteq \mathbb{R}^2$ is called a cinematic $(u,v)$-Furstenberg set (with respect to the family of curves $\{\Sigma_{x,t}\}_{x,t}$) if there is a set $E \subseteq \mathbb{R}^2 \times [1,2]$ with $\dim E \geq v$, such that $\dim\left(\Sigma_{x,t} \cap F\right) \geq u$ for all $(x,t) \in E$.  \end{definition}

By the above discussion, the local smoothing inequality for Fourier integral operators of Gao, Liu, Miao, and Xi \cite[Theorem~1.4]{gaoliumiaoxi} (which generalises the Guth--Wang--Zhang local smoothing inequality for the wave equation) applies to the averaging operators \eqref{averaging} whenever $\{\Sigma_{x,t}\}_{x,t}$ is a family of cinematic functions. This can be used to generalise Theorem~\ref{theoremincidencecircle} and Theorem~\ref{theoremincidencecircleL2}, as follows. In the proof of Theorem~\ref{theoremincidencecircle}, if we replace the inequality at \eqref{referencedbelow}:
\[ \delta \lesssim \left(\sigma_t \ast \chi_T\right)(x), \]
with 
\[ \delta \lesssim \int \chi_T \, d\sigma_{x,t}, \]
for any $(x,t) \in B$, and then sum over those discs $T \in \mathbb{T}$ with $T \cap B \neq \emptyset$, integrate over $B$, and then sum over $B \in P$, we get
\[ I(P, \mathbb{T}) \lesssim \delta^{-4} \int_{\mathbb{R}^2 \times [1,2]} g \left( \int f\, d\sigma_{x,t}\right) \, dx \, dt \]
where 
\[ g = \sum_{B \in P} \chi_B, \qquad f = \sum_{T \in \mathbb{T}} \chi_T. \]
In the high case for the incidence bound analogous to Theorem~\ref{theoremincidencecircle}, we apply Hölder's inequality, and then the Gao--Liu--Miao--Xi local smoothing inequality for Fourier integral operators from \cite[Theorem~1.4]{gaoliumiaoxi} in place of the Guth--Wang--Zhang local smoothing inequality, which yields that
\begin{equation} \label{variablecoefficientaveraging} \left\lVert \int f\, d\sigma_{x,t} \right\rVert_{L^4_{x,t}(\mathbb{R}^2 \times [1,2] )} \leq C_{\epsilon} R^{\epsilon-\frac{1}{2}} \|f\|_4, \end{equation}
when $f$ is Fourier supported in $|\xi| \approx R$. The rest of the proof is similar to the circular case.

In the high case for the incidence bound analogous to Theorem~\ref{theoremincidencecircleL2}, we proceed as in the proof of Theorem~\ref{theoremincidencecircleL2}, except that at \eqref{latergeneralisation} we use the fixed time $L^2$ inequality for Fourier integral operators of order $-1/2$ (see, for example, \cite[Theorem~2.1]{beltranhickmansogge}) in place of Plancherel's theorem:
\[ \left\lVert \int f\, d\sigma_{x,t} \right\rVert_{L^2_{x}(\mathbb{R}^2)} \leq C_{\epsilon} R^{\epsilon-\frac{1}{2}} \|f\|_2, \]
when $f$ is Fourier supported in $|\xi| \approx R$. Again the rest of the proof is similar to the circular case. These two modifications yield the following.

\begin{theorem}  Let $\alpha \in [0, 2]$, $\beta \in [0,3]$, and $0 < \delta <1$. Let $\{\Sigma_{x,t}\}_{x,t}$ as above be a family of cinematic functions, where 
\[ \Sigma_{x,t} = \left\{ (y_1,g(x,y_1,t) ) : y_1 \in [0,1] \right\}, \]
and $g$ is smooth. Let $P$ be a $\left( \delta, \beta, K_{\beta,P,\delta}\right)$-set of $\delta$-neighbourhoods of curves $\Sigma_{x,t}$; meaning that the corresponding family of points $(x,t) \in \mathbb{R}^3$ form a $\left( \delta, \beta, K_{\beta,P,\delta}\right)$-set. Let $\mathbb{T}$ be a $\left( \delta, \alpha, K_{\alpha, \mathbb{T}, \delta}\right)$-set of $\delta$-discs in the plane. If $3\alpha + \beta \leq 7$, then for any $\epsilon >0$, 
\[ I(P, \mathbb{T}) \leq C_{\epsilon} \delta^{-\epsilon} \delta^{-3/4} K_{\beta,P,\delta}^{1/4} K_{\alpha,\mathbb{T},\delta}^{3/4}   |P|^{3/4} |\mathbb{T}|^{1/4} . \]
If $3\alpha+\beta > 7$, then with $\lambda = \frac{4}{3\alpha+\beta-3} < 1$, for any $\epsilon >0$,
\[ I(P, \mathbb{T}) \leq C_{\epsilon} \delta^{-\epsilon} \delta^{-3\lambda/4}   K_{\beta,P,\delta}^{\lambda/4} K_{\alpha,\mathbb{T},\delta}^{3\lambda/4}  |P|^{1-\frac{\lambda}{4}} |\mathbb{T}|^{1-\frac{3\lambda}{4}} . \]

If $\alpha + \beta \leq 4$, then for any $\epsilon >0$, 
\[ I(P, \mathbb{T}) \leq C_{\epsilon} \delta^{-\epsilon} \delta^{-1}   K_{\beta,P,\delta}^{1/2} K_{\alpha,\mathbb{T},\delta}^{1/2}  |P|^{1/2} |\mathbb{T}|^{1/2} .  \]
If $\alpha + \beta > 4$, then for any $\epsilon >0$, 
\[ I(P, \mathbb{T}) \leq C_{\epsilon} \delta^{-\epsilon} \delta^{-2\lambda}  K_{\beta,P,\delta}^{\lambda}  K_{\alpha,\mathbb{T},\delta}^{\lambda} |P|^{1-\lambda} |\mathbb{T}|^{1-\lambda}, \]
where $\lambda = \frac{1}{\alpha+\beta-2} < 1/2$.
\end{theorem} 

\section{Converting incidence bounds to lower bounds for curved Furstenberg sets} \label{conversion}\begin{sloppypar} The following proposition is a straightforward adaptation of \cite[Proposition~3.15]{orponennotes}, which converts incidence bounds into lower bounds for the dimension of Furstenberg sets. \end{sloppypar}
\begin{proposition} \label{orponenproposition} Let $K_1, K_2 >0$. Suppose that $f: [0,2] \times [0,3] \to [0, \infty)$ is a continuous function with the property that
\[ I(P, \mathbb{T}) \leq C_{\alpha,\beta,K_1, K_2} \delta^{-f(\alpha, \beta) }, \]
for any $\delta \in (0,1)$, $\alpha \in [0,2]$ and $\beta \in [0,3]$ such that $P$ is $\left( \delta, \beta, K_1\right)$-set of $\delta$-neighbourhoods of circles in the plane of radii between 1 and 2, and $\mathbb{T}$ is a $( \delta, \alpha, K_2)$-set of $\delta$-discs in the plane. Then, for any $0 < u \leq 1$ and $0 < v \leq 3$, any circular $(u,v)$-Furstenberg set $F \subseteq \mathbb{R}^2$   satisfies 
\begin{equation} \label{finequality} f( \dim F, v) \geq u+v. \end{equation}
More generally, this holds if the family of circles is replaced by any family of cinematic curves.\end{proposition}
For concreteness we will prove this only for circles; the generalisation to cinematic curves \eqref{bound1'} will be evident. The proof is taken almost verbatim from \cite[Proposition~3.15]{orponennotes} as we just replace lines by circles, but we will include it here for completeness.
\begin{proof}[Proof of Proposition~\ref{orponenproposition}]
We will instead prove that for any $\epsilon > 0$ and any circular $(u+\epsilon, v+\epsilon)$-Furstenberg set $F \subset \R^2$, we have \eqref{finequality}. This version formally implies Proposition~\ref{orponenproposition} by continuity of $f$.

In this proof, we associate a circle $C(x, t)$ with the pair $(x, t) \in \R^2 \times (0, \infty)$ without further comment. Fix a circular $(u+\epsilon, v+\epsilon)$-Furstenberg set $F \subset \R^2$. By scaling and by countable subadditivity of the Hausdorff content, we can assume the circles defining the Furstenberg property for $F$ have radii between 1 and 2.
Therefore, following the proof of Lemma 3.3 in \cite{hera}, we can find a constant $c > 0$ and a family $E$ of circles of radii between 1 and 2 with $\mathcal{H}^v_{\infty}(E) \geq c$, such that $\mathcal{H}^u_{\infty}(F \cap C) \geq c$ for all $C \in E$. We will assume $c = 1$ for convenience; the exact value of $c$ is irrelevant and only affects the values of some constants.

Fix $\alpha > \dim F$. By \cite[Lemma~19]{gan}, for any fixed large integer $k_0$, we can find a sequence $\{\mathbb{T}_k\}_{k \geq k_0}$ of sets of dyadic cubes, such that $\bigcup_{k \geq k_0} \mathbb{T}_k$ is a cover of $F$, and where each $\mathbb{T}_k$ consists of disjoint (semi-open) dyadic cubes of side length $2^{-k}$, with at most $\left( \frac{2^{-l}}{2^{-k} } \right)^{\alpha}$ cubes inside any dyadic cube of side length $2^{-l}$, for any $l < k$. By the covering property, $\mathcal{H}^u_{\infty}\left( \bigcup_{k \geq k_0} \bigcup_{D \in \mathbb{T}_k} D \cap C \right) \geq 1$ for every $C \in E$, so by subadditivity of the Hausdorff content, for any $C \in E$, there exists a positive integer $k=k(C) \geq k_0$ with $\mathcal{H}^u_{\infty}\left(  \bigcup_{D \in \mathbb{T}_k} D \cap C \right) \gtrsim k^{-2}$. By subadditivity of the Hausdorff content again, we can find a fixed $k \geq k_0$ and a set $E' \subseteq E$ with $\mathcal{H}^v_{\infty}(E') \gtrsim k^{-2}$, such that $k=k(C)$ for every $C \in E'$. Since $\mathcal{H}^v_{\infty}(E') \gtrsim k^{-2}$, by \cite[Lemma~2.13]{fasslerorponen} (or Lemma~3.13 of the arXiv version of \cite{fasslerorponen}), we can find a finite set $E'' \subseteq E'$, which is a $\left( 2^{-k}, v, C'\right)$-set for some absolute constant $C'$, and $|E''| \gtrsim k^{-2} 2^{kv}$. This set satisfies 
\[ \mathcal{H}^u_{\infty}\left( \bigcup_{D \in \mathbb{T}_k} D \cap C \right) \gtrsim k^{-2}, \]
for all $C \in E''$. This implies that 
\[ I\left( C, \mathbb{T}_k\right) \gtrsim 2^{ku} k^{-2}, \]
for all $C \in E''$. Summing over $C \in E''$ yields 
\[ I\left(E'',\mathbb{T}_k \right) \gtrsim k^{-4} 2^{k(u+v)}. \]
By assumption
\[ I\left( E'', \mathbb{T}_k\right) \lesssim 2^{k f(\alpha, v)}, \]
 Hence 
\[ k^{-4} 2^{k(u+v)} \lesssim 2^{kf(\alpha,v)}. \]
Since $k \geq k_0$, and $k_0$ can be chosen arbitrarily large in the above argument (and the implicit constants are independent of $k_0$), this yields that 
\[ u+v \leq f(\alpha, v). \]
By continuity of $f$, and since this holds for any $\alpha > \dim F$, it follows that
\[ f(\dim F, v) \geq u+v. \qedhere \]
\end{proof}
 \begin{sloppypar}
\begin{proof}[Proof of \eqref{bound1}, \eqref{bound1'}, and \eqref{bound2}] Again we consider only the case of circles, for concreteness. By Theorem~\ref{theoremincidencecircle}, if $P$ is a $\left( \delta, \beta, K_{\beta,P,\delta}\right)$-set of $\delta$-neighbourhoods of circles in the plane, of radii between 1 and 2, and if $\mathbb{T}$ is a $\left( \delta, \alpha, K_{\alpha, \mathbb{T}, \delta}\right)$-set of $\delta$-discs in the plane, then for any $\epsilon >0$,
\begin{equation} \label{incidencecount} I(P, \mathbb{T}) \leq C_{\epsilon} \delta^{-f(\alpha,\beta) } \end{equation}
where 
\[ f(\alpha,\beta)= \epsilon+ \max\left\{  \frac{\alpha+3\beta+3}{4}, 
\alpha+\beta-1 \right\}. \]
Hence, by Proposition~\ref{orponenproposition},
\begin{equation} \label{reusedbelow} f(\dim F, v) \geq u+v. \end{equation}
It follows (by letting $\epsilon \to 0$) that
\[ \max\left\{\frac{ \dim F + 3v +3}{4}, \dim F + v -1 \right\} \geq u+v.\]
Hence
\[ \dim F \geq \min\left\{ 4u+v-3, u+1\right\}, \]
which is \eqref{bound1}. The lower bound \eqref{bound2} is similar; using Theorem~\ref{theoremincidencecircleL2} to get \eqref{incidencecount} with 
\[ f(\alpha,\beta) = \epsilon+ \max\left\{1 + \frac{\alpha+\beta}{2}, 
\alpha+\beta-1 \right\}. \]
Proposition~\ref{orponenproposition} then yields \eqref{reusedbelow}, which gives (letting $\epsilon \to 0$)
\[ \max\left\{1 + \frac{\dim F+v}{2}, \dim F + v -1 \right\} \geq u+v,\]
which implies that 
\[ \dim F \geq \min\left\{2u+v-2, u+1\right\}, \]
which is \eqref{bound2}. 
\end{proof} \end{sloppypar}


The following is a broad version of Proposition~\ref{orponenproposition}. 
 \begin{proposition} \label{orponenproposition2} \begin{enumerate}
     \item
 Let $K_1, K_2 \geq 1$. Suppose that $f: [0,2] \times [0,3] \to [0, \infty)$ is a continuous function with the property that for any $\kappa>0$,
\[ I_{\kappa}(P, \mathbb{T}) \leq C_{\alpha,\beta,K_1,K_2} \cdot \kappa^{-O(1)} \delta^{-f(\alpha, \beta) }, \]
for any $\delta \in (0,1)$, $\alpha \in [0,2]$ and $\beta \in [0,3]$ such that $P$ is a $\left( \delta, \beta, K_1\right)$-set of $\delta$-neighbourhoods of sine waves in $[0, 2\pi] \times \mathbb{R}$, and $\mathbb{T}$ is a $( \delta, \alpha, K_2)$-set of $\delta$-discs in the plane. Then, for any $0 < u \leq 1$ and $0 \leq v \leq 3$, any sine wave $(u,v)$-Furstenberg set $F \subseteq \mathbb{R}^2$   satisfies 
\begin{equation} \label{finequality2} f( \dim F, v) \geq u+v. \end{equation}

\item Let $K_1, K_2 \ge 1$, and define $S = \left\{ (\alpha, \beta) \in [0,2] \times [0,3] : 3\alpha+2\beta \leq 9 \right\}$. Suppose that $g: S \to [0, \infty)$ is a continuous function with the property that for any $\delta >0$, $A \geq |\log \delta|$ and $\K \geq 1$,
\[ I_{A,\K}^{\broad}(P, \mathbb{T}) \leq C_{\alpha,\beta,K_1,K_2} \K^{O(1)} \delta^{-g(\alpha, \beta)}, \]
whenever $P$ is a $\left( \delta, \beta, K_1\right)$-set of $\delta$-neighbourhoods of circles in $\mathbb{R}^2$, of radii between 1 and 2, and $\mathbb{T}$ is a $( \delta, \alpha, K_2)$-set of $\delta$-discs in the plane. Then, for $0 < u \leq 1$ and $0 \leq v \leq 3$, any circular $(u,v)$-Furstenberg set $F \subseteq \mathbb{R}^2$ with $3 \dim F + 2v < 9$  satisfies 
\begin{equation} \label{finequality4} g( \dim F, v) \geq u+v. \end{equation}
This holds similarly if $P$ is $\left( \delta, \beta, K_1\right)$-set of $\delta$-neighbourhoods of sine waves in $[0, 2\pi] \times \mathbb{R}$.
\end{enumerate}
\end{proposition}
\begin{proof} By continuity of $f$, for the first part it suffices to obtain \eqref{finequality2} with $u$ and $v$ replaced by $u-\epsilon$ and $v-\epsilon$ for arbitrarily small $\epsilon$. Therefore, we can find a family $E$ of sine waves (identified with a subset of $\mathbb{R}^3$) with $\mathcal{H}^v_{\infty}(E) \geq 1$, such that $\mathcal{H}^u_{\infty}(F \cap \Gamma) \geq 1$ for all $\Gamma \in E$. Fix $\alpha > \dim F$. By \cite[Lemma~19]{gan}, for any fixed large integer $k_0$, we can find a sequence $\{\mathbb{T}_k\}_{k \geq k_0}$ of sets of dyadic cubes, such that $\bigcup_{k \geq k_0} \mathbb{T}_k$ is a cover of $F$, and where each $\mathbb{T}_k$ consists of disjoint dyadic cubes of side length $2^{-k}$, with at most $\left( \frac{2^{-l}}{2^{-k} } \right)^{\alpha}$ cubes inside any dyadic cube of side length $2^{-l}$, for any $l < k$. By the covering property, $\mathcal{H}^u_{\infty}\left( \bigcup_{k \geq k_0} \bigcup_{D \in \mathbb{T}_k} D \cap \Gamma \right) \geq 1$ for every $\Gamma \in E$, so by subadditivity of the Hausdorff content, for any $\Gamma \in E$, there exists a positive integer $k=k(\Gamma) \geq k_0$ with $\mathcal{H}^u_{\infty}\left(  \bigcup_{D \in \mathbb{T}_k} D \cap \Gamma \right) \gtrsim k^{-2}$. By subadditivity of the Hausdorff content again, we can find a fixed $k \geq k_0$ and a set $E' \subseteq E$ with $\mathcal{H}^v_{\infty}(E') \gtrsim k^{-2}$, such that $k=k(\Gamma)$ for every $\Gamma \in E'$. Since $\mathcal{H}^v_{\infty}(E') \gtrsim k^{-2}$, by \cite[Lemma~2.13]{fasslerorponen} (or Lemma~3.13 of the arXiv version of \cite{fasslerorponen}), we can find a finite set $E'' \subseteq E'$, which is a $\left( 2^{-k}, v, C'\right)$-set for some absolute constant $C'$, and $|E''| \gtrsim k^{-2} 2^{kv}$. This set satisfies 
\[ \mathcal{H}^u_{\infty}\left( \bigcup_{D \in \mathbb{T}_k} D \cap \Gamma \right) \gtrsim k^{-2}, \]
for all $\Gamma \in E''$. Therefore, by subadditivity of the Hausdorff content, for any $\Gamma \in E''$, we can find three subsets $\Gamma(1), \Gamma(2), \Gamma(3) \subseteq \Gamma$, whose projections down to the $x$-axis are separated by $\gtrsim k^{-2/u}$ from each other, such that for all $j \in \{1,2,3\}$
\[ \mathcal{H}^u_{\infty}\left( \bigcup_{D \in \mathbb{T}_k} D \cap \Gamma(j) \right) \gtrsim k^{-2}; \]
this is where we use the assumption $u >0$. If we now choose $\kappa \ll k^{-2/u}$, this implies 
\[ I_{\kappa}\left( \Gamma, \mathbb{T}_k\right) \gtrsim k^{-2}2^{ku} , \]
for all $\Gamma \in E''$. Summing over $\Gamma \in E''$, using the definition of $I_{\kappa}(E'', \mathbb{T}_k)$, yields 
\[ I_{\kappa}\left(E'',\mathbb{T}_k \right) \gtrsim k^{-4} 2^{k(u+v)}. \]
However, by assumption
\[ I_{\kappa}\left( E'', \mathbb{T}_k\right) \lesssim \kappa^{-O(1)} 2^{k f(\alpha, v)}, \]
Hence, since we can take $\kappa \sim k^{-2/u}$ in the above, we get 
\[ k^{-O(1)} 2^{k(u+v)} \lesssim 2^{kf(\alpha,v)}. \]
Since $k \geq k_0$, and $k_0$ can be chosen arbitrarily large in the above argument (and the implicit constants are independent of $k_0$), this yields that 
\[ u+v \leq f(\alpha, v). \]
By continuity of $f$, and since this holds for any $\alpha > \dim F$, it follows that 
\[ f(\dim F, v) \geq u+v.  \]
This proves \eqref{finequality2}.

 To prove \eqref{finequality4}, we only consider circles since the case of sine waves is similar. By scaling and by countable subadditivity of the Hausdorff content, we can assume that the radii of the circles lie between 1 and 2. By the continuity of $g$, it suffices to obtain \eqref{finequality4} with $u$ and $v$ replaced by $u-\epsilon$ and $v-\epsilon$ for arbitrarily small $\epsilon$. Therefore, we can find a family $E$ circles (identified with a subset of $\mathbb{R}^2 \times (0, \infty)$) with $\mathcal{H}^v_{\infty}(E) \geq 1$, such that $\mathcal{H}^u_{\infty}(F \cap C) \geq 1$ for all $C \in E$.  Fix $\alpha > \dim F$ sufficiently close to $\dim F$ so that $3\alpha + 2 v < 9$. By \cite[Lemma~19]{gan}, for any fixed large integer $k_0$, we can find a sequence $\{\mathbb{T}_k\}_{k \geq k_0}$ of sets of dyadic cubes, such that $\bigcup_{k \geq k_0} \mathbb{T}_k$ is a cover of $F$, and where each $\mathbb{T}_k$ consists of disjoint dyadic cubes of side length $2^{-k}$, with at most $\left( \frac{2^{-l}}{2^{-k} } \right)^{\alpha}$ cubes inside any dyadic cube of side length $2^{-l}$, for any $l < k$. By the covering property, $\mathcal{H}^u_{\infty}\left( \bigcup_{k \geq k_0} \bigcup_{D \in \mathbb{T}_k} D \cap C \right) \geq 1$ for every $C \in E$, so by subadditivity of the Hausdorff content, for any $C \in E$, there exists a positive integer $k=k(C) \geq k_0$ with $\mathcal{H}^u_{\infty}\left(  \bigcup_{D \in \mathbb{T}_k} D \cap C \right) \gtrsim k^{-2}$. By subadditivity of the Hausdorff content again, we can find a fixed $k \geq k_0$ and a set $E' \subseteq E$ with $\mathcal{H}^v_{\infty}(E') \gtrsim k^{-2}$, such that $k=k(C)$ for every $C \in E'$. Since $\mathcal{H}^v_{\infty}(E') \gtrsim k^{-2}$, by \cite[Lemma~2.13]{fasslerorponen} (or Lemma~3.13 of the arXiv version of \cite{fasslerorponen}), we can find a finite set $E'' \subseteq E'$, which is a $\left( 2^{-k}, v, C'\right)$-set for some absolute constant $C'$, and $|E''| \gtrsim k^{-2} 2^{kv}$. This set satisfies 
\[ \mathcal{H}^u_{\infty}\left( \bigcup_{D \in \mathbb{T}_k} D \cap C \right) \gtrsim k^{-2}, \]
for all $C \in E''$. Let $\varepsilon >0$ and $K = 2^{k \varepsilon}$, and choose $A \geq |\log(2^k) |$ with $A \sim \log(2^k)$. By subadditivity of the Hausdorff content again, we can (for $k$ sufficiently large, depending on $\varepsilon$ and $u$) find at least $100A$ arcs $\{C_j\}_j$ of $C$ of diameter $\K^{-1}$, separated by $\K^{-1}$ from each other, such that for all $1 \leq j \leq 100A$,
\[ \mathcal{H}^u_{\infty}\left( \bigcup_{D \in \mathbb{T}_k} D \cap C_j \right) \gtrsim k^{-2} \K^{-1}. \]
This step is where we use the assumption that $u>0$, and we use that $\K^{-u} A \ll k^{-2}$ if $k$ is sufficiently large. This yields 
\[ I_{A,\K}^{\broad}(C, \mathbb{T}_k) \gtrsim 2^{ku} k^{-2}\K^{-1},\]
for all $C \in E''$. Summing over $C \in E''$ gives 
\[ I_{A,\K}^{\broad}(E'', \mathbb{T}_k) \gtrsim 2^{k(u+v)} k^{-4}\K^{-1}. \] 
However, by assumption
\[ I_{A,\K}^{\broad}\left( E'', \mathbb{T}_k\right) \lesssim  C_{\alpha,v} \K^{O(1)}  2^{k g(\alpha, v)}. \]
 Hence 
\[ \K^{-O(1)}k^{-4} 2^{k(u+v)} \lesssim 2^{kg(\alpha,v)}. \]
Since $k \geq k_0$, and $k_0$ can be chosen arbitrarily large in the above argument (and the implicit constants are independent of $k_0$), this yields that 
\[ u+v \leq g(\alpha, v)+O(\varepsilon). \]
By continuity of $g$, and since this holds for any $\alpha > \dim F$ sufficiently close to $\dim F$, and all $\varepsilon >0$, it follows that 
\[ g(\dim F, v) \geq u+v.  \qedhere \]
\end{proof}

We now combine the previous proposition with the incidence bounds from the previous section to prove the second half of the main theorem.  

\begin{sloppypar} \begin{proof}[Proof of \eqref{sine1} for sine waves, and \eqref{sine2} for circles and sine waves] If $P$ is a $\left( \delta, \beta,L\right)$-set of $\delta$-neighbourhoods of sine waves in $[0, 2\pi] \times \mathbb{R}$, and if $\mathbb{T}$ is a $\left( \delta, \alpha,L\right)$-set of $\delta$-discs in the plane, then for any $\epsilon >0$ and $\kappa>0$,
\[ I_{\kappa}(P, \mathbb{T}) \leq C_{\epsilon,L} \kappa^{-O(1)} \delta^{-f(\alpha,\beta) }, \]
by Theorem~\ref{trilinearincidence}, where 
\[ f(\alpha,\beta)= \alpha + \frac{2\beta}{3}+\epsilon. \]
Hence, by Proposition~\ref{orponenproposition2}, for any sine wave $(u,v)$-Furstenberg set $F \subseteq \mathbb{R}^2$
\begin{equation} \label{furst1} f(\dim F, v) \geq u+v.\end{equation}
Eq.~\eqref{furst1} implies (by letting $\epsilon \to 0$) that $\dim F + \frac{2v}{3} \geq u+v$, which simplifies to $\dim F \geq u + \frac{v}{3}$,
which is \eqref{sine1}.   

To prove \eqref{sine2} for circles, by Theorem~\ref{trilinearincidencecircle}, if $P$ is a $\left( \delta, \beta, K_{\beta,P,\delta}\right)$-set of $\delta$-neighbourhoods of circles in $\mathbb{R}^2$ of radii between 1 and 2, and if $\mathbb{T}$ is a $\left( \delta, \alpha, K_{\alpha, \mathbb{T}, \delta}\right)$-set of $\delta$-discs in the plane, then for any $\epsilon >0$, $A \geq |\log \delta|$ and $K \geq 1$, if $3\alpha+2\beta \leq 9$,
\[ I_{A,\K}^{\broad}(P, \mathbb{T}) \leq  C_{\epsilon} \K^{100} \delta^{- g(\alpha,\beta)}, \]
where
\[ g(\alpha,\beta) = \frac{1}{2} + \frac{\alpha}{2} + \frac{2\beta}{3}+\epsilon. \]
Hence, by Proposition~\ref{orponenproposition2}, for any circular $(u,v)$-Furstenberg set $F \subseteq \mathbb{R}^2$ with $3\dim F + 2v < 9$,
 \begin{equation} \label{furst3} g(\dim F, v) \geq u+v ,\end{equation}
If $3\dim F + 2v < 9$, then Eq.~\ref{furst3} gives $\frac{1}{2} + \frac{\dim F}{2} + \frac{2v}{3} \geq u+v$, which simplifies to $\dim F \geq 2u + \frac{2v}{3} -1$. This lower bound equals $u+1$ when $3(u+1) + 2v = 9$, so in the case where $3\dim F + 2v \geq 9$, if (for a contradiction) the lower bound $\dim F \geq u+1$ fails, then we have $3(u+1)+2v >9$. Therefore, for arbitrarily small $\epsilon$, we can choose $v' \in [0, v)$ such that $3(u+1) + 2v'= 9-\epsilon$. By applying the result with $v'$ instead of $v$ we get $\dim F \geq 2u + \frac{2v'}{3}-1 = u+1-O(\epsilon)$, and by letting $\epsilon \to 0$ we get $\dim F \geq u+1$ (a contradiction). Therefore, since we have shown that either $\dim F \geq 2u + \frac{2v}{3} -1$ or $\dim F \geq u+1$, this gives $\dim F \geq \min\left\{ 2u + \frac{2v}{3}-1, u+1\right\}$ which is \eqref{sine2}. The case of sine waves is similar. \end{proof} \end{sloppypar}

\section{Positive area of sets containing many cinematic curves} \label{positiveareasection}
  
In this section we give the proof of Theorem~\ref{positivearea}. As mentioned in the introduction, the proof is similar to the proof of Corollary~3 in \cite{wolff}; replacing decoupling (or local smoothing) with variable coefficient local smoothing. We refer to Definition~\ref{cinematicdefn} for the definition of cinematic curves.  
\begin{proof}[Proof of Theorem~\ref{positivearea}] Let $\left\{ \Sigma_{x,t} \right\}_{x,t}$ be given, let $E \subseteq \mathbb{R}^2 \times [1,2]$ have $\dim E > 1$, and let $F \subseteq \mathbb{R}^2$ be such that $ \mathcal{H}^1\left( \Sigma_{x,t} \cap F \right) >0$ for all $(x,t) \in E$.  By subadditivity of outer measures, there exists $N, \tau >0$ and a subset $E' \subseteq E$ with $\dim E' \geq 1 +\tau$, such that 
\[ \mathcal{H}^1\left( \Sigma_{x,t} \cap F \right) > \frac{1}{N} \qquad \forall \, (x,t) \in E'. \]
We will treat $N$ as constant, as it will not be important below. Suppose for a contradiction that $\mathcal{H}^2(F) =0$. Then for arbitrarily large $K$ (to be chosen later), there is a covering $\{D_j(x_j, r_j)\}_j$ of $F$ by discs of dyadic radii $r_j$ smaller than $2^{-K}$, such that 
\[ \sum_j r_j^2 \leq 2^{-K}. \]
Let $U = \bigcup_j D(x_j, r_j)$ which is an open set of area $\mathcal{H}^2(U) \leq \pi 2^{-K}$. For each $(x,t) \in E'$, 
\begin{equation} \label{Uintegral} \int \chi_U \, d\sigma_{x,t} \gtrsim 1, \end{equation}
where $\sigma_{x,t}$ is the pushforward of Lebesgue measure on $\mathbb{R}$ multiplied by a smooth cutoff $\psi$, under $y \mapsto g(x,y_1,t)$. Let $\eta$ be a radial Schwartz function with $\widehat{\eta}$ supported in $B_2(0,2) \setminus B_2(0,1/2)$, such that 
\[ \sum_{k \in \mathbb{Z}} \widehat{\eta}(\xi/2^k)  =1,  \qquad \xi \in \mathbb{R}^2 \setminus \{0\}. \]
Define $\eta_k$ by $\widehat{\eta_k}(\xi) = \widehat{\eta}(\xi/2^k)$. By \eqref{Uintegral},
\begin{equation} \label{Uintegral2} \sum_{k \in \mathbb{Z}} \left\lvert \int \left(\chi_U \ast \eta_k \right) \, d\sigma_{x,t}\right\rvert \gtrsim 1, \end{equation}
for each $(x,t) \in E'$. By Young's convolution inequality
\[ \|\chi_U \ast \eta_k \|_{\infty} \lesssim \mathcal{H}^2(U) 2^{2k} \lesssim 2^{2k-K}, \]
so by \eqref{Uintegral2},  
\[ \sum_{k \geq K/2} \left\lvert \int \left(\chi_U \ast \eta_k \right) \, d\sigma_{x,t}\right\rvert \gtrsim 1; \]
for $(x,t) \in E'$, where we now choose $K$ so that $2^{-K/2}$ is much smaller than the implicit constant in \eqref{Uintegral2}. By the above, for each $(x,t) \in E'$, there exists $k= k(x,t) \geq K/2$ such that 
\begin{equation} \label{fixedk} \left\lvert \int \left(\chi_U \ast \eta_k \right) \, d\sigma_{x,t}\right\rvert \gtrsim \frac{1}{k^2}. \end{equation}
For each $k$, let $E_k$ be the set of points in $E'$ such that \eqref{fixedk} holds. 

In \cite[Eq.~(79)]{wolff}, Wolff used the uncertainty principle to conclude that \eqref{fixedk}  continues to hold in a ball of radius $\sim k^{-2}2^{-k}$ around $(x,t)$ whenever it holds for $(x,t)$. However, for general curves $\Sigma_{x,t}$ (not necessarily circles), the Fourier supports of $\chi_U \ast \eta_k$ being contained in a ball of radius $2^k$ does not seem to imply anything about the Fourier supports of the functions 
\[ G(x,t) := \int \left(\chi_U \ast \eta_k \right) \, d\sigma_{x,t}. \] Because of this, we give a direct argument which avoids the uncertainty principle, but is based on the proof of Bernstein's inequality (which is usually used to prove the uncertainty principle; see \cite[Proposition~5.3]{wolff2}). By the definition of $\sigma_{x,t}$ and the chain rule, 
\begin{align*} \partial_{x_1} G(x,t) &=  \partial_{x_1} \int_{\mathbb{R}} \int_U \eta_k\left( \left( \theta, g(x, \theta, t) \right) - y \right) \psi(\theta) \, dy \, d\theta \\
&= \int_{\mathbb{R}}\int_U \left( \partial_{x_1}g\right)(x, \theta, t) \partial_2  \eta_k\left( \left( \theta, g(x, \theta, t) \right) - y \right) \psi(\theta) \, dy \, d\theta. \end{align*}
Since $\left\lVert \nabla \partial \eta_k \right\rVert_{\infty} \lesssim 2^{3k}$, and since $\partial_2 \eta_k$ is rapidly decaying outside the ball around the origin of radius $2^{-k}$, this gives (for any $(x,t)$)
\[ \left\lvert \partial_{x_1} G(x,t) \right\rvert \lesssim 2^k. \]
By a similar argument for the other derivatives of $G$, this gives
\[ \left\lVert \nabla G \right\rVert_{\infty} \lesssim 2^{k}.\]
By the mean value theorem, it follows that if \eqref{fixedk} holds for some $(x,t) = (x_0,t_0)$, then \eqref{fixedk} continues to hold for all $(x,t)$ in ball of radius $\sim 2^{-k}k^{-2}$ around $(x_0,t_0)$.

Fix $k$, let $\{(x_{l,k},t_{l,k})\}_l$ be a maximal $2^{-k}$-separated subset of $E_k$, let $\{\mathbb{D}_k\}$ be the corresponding set of balls of radius $2^{-k}$ centred at points $(x_{l,k},t_{l,k})$, and let $V_k=  \bigcup_{D \in \mathbb{D}_k} D$. Since \eqref{fixedk} continues to hold in any ball of radius $\sim k^{-2} 2^{-k}$ around $(x_{l,k}, t_{l,k})$, 
\[ \mathcal{H}^3\left(V_k\right) \lesssim k^{8} \int \chi_{V_k} \left\lvert \int  \left(\chi_U \ast \eta_k\right) \, d\sigma_{x,t} \right\rvert \, dx \, dt. \]
By Hölder's inequality applied to the above, followed by the $L^4$ variable coefficient local smoothing inequality \eqref{variablecoefficientaveraging}, this gives
\[ \mathcal{H}^3(V_k) \lesssim  \mathcal{H}^3(V_k)^{3/4} 2^{k\left( \epsilon - \frac{1}{2} \right) }, \]
where we choose $\epsilon \ll \tau$. This simplifies to
\[ \mathcal{H}^3(V_k) \lesssim 2^{-2k + 4k\epsilon)}. \]
It follows that 
\[ |\mathbb{D}_k| \lesssim 2^{k(1+4\epsilon)}. \]
Therefore, if we let $\mathbb{D} = \bigcup_{k \geq K/2} \mathbb{D}_k$, we get a covering $\{D\}_{D \in \mathbb{D}_k}$ of $E'$ by discs, such that 
\begin{align*} \sum_{D \in \mathbb{D}} r(D)^{1+\frac{\tau}{2} } &= \sum_{k \geq K/2} \sum_{D \in \mathbb{D}_k} 2^{-k\left(1+\frac{\tau}{2} \right)} \\
&\lesssim \sum_{k \geq K/2} 2^{k(1+4\epsilon)} 2^{-k\left(1+\frac{\tau}{2} \right)} \\
&\lesssim 2^{-K(\tau/4 - 2\epsilon)} \\
&\lesssim 2^{-K\tau/8}, \end{align*}
since $\epsilon \ll \tau$. By taking $K$ arbitrarily large, this shows that $\dim E' \leq 1+\frac{\tau}{2}$, which contradicts the property $\dim E' \geq 1+\tau$ from the construction of $E'$. 
\end{proof}

\section{Sharpness of Proposition~\ref{projection} when \texorpdfstring{$s = t/3$}{s=t/3}} \label{projectionsection}
By using the rational parametrisation of the unit circle 
\[\left( \frac{1-t^2}{1+t^2}, \frac{2t}{1+t^2} \right), \] scaling, and changing variables, to show that Proposition~\ref{projection} is sharp (in the sense that $\dim A/3$ cannot be increased while still expecting a zero dimensional exceptional set), we can replace the projections $\rho_{\theta}$ by the maps $\pi_{\theta}: \mathbb{R}^3 \to \mathbb{R}$, given by 
\[ \pi_{\theta}(x_1, x_2, x_3) = \left\langle \left(x_1,x_2,x_3\right),  \left(1,\theta,\theta^2 \right) \right\rangle, \quad \theta \in \mathbb{R}. \]
Since Conjecture~\ref{furstenbergconjecture} only literally implies Conjecture~\ref{oberlinR3} in the discretised setting, We will give both discrete and continuous counterexamples. The following proposition is the discrete counterexample, which is an adaption of the counterexample to the linear Furstenberg problem from \cite[p.~100]{wolff} (essentially the Szemerédi–Trotter example).
\begin{proposition} Let $t \in (0,3)$ and $s \in (t/3,1]$. There exists an absolute constant $C$ such that, for arbitrarily small $\delta>0$, there exists a $\left( \delta, t, C\right)$-set $A \subseteq B_3(0,1)$ and, for any sufficiently small $\epsilon >0$, there exists a $\left( \delta, s-\frac{t}{3} - 2\epsilon , C_{\epsilon,s,t}\right)$-set $\Theta \subseteq (0,1)$, such that the number $\left\lvert \pi_{\theta}(A) \right\rvert_{\delta}$ of $\delta$-balls required to cover $\pi_{\theta}(A)$ satisfies
\begin{equation} \label{coveringnumber} \left\lvert \pi_{\theta}(A) \right\rvert_{\delta} \leq C \delta^{-s+\epsilon} \quad \forall \, \theta \in \Theta. \end{equation}
\end{proposition} 
\begin{proof} The version of Jarnik's theorem from \cite[p.~99]{wolff} says that if $\alpha \in (0,1)$ and $\{n_j\}_{j=1}^{\infty}$ is a sequence of integers that increases sufficiently rapidly to $\infty$, then the set 
\[ T = \left\{ \theta \in (1/4, 3/4) : \forall j \, \exists p_j,q_j \in \mathbb{N} \text{ with } |q_j| \leq n_j^{\alpha} \text{ and } \left\lvert  \theta- \frac{p_j}{q_j} \right\rvert \leq \frac{1}{n_j^2} \right\},  \]
has 
\begin{equation} \label{jarnik} \dim T = \alpha. \end{equation}Let $j$ be a large positive integer, let $n = n_j^{2t/3}$, and let
\[ A = \left\{ \left( \frac{k}{n}, \frac{l}{n}, \frac{m}{n} \right) : 0 \leq k, l , m \leq n \right\}. \]
Let $\epsilon >0$ be small, and define $\alpha$ by $s = \frac{ t+3\alpha}{3}+\epsilon$. Then $\alpha \in (0,1)$ since $t/3 < s \leq 1$ and $t \geq 0$. Let $\delta = n^{-3/t}$, so that $A$ is a $\left(\delta, t,C\right)$-set for some sufficiently large constant $C$. For any $\epsilon >0$, by \eqref{jarnik} and \cite[Lemma~3.13]{fasslerorponen}, there is a $(\delta, \alpha, C_{\epsilon})$ set $\Theta \subseteq (1/4,3/4)$ with $\left\lvert \Theta \right\rvert \sim \delta^{\epsilon-\alpha}$, such that every $\theta \in \Theta$ is a rational number $p/q$ in lowest form, where $p,q$ are positive integers and $q \leq n^{3\alpha/2t}$. For any $\theta = \frac{p}{q} \in \Theta$ in lowest form and $x = \left( \frac{k}{n}, \frac{l}{n}, \frac{m}{n} \right) \in A$,
\[ \pi_{\theta}(x) = \frac{ kq^2 + lpq + mp^2}{nq^2}. \]
Since $q \leq n^{3\alpha/2t}$, the denominator is a positive integer of size at most $n^{1+ \frac{3\alpha}{t}}$, and the numerator is an integer of size $\lesssim 1$ times the denominator, so the projection $\pi_{\theta}(A)$ can be covered by $\lesssim n^{1+\frac{3\alpha}{t}} = \delta^{-\frac{ t+3\alpha}{3}} = \delta^{-s+\epsilon}$ many $\delta$-balls. This proves \eqref{coveringnumber}.\end{proof} 

\begin{proposition} For any $t \in (0,3]$ and $s \in \left(0, \min\{1,t\}\right]$, there is a Borel set $A \subseteq B_3(0,1)$ with $\dim A = t$ such that
\begin{equation} \label{exceptionalsetbound} \dim \left\{ \theta \in [0, 2\pi) : \dim \pi_{\theta}(A) < s  \right\} \geq s-\frac{t}{3}. \end{equation} 
\end{proposition} \begin{sloppypar}
\begin{proof} Let $\epsilon>0$ be very small, and define $\kappa \in (0,1)$ 
by
\begin{equation} \label{kappadefn} t-\frac{2\kappa t}{3} = s - \epsilon. \end{equation}
By the density of $\mathbb{Q}$ in $\mathbb{R}$, we can take $\epsilon$ arbitrarily small above that makes $\kappa$ rational, so below we will assume that $\kappa$ is rational. Let $\{n_j\}$ be a strictly increasing sequence of positive integers, such that $n_j^{1-\kappa}$ is a positive integer for each $j$ (this is possible since $\kappa$ is rational) with $n_{j+1} \geq n_j^j$ for all $j$ (which implies that $n_{j+1}^{1-\kappa} \geq (n_j^{1-\kappa})^j$ for all $j$). Let  
\[ A = \left\{ (x_1,x_2,x_3) \in B_3(0,1) : \forall i,j, \exists b_{i,j} \in \mathbb{Z} \text{ with } \left\lvert x_i - \frac{b_{i,j}}{n_j} \right\rvert \leq n_j^{-3/t}. \right\} \] 
Then $\dim A = t$ by Theorem~8.15b in \cite[p.~132]{falconer} (more precisely, we need a straightforward generalisation of \cite[Theorem~8.15b]{falconer} to $\mathbb{R}^3$ instead of $\mathbb{R}^2$). 
Let 
\begin{multline*} \Theta = \bigg\{ \theta \in [0,1] : \left\lvert \theta - \frac{b}{n} \right\rvert \leq n_j^{-3/t} \text{ for some $b \in \mathbb{N}$ and $n$ with } 1 \leq n \leq n_j^{1-\kappa},\\ \text{ for infinitely many $j$}\bigg\}. \end{multline*}
Define $\beta = \frac{3}{(1-\kappa)t}-1$. Then $\beta > 1$ since $s \leq 1$. Then $\dim \Theta = \frac{2(1-\kappa)t}{3} = \frac{2}{1+\beta}$ by Theorem~8.16b from \cite[p.~134]{falconer}. Let $\theta \in \Theta$, and let $j$ be such that $\left\lvert \theta - \frac{b}{n} \right\rvert \leq n_j^{-3/t} \text{ for some $b \in \mathbb{N}$ and $n$ with } 1 \leq n \leq n_j^{1-\kappa}$. If $x=(x_1,x_2,x_3) \in A$, then for any $i \in \{1,2,3\}$ we can find $b_{i,j} \in \mathbb{Z}$$ \text{ with } \left\lvert x_i - \frac{b_{i,j}}{n_j} \right\rvert \leq n_j^{-3/t}$. If we let $x' = \left( \frac{b_{1,j}}{n_j}, \frac{b_{2,j}}{n_j}, \frac{b_{3,j}}{n_j} \right)$, and we let $\theta' = \frac{b}{n}$, then 
\[ \pi_{\theta'}(x') = \frac{b_{1,j} n^2 + b_{2,j} bn + b_{3,j} b^2}{n^2n_j}. \] 
There are $\lesssim n_j^{3-2\kappa}$ possibilities for the above fraction, since the numerator is an integer of size $\lesssim$ of the denominator, and $n \le n_j^{1-\kappa}$. But 
\[ \left\lvert \pi_{\theta'}(x') - \pi_{\theta}(x) \right\rvert \lesssim |x-x'| + |\theta-\theta'| \lesssim n_j^{-3/t},  \]
so $\pi_{\theta}(A)$ can be covered by $\lesssim n_j^{3-2\kappa}$ balls of radius $n_j^{-3/t}$. Since we can take $j$ arbitrarily large in the argument, this implies (by the definition of Hausdorff dimension\footnote{This actually yields the same upper bound for the lower Minkowski dimension, which in general is larger than the Hausdorff dimension.}) that 
\[ \dim \pi_{\theta}(A) \leq t-\frac{2\kappa t}{3}. \]
Recalling \eqref{kappadefn}, this yields
\[ \left\{ \theta \in [0, 2\pi) : \dim \pi_{\theta}(A) \leq s-\epsilon \right\} \geq s-\epsilon - \frac{t}{3}, \]
for arbitrarily small $\epsilon$ ensuring that $\kappa$ in \eqref{kappadefn} is rational. Since $\epsilon$ can be taken arbitrarily small, this implies \eqref{exceptionalsetbound} and finishes the proof.\end{proof} \end{sloppypar}

\appendix 

\section{Sine wave incidences} 

In this section we prove some incidence estimates specific to sine waves. The first incidence estimate is equivalent to the one obtained via variable coefficient local smoothing in this special case, but here we give a simpler proof using small cap decoupling.

We then prove an $L^2$-based incidence bound which is stronger in some cases than the one obtained from $L^2$ bounds for Fourier integral operators, as it uses an additional fractal Strichartz estimate as input (explained below). Even though these improve the $L^2$-based incidence bounds in this case, they do not yield any improvement to lower bounds of sine wave Furstenberg sets. 

\subsection{Sine wave incidences via small cap decoupling}

\begin{theorem} \label{theoremincidence2} Let $\alpha \in [0, 2]$, $\beta \in [0,3]$, and $\delta \in (0,1)$. Let $P$ be a $\left( \delta, \beta, K_{\beta,P,\delta}\right)$-set of $\delta$-neighbourhoods of sine waves 
\[ \left\{  \left(\theta, a \cos \theta + b \sin \theta +c\right):  \theta \in [0, 2\pi) \right\}, \]
where $(a,b,c) \in B_3(0,1)$. Let $\mathbb{T}$ be a $\left( \delta, \alpha, K_{\alpha, \mathbb{T}, \delta}\right)$-set of $\delta$-discs in the plane.  If $3\alpha + \beta \leq 7$, then for any $\epsilon >0$
\begin{equation} \label{incidencebound7} I(P, \mathbb{T}) \leq C_{\epsilon} \delta^{-\epsilon} K_{\alpha,\mathbb{T},\delta}^{3/4} K_{\beta,P,\delta}^{1/4}  |P|^{3/4} |\mathbb{T}|^{1/4} \delta^{-3/4}. \end{equation}
If $3\alpha+\beta >  7$, then with $\lambda = \frac{4}{3\alpha+\beta-3} <1$, for any $\epsilon >0$,
\begin{equation} \label{incidencebound8} I(P, \mathbb{T}) \leq  C_{\epsilon} \delta^{-\epsilon}   K_{\alpha,\mathbb{T},\delta}^{3\lambda/4} K_{\beta,P,\delta}^{\lambda/4}  |P|^{1-\frac{\lambda}{4}} |\mathbb{T}|^{1-\frac{3\lambda}{4}} \delta^{-3\lambda/4}. \end{equation}
\end{theorem} 
\begin{proof}  A few times throughout the proof, we will use that \eqref{incidencebound7} implies \eqref{incidencebound8} if $3\alpha+\beta > 7$; the proof of this is similar to the start of the proof of Theorem~\ref{theoremincidencecircle}. 

Let $\Phi$ be the function which sends the $\delta$-neighbourhood of a sine curve $a \cos \theta +b \sin \theta+c$ in $[0, 2\pi] \times \mathbb{R}$ to the $10\delta$-ball centred at $(a,b,c) \in \mathbb{R}^3$. Let $\Psi$ be the function which sends the $\delta$-disc centred at $(\theta, t) \in [0, 2\pi] \times \mathbb{R}$ to the $10\delta$-neighbourhood of the plane $\frac{t}{2} \gamma(\theta) + \gamma(\theta)^{\perp}$, where $\gamma(\theta) =  \left( \cos \theta, \sin \theta, 1 \right)$. It is straightforward to check that
\[ I(P, \mathbb{T} ) \lesssim I(\Phi(P), \Psi(\mathbb{T})). \]
Therefore, we can assume that $P$ is a $\left( \delta, \beta, K_{\beta,P,\delta}\right)$-set of $\delta$-balls in $B_3(0,1)$, and that $\mathbb{T}$ be a $\left( \delta, \alpha, K_{\alpha, \mathbb{T}, \delta}\right)$-set of $\delta$-neighbourhoods of light planes in $\mathbb{R}^3$ (see Subsection~\ref{notation} for the definition of a $\left( \delta, \alpha, K_{\alpha, \mathbb{T}, \delta}\right)$-set), and we will prove the theorem under this assumption. We remark that this kind of duality was introduced in \cite{kaenmakiorponenvenieri} (see also \cite[Eq.~2.12]{pramanikyangzahl}).  

Let $f = \sum_{B \in P} \chi_B$ and let $g = \sum_{T \in \mathbb{T}} \chi_T$ where, for minor technical reasons, $\chi_T$ is a smooth bump function equal to $1$ on $T$ and vanishing outside $2T$. Then
\begin{equation} \label{incidenceintegral2} I(P, \mathbb{T}) \lesssim \delta^{-3} \int f g. \end{equation} Let $\phi$ be a smooth bump function equal to 1 on $B_3\left(0, \delta^{-1-\epsilon^2}\right) \setminus B_3\left(0, \delta^{\epsilon^2-1}\right)$,  vanishing outside $B_3\left(0, 2\delta^{-1-\epsilon^2} \right)$ and vanishing in $B_3\left(0, \delta^{\epsilon^2-1}/2\right)$. Let $\psi$ be a smooth bump function supported equal to $1-\phi$ in $B_3\left(0, \delta^{\epsilon^2-1} \right)$ and extended by zero outside this ball. We define the ``high'' and ``low'' parts $g_h$, $g_l$ of $g$ by
\[ \widehat{g_h} = \widehat{g} \phi, \qquad \widehat{g_l} = \psi \widehat{g}, \]
so that, apart from a very small error term, $g = g_h + g_l$. By \eqref{incidenceintegral2}, 
\begin{equation} \label{highlow2} I(P, \mathbb{T}) \lesssim \delta^{-3}  \int \left\lvert f g_h \right\rvert + \delta^{-3} \int \left\lvert f g_l \right\rvert. \end{equation}
If the high part dominates \eqref{highlow2}, then
\begin{multline*}  I(P, \mathbb{T}) \lesssim \delta^{-3} \left(\int |f|^{4/3}\right)^{3/4} \left( \int \left\lvert  g \ast \widecheck{\phi}\right\rvert^4 \right)^{1/4} \\
\lesssim \delta^{-3/4}K_{\beta,P,\delta}^{1/4} |P|^{3/4} \left( \int \left\lvert  g \ast \widecheck{\phi}\right\rvert^4 \right)^{1/4}. \end{multline*}
Therefore, it suffices (in the high case) to show that 
\begin{equation} \label{highcasereq}  \left( \int \left\lvert  g \ast \widecheck{\phi}\right\rvert^4 \right)^{1/4} \lesssim \delta^{-O(\epsilon^2) } K_{\alpha,\mathbb{T},\delta}^{3/4} |\mathbb{T}|^{1/4}, \end{equation}
where we do not consider \eqref{incidencebound8} separately; since \eqref{incidencebound7} implies \eqref{incidencebound8} when $3\alpha+\beta >7$.  By the small cap decoupling theorem from \cite[Theorem~4]{ganguthmaldague} with $t=1$,
\begin{equation} \label{smallcap} \int \left\lvert  g \ast \widecheck{\phi}\right\rvert^4  \lesssim \delta^{-O(\epsilon^2)}  \delta^{-1} \sum_{\theta \in \Theta} \int \left\lvert \sum_{T \in \mathbb{T}_{\theta} } \chi_T \ast \widecheck{\phi} \right\rvert^4,  \end{equation}
where $\Theta$ is a maximal $\delta$-separated subset of $[0,2\pi)$, and $\mathbb{T}_{\theta}$ consists of those $T \in \mathbb{T}$ with corresponding angle $\theta_T$ within $\delta$ of $\theta$. By Young's convolution inequality applied to \eqref{smallcap}, 
\[ \int \left\lvert  g \ast \widecheck{\phi}\right\rvert^4  \lesssim \delta^{-O(\epsilon^2)}  \delta^{-1} \sum_{\theta \in \Theta} \int \left\lvert \sum_{T \in \mathbb{T}_{\theta} } \chi_T  \right\rvert^4. \]
By the $\left( \delta, \alpha, K_{\alpha, \mathbb{T}, \delta} \right)$ assumption on $\mathbb{T}$, 
\[ \int \left\lvert  g \ast \widecheck{\phi}\right\rvert^4 \lesssim  \delta^{-O(\epsilon^2)}  \delta^{-1}  K_{\alpha,\mathbb{T},\delta}^{3} \sum_{\theta \in \Theta} \sum_{T \in \mathbb{T}_{\theta} } \int \chi_T \lesssim   \delta^{-O(\epsilon^2)}  K_{\alpha,\mathbb{T},\delta}^{3} \left\lvert \mathbb{T} \right\rvert .\]
This implies \eqref{highcasereq}, which finishes the proof if the high part dominates \eqref{highlow2}.

Now suppose that the low part dominates \eqref{highlow2}. This will be similar to the proof of the low case in Theorem~\ref{theoremincidencecircle}. Let $S = \delta^{-\epsilon^2 - \epsilon^4}$. By Hausdorff-Young, 
\[ \left\lVert \chi_T \ast \widecheck{\psi} \right\rVert_{\infty} \lesssim \delta^{-O(\epsilon^4)} S^{-1}. \] 
Hence, by \eqref{highlow2},
\begin{multline*} I(P, \mathbb{T}) \lesssim \delta^{-3} \int \left\lvert f g_{l} \right\rvert \lesssim \delta^{-3} \delta^{-O(\epsilon^4)} \int\left( \sum_{B \in P } \chi_B \right) \left( \sum_{T \in \mathbb{T}_S} S^{-1}\chi_{T_{S}} \right) \\
\lesssim  S^{-1} \delta^{-O(\epsilon^4)} I\left(P_S, \mathbb{T}_S\right), \end{multline*}
apart from a negligible error term, where the balls in $P_S$ are the balls from $P$ scaled by $S$, $\mathbb{T}_S$ consists of $\delta S$ neighbourhoods of the planes from $\mathbb{T}$, and $\chi_{T_S}$ now just means the indicator function of $T_S$. If we assume by induction on $\delta$ that the bound holds at scale $\delta S$ and apply it to the above, we get 
\begin{multline} \label{secondlastsine} I(P, \mathbb{T}) \lesssim \delta^{-O(\epsilon^4) }S^{-1} \times\\
\Bigg\{\begin{aligned} &(\delta S)^{-3/4-\epsilon}  K_{\alpha, \mathbb{T}_S, \delta S}^{3/4} K_{\beta,P_S,\delta S}^{1/4} |P_S|^{3/4} |\mathbb{T}_S|^{1/4} && 3\alpha + \beta \leq 7 \\
&(\delta S)^{-3\lambda/4-\epsilon} K_{\alpha, \mathbb{T}_S, \delta S}^{3\lambda/4}  K_{\beta,P_S,\delta S}^{\lambda/4}|P_S|^{1-\frac{\lambda}{4}} |\mathbb{T}_S|^{1-\frac{3\lambda}{4}} && 3\alpha + \beta >7. \end{aligned} \end{multline}
By the inequalities
\[ K_{\alpha, \mathbb{T}_S, \delta S} \lesssim S^{\alpha} K_{\alpha, \mathbb{T}, \delta}, \qquad K_{\beta,P_S,\delta S} \lesssim S^{\beta} K_{\beta,P,\delta},\]
and since $|\mathbb{T}_S| = |\mathbb{T}|$ and $|P_S| = |P|$, 
\eqref{secondlastsine} becomes
\begin{multline*}  I(P, \mathbb{T}) \lesssim \delta^{\epsilon^3-O(\epsilon^4) } \times\\
\Bigg\{\begin{aligned} & S^{-\frac{7}{4} + \frac{3\alpha}{4} + \frac{\beta}{4}} \delta^{-\frac{3}{4}-\epsilon}  K_{\alpha, \mathbb{T}, \delta}^{3/4} K_{\beta,P,\delta}^{1/4} |P|^{3/4} |\mathbb{T}|^{1/4} && 3\alpha + \beta \leq 7 \\
& S^{-1-\frac{3\lambda}{4} +\frac{3\lambda \alpha}{4} + \frac{\lambda \beta}{4} }\delta^{-
\frac{3\lambda}{4} - \epsilon} K_{\alpha, \mathbb{T}, \delta}^{3\lambda/4}  K_{\beta,P,\delta}^{\lambda/4}|P|^{1-\frac{\lambda}{4}} |\mathbb{T}|^{1-\frac{3\lambda}{4}} && 3\alpha + \beta >7. \end{aligned} \end{multline*}
In either case, the exponent of $S$ is non-positive, so by induction the conclusion of the theorem holds in the low case, and this finishes the proof. 

\end{proof}

\subsection{\texorpdfstring{$L^2$}{} methods for sine waves}

\begin{theorem}[{\cite[Theorem~3]{erdogan}}] \label{fractalstrichartz} Let $\alpha \in [0, 3]$. For any $R \geq 1$ and any $f$ with $\widehat{f}$ supported in $\mathcal{N}_1(R \Gamma)$, 
\[ \|f\|_{L^2(\mu)} \leq C_{\alpha,s} c_{\alpha}(\mu)^{1/2} R^s \|f\|_{L^2(\mathbb{R}^3 ) }, \]
for any Borel measure $\mu$ on $B_3(0,1)$, for any
\[ s > \frac{2-\beta(\alpha)}{2}, \]
where 
\[ \beta(\alpha) = \begin{cases} \alpha & \alpha \leq 1/2 \\
1/2 & 1/2 \leq \alpha \leq 1 \\
\alpha/2 & 1 \leq \alpha \leq 2 \\
\alpha-1 & 2 \leq \alpha \leq 3, \end{cases} \] 
and $c_{\alpha}(\mu) = \sup_{x \in \mathbb{R}^3,r>0} \frac{ \mu(B(x,r) ) }{r^{\alpha} }$.
\end{theorem}

\begin{theorem} \label{theoremincidence}  Let $\alpha \in [0, 2]$, $\beta \in [0,3]$, and $\delta \in (0,1)$. Let $P$ be a $\left( \delta, \beta, K_{\beta,P,\delta}\right)$-set of $\delta$-neighbourhoods of sine waves 
\[ \left\{  \left(\theta, a \cos \theta + b \sin \theta +c\right):  \theta \in [0, 2\pi) \right\}, \]
where $(a,b,c) \in B_3(0,1)$. Let $\mathbb{T}$ be a $\left( \delta, \alpha, K_{\alpha, \mathbb{T}, \delta}\right)$-set of $\delta$-discs in the plane. If $\alpha + \beta \leq 4$, then for any $\epsilon >0$,
\begin{equation} \label{incidencebound} I(P, \mathbb{T}) \leq C_{\epsilon} \delta^{-\epsilon} K_{\alpha,\mathbb{T},\delta}^{1/2} K_{\beta,P,\delta}^{1/2}  |P|^{1/2} |\mathbb{T}|^{1/2} \times \begin{cases} \delta^{-\frac{1}{2}} & 0 \leq \beta \leq 1/2 \\
\delta^{-\frac{1}{4} - \frac{\beta}{2}} & 1/2 \leq \beta \leq 1 \\
\delta^{-\frac{1}{2} - \frac{\beta}{4}}  & 1 \leq \beta \leq 2 \\
\delta^{-1}  & 2 \leq \beta \leq \min\{ 3, 4-\alpha \}. \end{cases}  \end{equation}
If $\alpha + \beta >4$, then for any $\epsilon >0$
\begin{equation} \label{incidencebound17} I(P, \mathbb{T}) \leq C_{\epsilon}\delta^{-\epsilon} K_{\alpha,\mathbb{T},\delta}^{\lambda} K_{\beta,P,\delta}^{\lambda}  |P|^{1-\lambda} |\mathbb{T}|^{1-\lambda} \delta^{-2\lambda}, \end{equation}
where $\lambda = \frac{1}{\alpha+\beta-2} < 1/2$.

For any $1/2 \leq \beta \leq 2$ and any $1 \leq \alpha \leq 2$, \eqref{incidencebound} is sharp in the sense that, for any $\delta>0$, there exist $P$ and $\mathbb{T}$ with both sides of \eqref{incidencebound} equal up to an absolute constant, without the $C_{\epsilon} \delta^{-\epsilon}$ factor, and with $K_{\alpha, \mathbb{T}, \delta} \sim K_{\beta,P,\delta} \sim 1$. Similarly, if $0 \leq \beta \leq 1/2$ and $1+\beta \leq \alpha$, then \eqref{incidencebound} is sharp.

\end{theorem} 
\begin{proof} As mentioned similarly at the start of the proof of Theorem~\ref{theoremincidence2}, we only need to prove \eqref{incidencebound} and \eqref{incidencebound17} when $P$ is a $\left( \delta, \beta, K_{\beta,P,\delta}\right)$-set of $\delta$-balls in $B_3(0,1)$, and $\mathbb{T}$ is a $\left( \delta, \alpha, K_{\alpha, \mathbb{T}, \delta}\right)$-set of $\delta$-neighbourhoods of light planes in $\mathbb{R}^3$, so we will assume we are in this setup. Moreover, by a similar argument, we only need to prove the sharpness statements for this setup, too. Let $f = \sum_{B \in P} \chi_B$ and let $g = \sum_{T \in \mathbb{T}} \chi_T$, where $\chi_T$ is a smooth bump function equal to 1 on $T$ and vanishing outside $2T$. Then 
\[ I(P, \mathbb{T}) \lesssim \delta^{-3} \int f g. \]
We do the same high-low decomposition as in the proof of Theorem~\ref{theoremincidence2}, and we get, analogously to~\eqref{highlow2}, 
\begin{equation} \label{highlow3} I(P, \mathbb{T}) \lesssim \delta^{-3}  \int \left\lvert f g_h \right\rvert + \delta^{-3} \int \left\lvert f g_l \right\rvert. \end{equation}
If the high part dominates \eqref{highlow3}, then by Cauchy-Schwarz,
\begin{equation} \label{twofactors}  I(P, \mathbb{T}) \lesssim \delta^{-3} \left(\delta^3 |P| \right)^{1/2} \delta^{\frac{3-\beta}{2} } \left( \int  \left\lvert g \ast \widecheck{\phi}\right\rvert^2 \, d\mu \right)^{1/2}, \end{equation}
where
\[ \mu = \frac{f}{\delta^{3-\beta}} =\frac{1}{\delta^{3-\beta} }  \sum_{B \in P} \chi_B, \]
which is a $\beta$-dimensional measure ($c_{\beta}(\mu) \lesssim K_{\beta,P,\delta}$). The function $g \ast \widecheck{\phi}$ is the inverse Fourier transform of a function supported in $\mathcal{N}_1(\delta^{-1}\Gamma)$. By Theorem~\ref{fractalstrichartz} (followed by Plancherel),
\begin{equation} \label{fractalstrichartz2} \int \left\lvert g \ast \widecheck{\phi}\right\rvert^2 \, d\mu \lesssim \delta^{-O(\epsilon^2)} \delta^{\max\left\{ \min\left\{ \beta, 1/2 \right\}, \beta/2 , \beta-1 \right\}-2} K_{\beta,P,\delta} \int \left\lvert \widehat{g} \right\rvert^2 \left\lvert \phi \right\rvert^2. \end{equation}
Let $\Theta$ be a maximal $\delta$-separated subset of $[0, 2\pi]$. Let $g_{\theta} = \sum_{T \in \mathbb{T}_{\theta}} \chi_T$, where $\{\mathbb{T}_{\theta}\}_{\theta \in \Theta}$ is a partition of $\mathbb{T}$ such that each $T \in \mathbb{T}_{\theta}$ has normal $\gamma(\theta_T)$ such that $\theta_T$ is within $\delta$ of $\theta$, so that $g = \sum_{\theta \in \Theta} g_{\theta}$. By frequency disjointness and then Plancherel,
\begin{equation} \label{l2decoupling} \int \left\lvert \widehat{g} \right\rvert^2 \left\lvert \phi \right\rvert^2 \lesssim  \sum_{\theta \in \Theta} \| g_{\theta} \|_2^2 \lesssim \delta  K_{\alpha,\mathbb{T},\delta}|\mathbb{T}|.  \end{equation}
Substituting \eqref{l2decoupling} into \eqref{fractalstrichartz2} gives
\[\int \left\lvert g \ast \widecheck{\phi}\right\rvert^2 \, d\mu \lesssim \delta^{-O(\epsilon^2)} \delta^{\max\left\{ \min\left\{ \beta, 1/2 \right\}, \beta/2 , \beta-1 \right\}-1} K_{\alpha,\mathbb{T},\delta} K_{\beta,P,\delta}|\mathbb{T} |. \]
If we use this to bound the second factor of \eqref{twofactors}, we get 
\begin{multline*} I(P, \mathbb{T}) \lesssim \delta^{-O(\epsilon^2)} \delta^{\max\left\{ \min\left\{ -\frac{1}{2}, - \frac{\beta}{2} - \frac{1}{4} \right\}, -\frac{\beta}{4} - \frac{1}{2}, -1 \right\}} |P|^{1/2} |\mathbb{T}|^{1/2} K_{\alpha,\mathbb{T},\delta}^{1/2} K_{\beta,P,\delta}^{1/2} \\ = \delta^{-O(\epsilon^2)} |P|^{1/2} |\mathbb{T}|^{1/2}  K_{\alpha,\mathbb{T},\delta}^{1/2} K_{\beta,P,\delta}^{1/2}\times \begin{cases} \delta^{-\frac{1}{2}} & 0 \leq \beta \leq 1/2 \\
\delta^{-\frac{1}{4} - \frac{\beta}{2}} & 1/2 \leq \beta \leq 1 \\
\delta^{-\frac{1}{2} - \frac{\beta}{4}}  & 1 \leq \beta \leq 2 \\
\delta^{-1}  & 2 \leq \beta \leq 3. \end{cases} \end{multline*}
This proves \eqref{incidencebound} if the high part dominates \eqref{highlow3}. 

If the low part dominates \eqref{highlow3}, let $S = \delta^{-\epsilon^2-\epsilon^4}$. Then
\begin{multline} \label{smoothing} I(P, \mathbb{T}) \lesssim \delta^{-3} \int \left\lvert  f g_{l} \right\rvert \lesssim \delta^{-3-O(\epsilon^4)} \int\left( \sum_{B \in P } \chi_B \right) \left( S^{-1}  \chi_{T_{S}} \right) \\\lesssim  S^{-1} \delta^{-O(\epsilon^4)} I\left(P_S, \mathbb{T}_S\right), \end{multline}
apart from a negligible error term, where the balls in $P_S$ are the balls from $P$ scaled by $S$, $\mathbb{T}_S$ consists of $\delta S$ neighbourhoods of the planes from $\mathbb{T}$, and $\chi_{T_S}$ now just means the indicator function of $T_S$. If $\alpha+\beta \leq 4$, and we assume by induction on $\delta$ that the bound holds at scale $\delta S$ and apply it to the above, we get 
\begin{multline*} I(P, \mathbb{T}) \lesssim \\ \delta^{-O(\epsilon^4)}  S^{-1} K_{\alpha, \mathbb{T}_S, \delta S}^{1/2} K_{\beta,P_S,\delta S}^{1/2}\left( \delta S \right)^{\max\left\{ \min\left\{ -\frac{1}{2}, - \frac{\beta}{2} - \frac{1}{4} \right\}, -\frac{\beta}{4} - \frac{1}{2}, -1 \right\}-\epsilon} |P_S|^{1/2} |\mathbb{T}_S|^{1/2} \\
\lesssim \delta^{-O(\epsilon^4)} S^{-1} S^{(\alpha+\beta)/2} K_{\alpha, \mathbb{T}, \delta}^{1/2} K_{\beta,P,\delta}^{1/2}\left( \delta S \right)^{\max\left\{ \min\left\{ -\frac{1}{2}, - \frac{\beta}{2} - \frac{1}{4} \right\}, -\frac{\beta}{4} - \frac{1}{2}, -1 \right\}-\epsilon} \\
\times |P|^{1/2} |\mathbb{T}|^{1/2}. \end{multline*}
The non-infinitesimal exponent of $S$ is 
\begin{multline*} \max\left\{ \min\left\{ \frac{\alpha+\beta- 3 }{2}, \frac{\alpha}{2} - \frac{5}{4} \right\}, \frac{\alpha}{2} + \frac{\beta}{4} - \frac{3}{2} , \frac{ \alpha+\beta}{2} - 2 \right\} \\
 = \begin{cases} \frac{\alpha+\beta-3}{2} & 0 \leq \beta \leq 1/2 \\
 \frac{\alpha}{2} - \frac{5}{4} & 1/2 \leq \beta \leq 1 \\
 \frac{\alpha}{2} + \frac{\beta}{4} - \frac{3}{2} & 1 \leq \beta \leq 2 \\
 \frac{\alpha+\beta}{2} -2  & 2 \leq \beta \leq 3. \end{cases} \end{multline*}
It is straightforward to check that for $\beta \leq 2$, the above exponent is non-positive (recall that $\alpha \leq 2$). In the fourth case, the exponent is also non-positive because we have assumed that $\alpha + \beta \leq 4$. This proves the theorem if $\alpha+\beta\leq 4$ and if the low part dominates \eqref{highlow3}.  If $\alpha+\beta > 4$, then \eqref{smoothing} becomes
\begin{multline*} I(P, \mathbb{T}) \lesssim \delta^{-O(\epsilon^4)} S^{-1} K_{\alpha, \mathbb{T}_S, \delta S}^{\lambda} K_{\beta,P_S,\delta S}^{\lambda}\left( \delta S \right)^{-2\lambda-\epsilon} |P_S|^{1-\lambda} |\mathbb{T}_S|^{1-\lambda} \\
\lesssim  \delta^{-O(\epsilon^4)}S^{\lambda(\alpha+\beta-2)-1} K_{\alpha, \mathbb{T}, \delta}^{\lambda} K_{\beta,P,\delta}^{\lambda} \delta^{-2\lambda-\epsilon} |P|^{1/2} |\mathbb{T}|^{1/2}, \end{multline*}
The non-infinitesimal exponent of $S$ is zero by the definition of $\lambda$, so this finishes the proof in the case $\alpha + \beta > 4$.

It remains to prove the sharpness statements. Fix a $1 \times \delta^{1/2} \times \delta$ plank $R$ tangent to the light cone, passing through the origin.

Assume first that $1/2 \leq \beta \leq 1$ and $1 \leq \alpha \leq 2$. Let $P$ be a set of $\delta$-balls in $R$ with $K_{\beta,P,\delta} \lesssim 1$ and $|P| \sim \delta^{-\beta}$. Let $\mathbb{T}$ be the set of $\delta$-neighbourhoods of light planes which intersect $R$ and such that the normal to $T \in \mathbb{T}$ makes an angle $\leq \delta^{1/2}$ with the normal to $R$. In parameter space, $\mathbb{T}$ corresponds to a set of $\delta$-discs covering a $\delta^{1/2} \times \delta$-interval, which means that $|\mathbb{T}| \sim \delta^{-1/2}$ and $K_{\alpha,\mathbb{T},\delta} \lesssim 1$ (this is where we use the assumption that $\alpha \geq 1$).  Each $\delta$-neighbourhood of a light plane in $\mathbb{T}$ essentially covers the entire plank $R$, and therefore intersects every ball in $P$. Hence, the left-hand side of \eqref{incidencebound} is
\[ I(P, \mathbb{T}) \sim \delta^{-\beta - \frac{1}{2}}, \]
which is comparable to the right-hand side of \eqref{incidencebound}. 

Now suppose that $1 \leq \beta \leq 2$. Let $P$ be a set of $\delta$-balls in $R$ with $|P| \sim \delta^{-\frac{1}{2}- \frac{\beta}{2}}$ and $K_{\beta,P,\delta} \sim 1$. Such a set can be constructed by taking a line of $\delta^{-1}$ many $\delta$-balls along a light ray in $R$, translating it by a distance $\delta^{\beta/2}$ along the medium direction of the plank a total of $\delta^{(1-\beta)/2}$ many times, and then taking a union of these translates.  We again let $\mathbb{T}$ be the set of $\delta$-neighbourhoods of light planes which intersect $R$ and such that the normal to $T \in \mathbb{T}$ makes an angle $\leq \delta^{1/2}$ with the normal to $R$. The left-hand side of \eqref{incidencebound} is 
\[ I(P, \mathbb{T}) \sim \delta^{-1-\frac{\beta}{2} }, \]
which matches the right-hand side of \eqref{incidencebound}. 

Finally, the sharpness in the case $0 \leq \beta \leq 1/2$ and $1+\beta \leq \alpha$ follows from sharpness of the case $0 \leq \beta \leq 1/2$ and $1+\beta \leq \alpha$ in Corollary~\ref{incidencecorollary} below; since we can obtain this case in Corollary~\ref{incidencecorollary} (at least if $K_{\alpha, \mathbb{T}, \delta} \sim K_{\beta,P,\delta} \sim 1$)  by reducing to the case where every $T \in \mathbb{T}$ intersects some $B \in P$, applying \eqref{incidencebound}, and then using the bound $\left\lvert \mathbb{T} \right\rvert \lesssim  \left\lvert P \right\rvert \delta^{-1}$, which follows from pigeonholing and using that every $B \in P$ intersects $\lesssim \delta^{-1}$ many $T \in \mathbb{T}$. \end{proof}

\begin{corollary} \label{incidencecorollary}  Let $\alpha \in [0, 2]$, $\beta \in [0,3]$, and $\delta \in (0,1)$. Let $P$ be a $\left( \delta, \beta, K_{\beta,P,\delta}\right)$-set of $\delta$-neighbourhoods of sine waves 
\[ \left\{  \left(\theta, a \cos \theta + b \sin \theta +c\right):  \theta \in [0, 2\pi) \right\}, \]
where $(a,b,c) \in B_3(0,1)$.  Let $\mathbb{T}$ be a $\left( \delta, \alpha, K_{\alpha, \mathbb{T}, \delta}\right)$-set of $\delta$-discs in the plane. Then for any $\epsilon >0$, 
\begin{equation} \label{incidenceboundcorollary} I(P, \mathbb{T}) \leq C_{\epsilon} \delta^{-\epsilon} K_{\alpha, \mathbb{T}, \delta } K_{\beta, P } \begin{cases} \delta^{-\frac{1}{2} - \frac{\beta}{2} - \frac{\alpha}{2}} & 0 \leq \beta \leq 1/2 \text{ and } \alpha < 1 + \beta \\
\delta^{-\frac{1}{4} - \beta - \frac{\alpha}{2}} & 1/2 \leq \beta \leq 1  \text{ and } \alpha \leq 2-\beta \\
\delta^{-\frac{3}{4} - \frac{3\beta}{4} - \frac{\alpha}{4}} & 1/2 \leq \beta \leq 1  \text{ and } 2- \beta < \alpha < 1+\beta \\
\delta^{-1-\beta} & 0 \leq \beta \leq 1 \text{ and } \alpha \geq 1+\beta \\
\delta^{-\frac{1}{2} - \frac{3\beta}{4} - \frac{\alpha}{2}}  & 1 \leq \beta \leq 2 \text{ and } \alpha \leq 1\\
\delta^{-\frac{3}{4} - \frac{3\beta}{4} - \frac{\alpha}{4} }  & 1 \leq \beta \leq 2\text{ and } 1 < \alpha \leq \frac{7-\beta}{3} \\
\delta^{-\frac{3}{4} - \frac{3\beta}{4} - \frac{\alpha}{4} }  & 2 \leq \beta \leq 5/2 \text{ and } \beta-1 < \alpha \leq \frac{7-\beta}{3} \\
\delta^{-\frac{\alpha}{2} - \frac{\beta}{2} -1 }  & 2 \leq \beta \leq 5/2 \text{ and } \alpha \leq \beta-1 \\
\delta^{-\frac{\alpha}{2} - \frac{\beta}{2} -1 } & 5/2 \leq \beta \leq \min\{ 3, 4-\alpha \} \\
\delta^{1- \alpha-\beta} & \alpha+\beta > 4 \text{ or } 3\alpha+\beta > 7. \end{cases}  \end{equation}
If $0 \leq \beta \leq 1$ and $1+\beta \leq \alpha$, or if $1 \leq \alpha \leq 3/2 $ and $1+\alpha \leq \beta \leq 4-\alpha$, then \eqref{incidenceboundcorollary} is sharp in the sense that there exist $P$ and $\mathbb{T}$ with both sides equal up to a constant. 
\end{corollary}
\begin{proof} If we apply the upper bounds $|P| \lesssim K_{\beta,P,\delta} \delta^{-\beta}$ and $|\mathbb{T}| \lesssim K_{\alpha, \mathbb{T}, \delta} \delta^{-\alpha}$, then Theorem~\ref{theoremincidence} implies the 1st, 2nd, 5th, 8th and 9th lines of \eqref{incidenceboundcorollary}, and the case $\alpha+\beta>4$ of the 10th line. Theorem~\ref{theoremincidence2} implies the 3rd, 6th and 7th lines, and the $3\alpha+\beta>7$ case of the 10th line. For the remaining 4th line, in the case where $\alpha \geq 1+\beta$, we can obtain the bound $I(P, \mathbb{T}) \lesssim K_{\alpha, \mathbb{T}, \delta} K_{\beta,P,\delta} \delta^{-\beta-1}$ directly by a counting argument (i.e.~without using Theorem~\ref{theoremincidence}); since for each $B \in P$, there are $\lesssim K_{\alpha, \mathbb{T}, \delta} \delta^{-1}$ elements of $\mathbb{T}$ passing through $B$. 

It remains to prove the statements about sharpness. 

When $0 \leq \beta \leq 1$ and $1+\beta \leq \alpha$, then \eqref{incidenceboundcorollary} implies the incidence bound $I(P, \mathbb{T}) \leq C_{\epsilon} \delta^{-\epsilon} \delta^{- 1-\beta}$ for tubes and discs in the plane, which matches \cite[Theorem~1.4]{furen} (in this case), which is sharp in this range as stated in \cite{furen}. 

Suppose that $1 \leq \alpha \leq 3/2$ and $1+\alpha \leq \beta \leq 4-\alpha$. Let $\widetilde{\alpha} = \alpha$ and let $\widetilde{\beta} = \beta - 1$. Let $\widetilde{P}$ and $\widetilde{\mathbb{T}}$ be the sets from the example in \cite[Construction 1]{furen}, but with the $\alpha$ and $\beta$ in \cite{furen} replaced by $\widetilde{\alpha}$ and $\widetilde{\beta}$. Let $\mathbb{T}$ be the set of $\delta$-neighbourhoods $T$ of light planes in $\mathbb{R}^3$ such that the intersection of $T$ with $\mathbb{R}^2 = \mathbb{R}^2 \times \{0\}$ is a $\delta$-tube from $\widetilde{\mathbb{T}}$. In \cite[Subsection~2.1]{furen}, the points in $\widetilde{P}$ are partitioned into rectangles $R_i$ of dimensions $\delta \times \delta^{1-\gamma}$ (and they completely fill such rectangles), where $\gamma = \frac{1-\widetilde{\alpha}+\widetilde{\beta}}{2} \in [1/2,1 ]$, with each $R_i$ contained in the intersection of a ``bundle'' of tubes $\widetilde{\mathbb{T}_i}$ from $\widetilde{\mathbb{T}}$, which all pass through the centre of $R_i$ and have angles within a distance $\delta^{\gamma}$ of each other. Since $\gamma \geq 1/2$ (this is where we use that $\beta \geq 1+\alpha$), the corresponding family of $\delta$-neighbourhoods of light planes from $\mathbb{T}$ must contain a box of dimensions $1 \times \delta^{1-\gamma} \times \delta$. This means that we can construct a $\beta = \left(\widetilde{\beta} +1\right)$-dimensional family of balls $P$ in $\mathbb{R}^3$ by repeating each ball in $\widetilde{P}$ $\delta^{-1}$ many times along a light ray inside the $1 \times \delta^{1-\gamma} \times \delta$ box in which it is contained. This increases the total number of incidences by $\delta^{-1}$ over the planar example (at least those incidences counted in each bundle in \cite[p.~6]{furen}), and the dimension of $P$ is $\widetilde{\beta}+1 = \beta$. The number of incidences in the planar example is $\delta^{ - \frac{\alpha}{2} - \frac{\beta}{2}}$, so this gives an example with $I(P, \mathbb{T}) \gtrsim \delta^{ - \frac{\alpha}{2} - \frac{\beta}{2} - 1}$, which matches the right-hand side of \eqref{incidenceboundcorollary} in this case. \end{proof}

\end{document}